\documentclass[11pt]{article}
\usepackage{amsmath, amsfonts, amsthm, amssymb, hyperref}
\usepackage{mathabx}
\usepackage{authblk, color, bm, graphicx}
\usepackage[margin = 1.5cm, font = small, labelfont = bf, labelsep = period]{caption}
\usepackage{xspace}
\usepackage[utf8]{inputenc}
\usepackage[round]{natbib}
\usepackage{rotating, multirow}
\usepackage{xcolor}
\usepackage{fullpage}[2cm]
\usepackage{cleveref, booktabs, graphicx, amsmath, subcaption, bm, tikz, amssymb, breqn, environ, algorithm, algpseudocode, titlesec, amsthm}
\allowdisplaybreaks

\usepackage{algorithm}
\usepackage{algpseudocode}
\usepackage{tikz}
\usepackage{graphicx, booktabs, graphicx, amsmath, subcaption, cancel, bm, hyperref, cleveref, amsfonts, amssymb} 
\usepackage{booktabs}
\usepackage{pifont}

\newcommand{\cmark}{\ding{51}}
\newcommand{\xmark}{\ding{55}}

\newcommand{\mb}[1]{\mathbb{#1}}
\newcommand{\mc}[1]{\mathcal{#1}}
\newcommand{\bzz}{\bm{\zeta}_0}
\newcommand{\bz}{\bm{\zeta}}
\newcommand{\mcge}{\mathcal{G}(\mathcal{E})}
\newcommand{\Yy}{(\bm{Y}, \bm{y})}
\newcommand{\Yyi}{(\bm{Y}^i, \bm{y}^i)}
\newcommand{\Yyj}{(\bm{Y}^j, \bm{y}^j)}
\newcommand{\hatYy}{(\hat{\bm{Y}}, \hat{\bm{y}})}
\newcommand{\uga}{\lfloor\Gamma\rfloor}
\usepackage{changepage}
\newenvironment{tightquote}
{\begin{adjustwidth}{1.4em}{1.4em}\itshape}
	{\end{adjustwidth}}

\newtheorem{definition}{Definition}
\newtheorem{theorem}{Theorem}
\newtheorem{proposition}{Proposition}
\newtheorem{lemma}{Lemma}

\newtheorem{corollary}{Corollary}

\title{When Is Two-Stage Sample Robust Optimization Asymptotically Optimal? A Simple Perspective}

\date{}

\author{Ling Dai}
\author{Chin Pang Ho}

\affil{\small \textit{Department of Data Science, City University of Hong Kong} \protect\\ \texttt{ling.dai@cityu.edu.hk}, \texttt{clint.ho@cityu.edu.hk}}

\makeatletter
\newcommand{\customtaglabel}[2]{%
	\tag{#1}
	\edef\@currentlabel{#1}
	\label{#2}
}
\makeatother

\usepackage{color-edits}
\addauthor{dl}{magenta}
\addauthor{cp}{blue}
\addauthor{zc}{violet}
\addauthor{rv}{brown}

\begin{document}
	
	\maketitle
	
	\begin{abstract}
		Two-stage sample robust optimization with linear decision rules is a standard data-driven approach to two-stage stochastic linear programs with unknown distributions. \cite{bertsimas2022two} established that this approach is asymptotically optimal under several conditions. Most conditions are mild, but one imposes a technical requirement on the recourse system. Whether this condition is implied by the others was left as an open question. We resolve this open question through a geometric feature of the support set: simpleness, the property that every vertex is incident to exactly as many edges as the dimension of the support. For simple support sets, the technical condition is automatic, and asymptotic optimality holds under the remaining standard assumptions. For every non-simple support, we construct a recourse system that is feasible everywhere in the support but violates the technical condition, and this construction resolves the open question. We further characterize the simpleness of one-sided and two-sided budgeted uncertainty sets and identify when budget perturbations can restore simpleness. We also develop a polynomial-time algorithm that certifies the technical condition at a given vertex using its incident edge directions. Numerical experiments on newsvendor and network lot-sizing problems illustrate the characterization and the effectiveness of the verification algorithm.
		
	\end{abstract}

	\section{Introduction} 
	Two-stage stochastic linear programs arise across operations research \citep{dantzig1955linear, beale1955minimizing}. There is a wide range of applications of this prominent framework \citep{birge1997introduction}, including production and capacity planning \citep{nishizaki2023two, yu2025value} and network design \citep{sen2016mitigating, koutsokosta2024stochastic}. In all of these problems, a decision maker first commits to an initial decision before uncertainty is realized and then takes a recourse action after observing the realization. The objective of the decision maker is to choose first-stage decisions that minimize the sum of an immediate cost and the expected second-stage cost.
	
	The central difficulty in solving two-stage stochastic linear programs is that the true distribution of the underlying randomness is rarely known. Instead, one typically has access to a finite sample of historical observations together with information on the support. To address this challenge, researchers have developed a variety of methods that solve the programs without requiring the exact distribution \citep{berkelaar2002primal, sen2016mitigating, gangammanavar2021stochastic, pasupathy2021adaptive, bertsimas2023data}. However, most existing methods cannot achieve both asymptotic optimality and sample efficiency simultaneously. Specifically, \cite{berkelaar2002primal} propose a primal-dual decomposition-based interior point method, but it offers no convergence guarantee. On the other hand, the stochastic decomposition method \citep{sen2016mitigating, gangammanavar2021stochastic} and the sample average approximation method \citep{pasupathy2021adaptive} both enjoy convergence guarantees, but they require a large number of samples. This requirement severely limits their applicability when only limited historical data is available \citep{bertsimas2022two}.
	
	To effectively approximate a two-stage stochastic linear program from limited data, a substantial line of work has turned to distributionally robust optimization \citep{hanasusanto2018conic, xie2020tractable, bertsimas2022two, bertsimas2023data}. In this approach, the decision maker seeks the first-stage decision that performs best in expectation under an adversarially chosen probability distribution from a prescribed ambiguity set. Although two-stage distributionally robust optimization with a Wasserstein ambiguity set is known to be NP-hard \citep{feige2007robust, hanasusanto2018conic}, considerable effort has been devoted to designing tractable approximations. Two recent works are particularly influential in this direction. First, \citet{bertsimas2023data} establish an equivalence between two-stage distributionally robust linear optimization with the type-$\infty$ Wasserstein ambiguity set and two-stage sample robust optimization (SRO). Building on this, \citet{bertsimas2022two} propose a multi-policy approximation for SRO and prove that, under suitable assumptions, this approximation with linear decision rules is asymptotically optimal. Together, these results establish the multi-policy SRO with linear decision rules as a powerful and practical framework for two-stage stochastic linear programs, combining both asymptotic optimality and sample efficiency.
	
	However, the convergence result of two-stage SRO relies on certain conditions on the problem. For self-containedness, we include it as \Cref{thm:asympt_opt} in the next section. Most of these conditions are mild and standard in stochastic optimization, and widely used in applications. However, one technical condition is significantly different. More precisely, it asks whether there exists a single first-stage decision under which affine recourse policies are locally feasible everywhere in the support. This condition is geometrically local in nature, and it has no obvious counterpart in the classical literature on two-stage stochastic programs. In fact, \citet{bertsimas2022two} themselves conjectured that this condition might be redundant. Despite their efforts, they could not find an instance satisfying their other standard assumptions but violating this one. They further remarked that ``The question of whether these conditions are equivalent is open for future research.'' This question has remained open since their original preprint appeared in $2020$. Resolving it matters in two ways. On the one hand, if the condition is redundant, then it can be dropped, and the asymptotic optimality of multi-policy SRO with linear decision rules holds under genuinely mild assumptions. On the other hand, if the condition is not redundant, then decision makers face a more serious issue. Without further guidance, they cannot tell whether two-stage SRO yields a trustworthy approximation for the original problem. Therefore, decision makers need a way to recognize when the condition may fail, and ideally a systematic procedure to verify it for given problem instances.
	
	We resolve this question by examining the geometry of the uncertainty support. In detail, we observe that the other mild conditions are global in nature, while the technical condition is local. This contrast motivates two feasibility notions for recourse systems on the support: global feasibility and locally linear feasibility (see Definitions~\ref{def:gf} and \ref{def:llf}). Then, the technical condition is equivalent to requiring a single first-stage decision for which the resulting recourse system is locally linear feasible at \textit{every} vertex of the support. These notions reframe the above open question as a sharp polyhedral question: when does global feasibility imply locally linear feasibility?
	
	The answer depends on a single geometric feature of the support set. A polyhedron is called simple when every vertex meets exactly as many edges as the dimension of the polyhedron, which coincides with vertex non-degeneracy in the sense of linear programming. For example, boxes and simplices are simple, while octahedra are not. Our main result establishes a characterization along this feature. If the support set is simple, then global feasibility implies locally linear feasibility for every recourse system on the support. Conversely, if any vertex of the support set is non-simple, then there exists a recourse system on the support for which global feasibility holds but locally linear feasibility fails at that vertex. In the language of \citet{bertsimas2022two}, the technical condition is automatic from the standard mild conditions if the support set is simple, but it is not automatic on non-simple supports. Therefore, the open question in \cite{bertsimas2022two} is resolved.

	The characterization of simpleness raises a practical question: when are commonly used uncertainty supports simple? We address this question for one-sided and two-sided budgeted uncertainty sets \citep{bertsimas2004price}. By explicitly describing their vertices and incident edge directions, we derive necessary and sufficient conditions for simpleness in terms of the uncertainty dimension and budget. These results reveal an important distinction between the two types of sets. For one-sided budgeted uncertainty sets, any non-integer budget yields a simple support, so an arbitrarily small budget perturbation can always restore simpleness. This suggests a simple and effective modeling adjustment: choosing a nearby non-integer budget guarantees asymptotic optimality under the remaining standard assumptions and leads to trustworthy decisions. For two-sided budgeted uncertainty sets, however, a non-integer budget alone does not ensure simpleness, and small budget perturbations cannot always restore it. When the support is non-simple, the technical condition may still hold for some specific recourse systems, but the geometry of the support alone no longer guarantees it. Therefore, the decision maker needs a method to verify the technical condition on the given recourse system together with the local geometry of the support.
	
	To this end, we develop an algorithm that checks locally linear feasibility at a given vertex. Given the recourse system and the directions of the edges incident to that vertex, the algorithm decides in polynomial time whether locally linear feasibility holds at that vertex. The algorithm is based on a characterization of locally linear feasibility through an index set identifying the relevant active constraints. By iteratively constructing surrogates for this index set, the algorithm verifies locally linear feasibility in finitely many iterations. Moreover, in each iteration, the decision maker solves a collection of linear programs, with both their number and their sizes bounded by polynomials in the dimensions of the recourse system and the number of incident edges. Therefore, the total running time is polynomial in the dimensions of the recourse system and the number of incident edges. 
	
	Taken together, the geometric characterization, the results on budgeted uncertainty sets, and the verification algorithm provide practical guidance for decision makers using the multi-policy approximation with linear decision rules. For simple supports, including boxes and simplices, asymptotic optimality follows under the remaining standard assumptions without further verification of the technical condition. For one-sided budgeted uncertainty sets, decision makers can obtain the same guarantee by choosing a non-integer budget arbitrarily close to the original value. For two-sided budgeted uncertainty sets, our simpleness conditions identify when such an adjustment is possible. When the support remains non-simple, the verification algorithm allows decision makers to check locally linear feasibility for the given recourse system at individual vertices. In \Cref{sec:casestudy}, we illustrate these implications through newsvendor and network lot-sizing problems, examining both the verification procedure and the convergence of the multi-policy approximation.
	
	The contributions of this paper are summarized as follows.
	\begin{enumerate}
		\item We resolve the open question of \citet{bertsimas2022two} by reformulating it as a polyhedral question through the notions of global feasibility and locally linear feasibility. We prove that global feasibility implies locally linear feasibility uniformly over all recourse systems if and only if the support is simple. We also provide an explicit three-dimensional counterexample showing that the technical condition need not follow from the remaining assumptions.
		
		\item We derive necessary and sufficient conditions for the simpleness of one-sided and two-sided budgeted uncertainty sets by describing their vertices and incident edge directions. Specifically, for one-sided budgeted uncertainty, an arbitrarily small budget perturbation can always restore simpleness, whereas this is not always possible for two-sided budgeted uncertainty sets. These results provide practical guidance for decision makers on adjusting the uncertainty budget to obtain a simple support and ensuring asymptotic optimality under mild assumptions.
		
		\item We develop an algorithm that verifies locally linear feasibility at a given vertex using the recourse system and the incident edge directions. Based on a characterization of locally linear feasibility, the algorithm proceeds by solving a sequence of linear programs. We prove that this procedure terminates after finitely many iterations and correctly determines whether locally linear feasibility at the given vertex holds. We further establish that its total running time is polynomial in the dimensions of the recourse system and the number of incident edges.
		
		\item We conduct numerical experiments on newsvendor and network lot-sizing problems with two-sided and one-sided budgeted uncertainty sets, respectively. The experiments evaluate the computational performance of the verification algorithm and examine the convergence of the multi-policy approximation with linear decision rules. They also illustrate how budget perturbations that restore simpleness can recover the technical condition and enable convergence.
	\end{enumerate}
	
	The remainder of the paper is organized as follows. \Cref{sec:setup} introduces the two-stage sample robust optimization and states the asymptotic optimality theorem of \citet{bertsimas2022two}. It then reformulates the open question through the two feasibility notions of global feasibility and locally linear feasibility. \Cref{sec:geometric_character} establishes the geometric characterization through simple polyhedral supports and presents an explicit three-dimensional counterexample. \Cref{sec:budgeted_uncertainty_set} characterizes the simpleness of one-sided and two-sided budgeted uncertainty sets and discusses the implications of budget perturbations. \Cref{sec:algorithm} develops the verification algorithm and establishes its correctness and polynomial-time complexity at a given vertex. \Cref{sec:casestudy} presents numerical experiments on newsvendor and network lot-sizing problems. \Cref{sec:discussion} concludes and discusses directions for future research. Proofs and additional technical details are provided in the Appendix.
	
	\textbf{Notation.} We denote vectors and matrices by boldface lowercase and uppercase letters, respectively. Vector inequalities are interpreted componentwise. The all-ones vector and the $i$-th standard basis vector are denoted by $\bm e$ and $\bm e_i$, respectively. We use $\bm 0$, $\bm O$, and $\bm I$ to denote the zero vector, zero matrix, and identity matrix, respectively, with dimensions determined by the context or indicated by subscripts; in particular, $\bm O_k$ and $\bm I_k$ are $k\times k$ matrices. For any $\bm w\in\mathbb R^k$, $\operatorname{Diag}(\bm w)$ denotes the diagonal matrix with $\bm w$ on its diagonal. The symbol $\circ$ denotes componentwise multiplication. For any $\Gamma\geq0$, $\lfloor\Gamma\rfloor$ denotes the largest integer not exceeding $\Gamma$. For any positive integer $S$, we write $[S]=\{1,\ldots,S\}$ and define the probability simplex $\Delta_S=\{\bm\alpha\in\mathbb R_+^S:\bm e^\top\bm\alpha=1\}$, where $\mathbb R_+^S$ denotes the nonnegative orthant. For any finite set $\mathcal E$, $|\mathcal E|$ denotes its cardinality. We use the convention $[0]=\emptyset$ and interpret a sum over an empty index set as zero. For a set $A$, $\operatorname{conv}(A)$, $\operatorname{cone}(A)$, $\operatorname{span}(A)$, and $\operatorname{aff}(A)$ denote its convex hull, conic hull, linear span, and affine hull, respectively. The dimension $\dim(\Xi)$ of a polyhedron $\Xi$ refers to its affine dimension; that is, the dimension of the linear space parallel to ${\rm aff}(\Xi)$. Given a norm $\|\cdot\|$, we write $B(\bm\zeta,\epsilon)=\{\bm\xi:\|\bm\xi-\bm\zeta\|<\epsilon\}$ for the open ball centered at $\bm\zeta$ with radius $\epsilon>0$. For a cone $\mathcal C\subseteq\mathbb R^d$, its dual cone is $\mathcal C^\star=\{\bm u\in\mathbb R^d:\bm u^\top\bm v\geq0\ \forall\bm v\in\mathcal C\}$, and $\bm a\preceq_{\mathcal C}\bm b$ means $\bm b-\bm a\in\mathcal C$. Finally, $\mathcal P(\Xi)$ denotes the set of all probability distributions supported on $\Xi$, and $\operatorname{Unif}([a,b]^k)$ denotes the uniform distribution on the box $[a,b]^k$.
	
	\section{Preliminaries and the Open Question}\label{sec:setup}
	
	Two-stage sample robust optimization is a standard data-driven approach to solve the two-stage stochastic linear program \citep{bertsimas2022two}. Specifically, we consider a two-stage problem 
	\begin{equation}\label{eqn:two_stage}
		v^* \triangleq \min_{\bm{x} \in \mathbb{R}^n} \left\{V^*(\bm{x}) \triangleq \bm{c}^{\top} \bm{x}+\mathbb{E} \left[Q(\bm{x}, \tilde{\bm \xi})\right]\right\},
	\end{equation}
	where the recourse function $Q$ is defined as
	\[
	Q(\bm{x},\bm{\xi}) \;\triangleq\; \min_{\bm{y}\in\mathbb{R}^{r}}\Bigl\{\bm{q}^{\top}\bm{y} \;:\; \bm{T}\bm{x}+\bm{W}\bm{y}\ge \bm{h}(\bm{\xi})\Bigr\}.
	\]
	Here, the decision variable $\bm{x}\in\mathbb{R}^{n}$ in~\eqref{eqn:two_stage} is the first-stage decision and $\bm{c}^{\top}\bm{x}$ is the first-stage cost. Random vector $\tilde{\bm{\xi}}\in\Xi\subset \mathbb{R}^{d}$ is governed by a distribution $\mb P$ on the polyhedron support $\Xi$, and $Q(\bm{x},\tilde{\bm{\xi}})$ is the random second-stage cost. The right-hand side $\bm{h}(\bm{\xi})=-\bm{H}\bm{\xi}+\bm{h}_{0}\in\mathbb{R}^{m}$ depends affinely on $\bm{\xi}$. Matrices $\bm{T}\in\mathbb{R}^{m\times n}$, $\bm{W}\in\mathbb{R}^{m\times r}$, and $\bm{H}\in\mathbb{R}^{m\times d}$, and vectors $\bm{c}\in\mathbb{R}^{n}$, $\bm{q}\in\mathbb{R}^{r}$ and $\bm{h}_{0}\in\mathbb{R}^{m}$ are given.
	
	In practice, the distribution $\mathbb{P}$ of $\tilde{\bm{\xi}}$ is unknown. Instead, the decision maker only has access to the polyhedron support $\Xi$ and a finite sample $\bm{\xi}^{1},\dots,\bm{\xi}^{N}$. To address this challenge, it is natural to consider a distributionally robust formulation of the two-stage problem~\eqref{eqn:two_stage}. In particular, \cite{bertsimas2023data} show that the resulting distributionally robust two-stage problem with type-$\infty$ Wasserstein ambiguity set is equivalent to the following two-stage \emph{sample robust} problem
	\begin{equation}\label{eqn:sample_robust}
		\begin{array}{rcl}
			\hat{v}_N^{\mathrm{SRO}} = & \displaystyle\min_{\bm x\in \mathbb R^n, \bm y(\cdot)\in \mathcal R^{d, r}} & \bm{c}^{\top} \bm{x} + \displaystyle\frac{1}{N} \sum_{i=1}^N \sup _{\boldsymbol{\zeta} \in \mathcal{U}_N^i} \bm q^\top \bm y(\bm \zeta)\\
			&{\rm s.t.}& \bm{T}\bm{x}+\bm{W}\bm{y}(\bm{\zeta})\ge\bm{h}(\bm{\zeta}) \quad \forall\bm{\zeta}\in\displaystyle\bigcup_{i=1}^{N}\mathcal{U}_{N}^{i},
		\end{array}
	\end{equation}
	where $\mathcal R^{d, r}$ denotes the set of all functions from $\mathbb{R}^{d}$ to $\mathbb{R}^{r}$, and $\mathcal{U}_{N}^{i}\triangleq\Xi\cap\{\bm{\xi}:\|\bm{\xi}-\bm{\xi}^{i}\|\le\epsilon_{N}\}$ is the uncertainty set centered at sample $\bm{\xi}^{i}$. The norm defining $\mc U^i_N$ is the one used in the type-$\infty$ Wasserstein distance, and $\epsilon_{N}$ is the radius of the Wasserstein ambiguity set. 
	
	Program~\eqref{eqn:sample_robust} remains intractable because of the infinite-dimensional policy space $\bm y(\cdot)\in \mathcal R^{d, r}$ and the semi-infinite constraints. To address this, \citet{bertsimas2022two} propose a \emph{multi-policy approximation}, in which a separate policy $\bm{y}^{i}(\cdot)$ is used on each sample neighborhood $\mathcal{U}_{N}^{i}$. Within the framework of this multi-policy approximation, separable linear decision rules can be further imposed on each policy $\bm{y}^{i}(\cdot)$, leading to the following optimization problem: 
	\begin{equation}\label{eqn:multi_policy}
		\begin{array}{rcl}
			\hat{v}_N^{\mathrm{MP}} = & \displaystyle\min_{\bm x\in \mathbb R^n, \bm y^1(\cdot), \ldots, \bm y^N(\cdot)\in \mathcal L} & \bm{c}^{\top} \bm{x} + \displaystyle\frac{1}{N} \sum_{i=1}^N \sup _{\boldsymbol{\zeta} \in \mathcal{U}_N^i} \bm q^\top \bm y^i(\bm \zeta)\\
			&{\rm s.t.}& \bm T\bm x+\bm W\bm y^i(\bm \zeta) \geq \bm h(\bm \zeta) \quad \forall \bm \zeta \in \mathcal{U}_N^i, \; i\in [N],
		\end{array}
	\end{equation}
	where $\mathcal{L}=\{\bm{y}(\cdot)\in\mathcal{R}^{d,r}:\exists\bm{y}_{0}\in\mathbb{R}^{r},\,\bm{Y}\in\mathbb{R}^{r\times d}\text{ s.t. }\bm{y}(\bm{\xi})=\bm{y}_{0}+\bm{Y}\bm{\xi}\}$ is the set of all affine functions from $\mathbb R^d$ to $\mathbb R^r$. Since each $\mathcal{U}_{N}^{i}$ is the intersection of a polyhedron and a norm ball, program~\eqref{eqn:multi_policy} reduces to a linear program when the norm is $\ell_{1}$ or $\ell_{\infty}$, and to a second-order conic program when the norm is $\ell_{2}$. In all three cases, program~\eqref{eqn:multi_policy} is computationally tractable.

	
	
	A natural question is whether programs~\eqref{eqn:sample_robust} and~\eqref{eqn:multi_policy} are asymptotically optimal. \citet{bertsimas2022two} answer this in the affirmative under some assumptions. We state their result in a fully unfolded form as follows. Specifically, we list the standard conditions A1--A3 in Theorem 2 of \cite{bertsimas2022two} and divide the technical condition A4 as the conjunction of its equivalent two parts A4a and A4b. The decomposition of A4 follows from Lemma~1 of \citet{bertsimas2022two}.
	
	\begin{theorem}[Theorem 2 and Lemma 1 of \citealt{bertsimas2022two}]\label{thm:asympt_opt}
		Let $\Xi$ be a polyhedron with at least one vertex. Assume that
		\begin{itemize}
			\item[\textnormal{A1.}] Radius $\epsilon_{N}$ of $\mathcal U^i_N$ in \eqref{eqn:sample_robust} is with $\epsilon_{N}\to 0$ as $N\to\infty$.
			\item[\textnormal{A2.}] The optimal value $v^{\star}$ for \eqref{eqn:two_stage} is finite and $\mathbb{E}[\|\tilde{\bm{\xi}}\|]<\infty$.
			\item[\textnormal{A3.}] The set of optimal first-stage decisions for~\eqref{eqn:two_stage} is nonempty and bounded.
			\item[\textnormal{A4a.}] There exists $\bm{x}\in\mathbb{R}^{n}$ such that $Q(\bm{x},\bm{\zeta})<\infty$ for all $\bm{\zeta}\in\Xi$.
			\item[\textnormal{A4b.}] For the same $\bm{x}\in\mathbb{R}^{n}$ as in A4a and for some $\epsilon>0$, at every vertex $\bm{\zeta}_{0}$ of $\Xi$, there exists an affine policy $\bm{y}_{\bm{\zeta}_{0}}(\cdot)\in\mathcal{L}$ such that
			\[
			\bm{T}\bm{x}+\bm{W}\bm{y}_{\bm{\zeta}_{0}}(\bm{\zeta})\ge\bm{h}(\bm{\zeta}) \qquad \forall\, \bm{\zeta}\in\Xi\cap B(\bm{\zeta}_{0},\epsilon).
			\]
		\end{itemize}
		Then
		\[
		\lim_{N\to\infty}\hat{v}_{N}^{\mathrm{SRO}} \;=\; \lim_{N\to\infty}\hat{v}_{N}^{\mathrm{MP}} \;=\; v^{\star}\qquad a.s.,
		\]
		and any accumulation point of $\{\bm{x}_{N}^{\mathrm{SRO}}\}_{N\in\mathbb{N}}$ or $\{\bm{x}_{N}^{\mathrm{MP}}\}_{N\in\mathbb{N}}$ is almost surely an optimal first-stage decision for~\eqref{eqn:two_stage}, where $\bm{x}_{N}^{\mathrm{SRO}}$ and $\bm{x}_{N}^{\mathrm{MP}}$ are optimal first-stage decisions for~\eqref{eqn:sample_robust} and~\eqref{eqn:multi_policy}, respectively.
	\end{theorem}
	
	Theorem~\ref{thm:asympt_opt} is remarkable, since it establishes asymptotic optimality for both~\eqref{eqn:sample_robust} and~\eqref{eqn:multi_policy} in optimal values and in optimal solutions simultaneously. Among the listed conditions, A1--A3 and A4a are mild and standard. The condition A4b, however, is significantly more technical. It imposes a local affine-policy property at every vertex of $\Xi$ and is therefore difficult to verify in practice. \citet{bertsimas2022two} themselves note that the role of this technical condition is unclear:
	\begin{quote}
		``At first glance this condition might appear to be limiting; nonetheless, despite our efforts, we have been unable to find an instance of two-stage stochastic linear optimization which satisfies A2 and A3 and $\{\bm{x}\in\mathbb{R}^{n}:Q(\bm{x},\bm{\zeta})<\infty\;\forall\bm{\zeta}\in\Xi\}\neq\emptyset$ but does not satisfy A4.'' The authors further remark: ``The question of whether these conditions are equivalent is open for future research.''
	\end{quote}
	
	We refer to the gap left by \citet{bertsimas2022two} as the \emph{open question}, which this paper aims to resolve. To recast it in polyhedral language, we observe that A4a and A4b refer to the same first-stage decision $\bm{x}$. Setting $\bm{b}=\bm{h}_{0}-\bm{T}\bm{x}$, A4a becomes the feasibility of the \emph{recourse system}
	\begin{equation}\label{eqn:linear_sys}
		\bm{W}\bm{y}+\bm{H}\bm{\xi}\ge\bm{b}
	\end{equation}
	for every $\bm{\xi}\in\Xi$, while A4b becomes the existence, at each vertex $\bm{\zeta}_{0}$ of $\Xi$, of an affine map $\bm{y}(\bm{\xi})=\bm{y}_{0}+\bm{Y}\bm{\xi}$ that satisfies~\eqref{eqn:linear_sys} on a neighborhood of $\bm{\zeta}_{0}$ in $\Xi$. We make these two notions explicit.
	
	\begin{definition}[Globally feasible]\label{def:gf}
		The recourse system~\eqref{eqn:linear_sys} is \emph{globally feasible} on $\Xi$ if there exists a policy $\bm{y}(\cdot)\in\mathcal{R}^{d,r}$ such that $\bm{W}\bm{y}(\bm{\xi})+\bm{H}\bm{\xi}\ge\bm{b}$ for every $\bm{\xi}\in\Xi$.
	\end{definition}
	
	\begin{definition}[Locally linear feasible]\label{def:llf}
		The recourse system~\eqref{eqn:linear_sys} is \emph{locally feasible} at a vertex $\bm{\zeta}_{0}$ of $\Xi$ if there exists $\epsilon>0$ and a policy $\bm y(\cdot) \in \mathcal R^{d, r}$ such that for all $\bm \xi \in \Xi\cap B(\bm \zeta_0, \epsilon)$,~\eqref{eqn:linear_sys} is valid by substituting $\bm y = \bm y(\bm \xi)$. Moreover, the recourse system~\eqref{eqn:linear_sys} is called \emph{locally linear feasible} at the vertex $\bm \zeta_0$ if the above policy can be selected as $\bm y(\cdot)\in \mathcal L$.
	\end{definition}
	
	With Definitions~\ref{def:gf} and~\ref{def:llf}, condition A4a is exactly global feasibility of~\eqref{eqn:linear_sys}, and condition A4b is exactly local linear feasibility of~\eqref{eqn:linear_sys} at every vertex of $\Xi$. Therefore, the open question reduces to the following purely polyhedral question on recourse systems.
	\vspace{2mm}
	\begin{tightquote}
		\textbf{Central Question:} For any given globally feasible recourse system~\eqref{eqn:linear_sys}, is it true (or under what conditions) that~\eqref{eqn:linear_sys} is also locally linear feasible at every vertex of $\Xi$? 
	\end{tightquote}
	\vspace{2mm}
	If the above central question admits an unconditional answer ``yes'', then assumption A4b follows from A4a and is therefore redundant in Theorem~\ref{thm:asympt_opt}. If the answer is ``no'' for some polyhedron $\Xi$ and recourse system, then this pair constitutes an instance satisfying A2, A3, and A4a but violating A4b. Such an instance settles the open question of \citet{bertsimas2022two} in the negative. In the next section, we resolve the central question through a geometric characterization: whether global feasibility uniformly implies local linear feasibility depends precisely on the {\emph{simpleness}} of $\Xi$.
	
	\section{A Geometric Characterization via Simple Polyhedral Supports}\label{sec:geometric_character}
	In this section, we study for which polyhedra $\Xi$ global feasibility implies local linear feasibility uniformly over all recourse systems. The answer is exactly simple polyhedra. \Cref{subsec:sufficient} establishes the sufficient direction: global feasibility implies local linear feasibility whenever $\Xi$ is simple. \Cref{subsec:counterex} then exhibits a concrete counterexample of non-simple $\Xi$ and a recourse system in which the implication fails, thereby resolving the open question in \citet{bertsimas2022two}. \Cref{subsec:necessary} shows that this counterexample is not isolated: for every non-simple $\Xi$, a globally feasible recourse system violating local linear feasibility can be constructed. Combining the two directions yields an exact characterization. The proof of results in this section is in \ref{ec_proof_geometric}.
	
	\subsection{Sufficiency: Simple Supports Guarantee Local Linear Feasibility}\label{subsec:sufficient}
	A pointed polyhedron $\Xi$ is called \emph{simple} if every vertex of $\Xi$ is incident to exactly $\dim(\Xi)$ edges. By a slight abuse of terminology, we call a vertex $\bzz \in \Xi$ \emph{simple} if it is incident to exactly $\dim(\Xi)$ edges. Then a pointed polyhedron is simple if and only if all its vertices are simple. Our first result, as shown in the following theorem, indicates that, for a globally feasible recourse system, local linear feasibility at a vertex is guaranteed by the simpleness of that vertex.
	
	\begin{theorem}\label{thm:no_more_d_edges}
		Let $\bm \zeta_0$ be a simple vertex of $\Xi$; that is, there are exactly $\dim (\Xi)$ edges incident to the vertex $\bm \zeta_0$ on $\Xi$, then any globally feasible system~\eqref{eqn:linear_sys} on $\Xi$ is also locally linear feasible at $\bm \zeta_0$. 
	\end{theorem}

	\Cref{thm:no_more_d_edges} illustrates an important sufficient condition under which global feasibility implies local linear feasibility. To apply it, one needs only verify the number of edges incident to each vertex of $\Xi$. A direct consequence of \Cref{thm:no_more_d_edges} is that, if $\Xi$ is a simple polyhedron, then the technical assumption A4b in \Cref{thm:asympt_opt} is satisfied when the milder assumption A4a is imposed. 
	\begin{corollary}\label{coro:simple_poly}
		Let $\Xi$ be a simple polyhedron, and A1-A3 and A4a in \Cref{thm:asympt_opt} hold, then
		\[
		\lim_{N\to\infty}\hat{v}_{N}^{\text{SRO}} \;=\; \lim_{N\to\infty}\hat{v}_{N}^{\text{MP}} \;=\; v^{\star}\qquad a.s.,
		\]
		and any accumulation point of $\{\bm{x}_{N}^{\text{SRO}}\}_{N\in\mathbb{N}}$ or $\{\bm{x}_{N}^{\text{MP}}\}_{N\in\mathbb{N}}$ is almost surely an optimal first-stage decision for~\eqref{eqn:two_stage}, where $\bm{x}_{N}^{\text{SRO}}$ and $\bm{x}_{N}^{\text{MP}}$ are optimal first-stage decisions for~\eqref{eqn:sample_robust} and~\eqref{eqn:multi_policy}, respectively.
	\end{corollary}
	\begin{proof}[Proof of \Cref{coro:simple_poly}]
		By assumption A4a, the recourse system \eqref{eqn:linear_sys} is globally feasible. By \Cref{thm:no_more_d_edges}, the system \eqref{eqn:linear_sys} is locally linear feasible at every vertex of $\Xi$. Therefore, assumption A4b is satisfied. Hence, the conclusion follows directly from \Cref{thm:asympt_opt}. 
	\end{proof}
	
	A simple polyhedron is a standard object in polyhedral theory \citep{schrijver1998theory, ziegler2012lectures}. It coincides with non-degeneracy at every vertex in the sense of linear programming, and is therefore a familiar regularity property in optimization. To the best of our knowledge, Corollary~\ref{coro:simple_poly} is the first result that establishes the role of simple polyhedra in connecting local and global feasibility of recourse systems, as well as in the asymptotic optimality of two-stage problems. \Cref{fig:simple_poly_examples} provides some instances of simple polyhedra. As shown in \Cref{fig:simple_poly_examples}, our \Cref{thm:no_more_d_edges} strictly generalizes Theorem~1 in \cite{bertsimas2012power}, which focuses on a simplex support.
	
	\begin{figure}[t]
		\centering
		\includegraphics[width=1.0\textwidth]{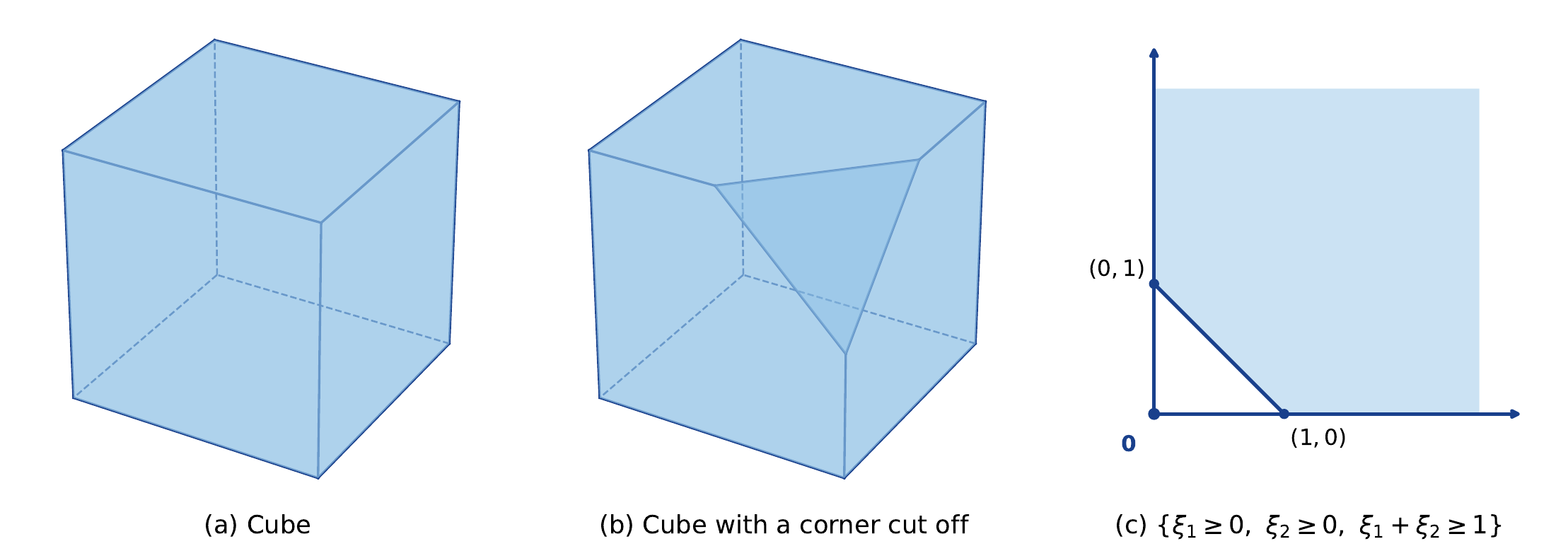}
		\caption{Some instances of simple polyhedra beyond the simplex and simplicial cone.}
		\label{fig:simple_poly_examples}
	\end{figure}
	
	Moreover, the geometric feature in \Cref{thm:no_more_d_edges} is not only of theoretical interest. It is also applicable to many real-world two-stage problems. For instance, the one-sided budgeted uncertainty set with a non-integer budget \citep{bertsimas2004price,ben2004adjustable} is a simple polyhedron (see \Cref{coro_simple_one_sided_budget} in \Cref{sec:budgeted_uncertainty_set}). In addition, the box uncertainty set used in two-stage adjustable robust linear optimization is also a simple polyhedron. In this case, by \Cref{coro:simple_poly}, we obtain the asymptotic optimality of~\eqref{eqn:sample_robust} and \eqref{eqn:multi_policy} under the mild assumptions A1--A3 and A4a in \Cref{thm:asympt_opt}.
	
	Thanks to \Cref{thm:no_more_d_edges}, we only need to search for the polyhedron with at least one vertex incident to more than ${\rm dim}(\Xi)$ edges, \textit{i.e.}, a non-simple vertex, if any counterexample that is globally feasible but not locally linear feasible exists. The next subsection exhibits such a counterexample explicitly.
	

	\subsection{A Three-Dimensional Counterexample}\label{subsec:counterex}
	In this subsection, we construct a three-dimensional polyhedral support $\Xi$ together with a recourse system~\eqref{eqn:linear_sys} that is globally feasible but not locally linearly feasible. As a byproduct, this yields a problem instance that gives a direct negative answer to the open question posed by \citet{bertsimas2022two}.
	The support $\Xi \subset \mathbb{R}^{3}$ is depicted in \Cref{fig:counterex}. The origin is the unique non-simple vertex of the polytope, where seven edges are incident. In the following proposition, we construct a recourse system that is pointwise feasible everywhere on $\Xi$ but with no affine recourse policy that satisfies the recourse constraints in any neighborhood of the origin. Hence, the failure of local linear feasibility at this specific example is attributable to the non-simpleness of this single vertex.
	
	\begin{figure}[t]
		\centering
		\includegraphics[width=0.6\textwidth]{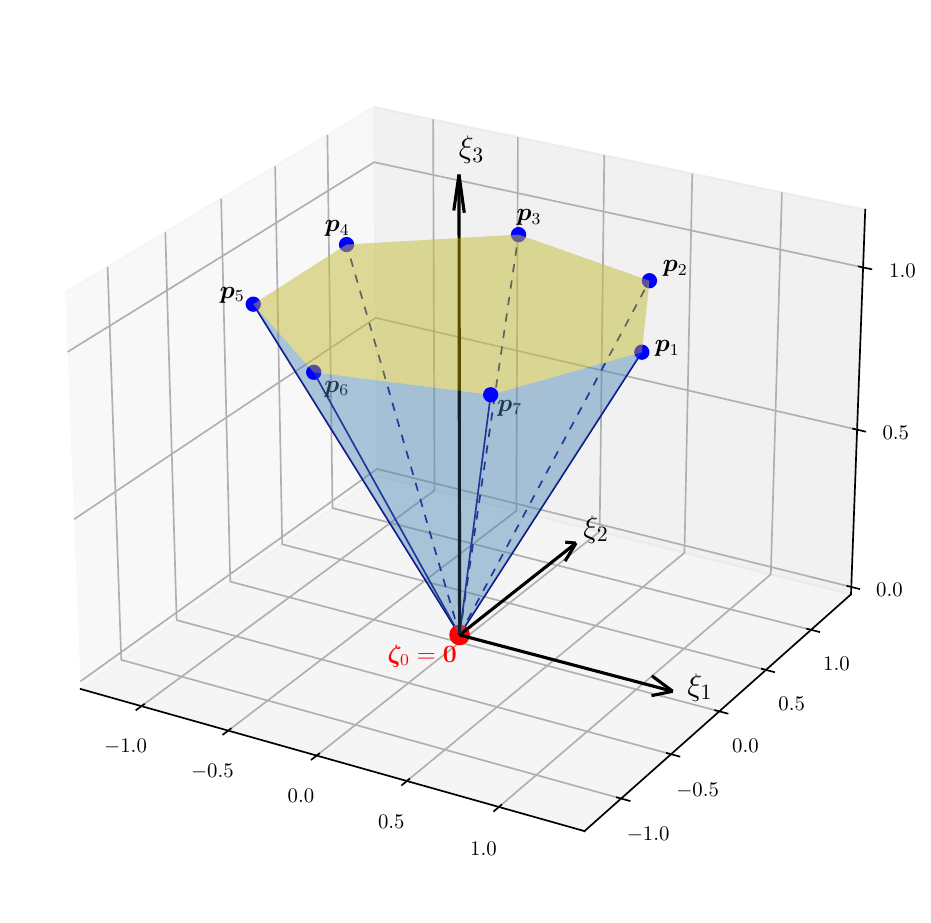}
		\caption{The polytope $\Xi = {\rm conv}\{\bm 0, \bm p_1, \ldots, \bm p_7\}$ used in \Cref{prop:counterexample}.}
		\label{fig:counterex}
	\end{figure}
	\begin{proposition}\label{prop:counterexample}
		Let $m=4$, $r=1$ and $d=3$. Let the matrices and vectors
		\begin{equation*}
			\bm W=\begin{bmatrix}
				1\\1\\-1\\-1
			\end{bmatrix}, \quad \bm H = \begin{bmatrix}
				1 & 2 & 1\\
				1&-2&3\\
				2&-2&2\\
				-1&2&3\\
			\end{bmatrix}, \quad \text{and} \quad \bm b=\begin{bmatrix}
				0\\0\\0\\0
			\end{bmatrix}.
		\end{equation*}
		Let the polyhedron $\Xi$ be
		\begin{equation*}
			\Xi = {\rm conv}\{\bm 0, \bm p_1, \ldots, \bm p_7\}, \text{ where } \bm p_j = [\cos \theta_j, \sin \theta_j, 1]^\top, \; \theta_j = \frac{2\pi (j-1)}{7}, \; \forall j\in [7]. 
		\end{equation*}
		Then~\eqref{eqn:linear_sys} is globally feasible on $\Xi$, but not locally linear feasible at the vertex $\bm \zeta_0 = \bm 0$.
	\end{proposition}
	
	The example in \Cref{prop:counterexample} provides a concrete instance of a recourse system that is globally feasible but not locally linear feasible. In the following corollary, we embed this recourse system into a specified two-stage stochastic linear program. The resulting problem instance explicitly violates the technical assumption A4b in Theorem~\ref{thm:asympt_opt}, while preserving the remaining standard assumptions.
	
	\begin{corollary}\label{coro:counterexample}
		Let $\Xi$ be as in \Cref{prop:counterexample}, and let $\mb P$ be any
		distribution supported on $\Xi$ with $\mb E[\|\tilde{\bm \xi}\|] < \infty$, \textit{e.g.}, the uniform distribution on $\Xi$. Consider the two-stage stochastic linear program:
		\begin{equation}\label{eqn:counterexample}
			\min_{\bm x\in \mb R^n} \bm c^\top \bm x + \mb E\left[\min_{y\in \mb R} \left\{ 0 ~:~ \bm T\bm x + \bar{\bm W}y \geq \bm h(\tilde{\bm \xi})\right\} \right],
		\end{equation}
		where 
		\begin{equation*}
			\bm T = \begin{bmatrix}
				\bm I\\
				-\bm I\\
				\bm O
			\end{bmatrix}, \quad \bar{\bm W} = \begin{bmatrix}
				\bm O\\
				-\bm O\\
				\bm W
			\end{bmatrix}, \quad \bm h(\bm \xi) = \begin{bmatrix}
				\bm e\\
				-\bm e\\
				\bm b - \bm H\bm \xi
			\end{bmatrix} \quad \forall \bm \xi\in \Xi,
		\end{equation*}
		and the matrices $\bm W$, $\bm H$ and the vector $\bm b$ are defined as in
		\Cref{prop:counterexample}. Then problem \eqref{eqn:counterexample}
		satisfies assumptions A2, A3, and A4a of \Cref{thm:asympt_opt} but
		violates A4b. Consequently, it provides a counterexample that resolves the
		open question raised in \cite{bertsimas2022two}.
	\end{corollary}
	
	\begin{proof}[Proof of \Cref{coro:counterexample}]
		The recourse system in \eqref{eqn:counterexample} is equivalent to that $\bm x = \bm e$ and the system \eqref{eqn:linear_sys}. By \Cref{prop:counterexample}, the system \eqref{eqn:linear_sys} is globally feasible; that is, when $\bm x = \bm e$, for any $\bm \xi\in \Xi$, there exists some $y$ that satisfies the recourse system in \eqref{eqn:counterexample}. This verifies assumption A4a in \Cref{thm:asympt_opt}. By construction, the only feasible first-stage decision of \eqref{eqn:counterexample} is $\bm x = \bm e$, so the optimal solution of \eqref{eqn:counterexample} is given by $\bm x^\star = \bm e$, and assumption A3 in \Cref{thm:asympt_opt} is satisfied. The optimal value $v^\star = \bm c^\top \bm e$ is finite, so assumption A2 is also satisfied. However, by \Cref{prop:counterexample}, the system \eqref{eqn:linear_sys} is not locally linear feasible at $\bm 0$. Hence, the feasible $y$ for the recourse system in \eqref{eqn:counterexample} can not be selected as an affine function in $\Xi \cap B(\bm 0, \epsilon)$, which violates A4b in \Cref{thm:asympt_opt} and finishes the proof.
	\end{proof}

	Corollary~\ref{coro:counterexample} gives a direct negative answer to the open question in \citet{bertsimas2022two}: A4b is generally not implied by A2, A3, and A4a. Moreover, the failure has an immediate implication for the multi-policy approximation~\eqref{eqn:multi_policy} of the two-stage sample robust problem. Consider the problem instance~\eqref{eqn:counterexample}, and suppose the true distribution $\mathbb{P}$ has positive mass at $\bm{0}$. Then, with positive probability, some sample point satisfies $\bm{\xi}^i = \bm{0}$. Any feasible $\bm{y}^i(\cdot)$ for~\eqref{eqn:multi_policy} would then be a locally linear feasible policy for the recourse system in~\eqref{eqn:counterexample} at $\bm{\xi}^i = \bm{0}$. Since no such policy exists by \Cref{prop:counterexample}, the formulation~\eqref{eqn:multi_policy} in this case is infeasible, and its asymptotic optimality fails.
	
	\subsection{Necessity and Exact Characterization}\label{subsec:necessary}
	The preceding example in \Cref{subsec:counterex} shows that locally linear feasibility may fail on a particular non-simple support. We now show that this failure is not incidental but structural: whenever a polyhedron has a non-simple vertex, one can construct a recourse system that is globally feasible on the support but not locally linear feasible at that vertex. This completes the necessary direction of our geometric characterization for the implication from global feasibility to locally linear feasibility. 
	
	\begin{theorem}\label{thm:general_impossible}
		For any pointed polyhedron $\Xi$ with a non-simple vertex $\bzz$, there exists a recourse system~\eqref{eqn:linear_sys} that is globally feasible on $\Xi$ but not locally linear feasible at $\bzz$.
	\end{theorem}
	
	Theorem~\ref{thm:general_impossible} proves the necessary direction of the characterization. If a polyhedron support has even a single non-simple vertex, then global feasibility does not imply locally linear feasibility uniformly over all recourse systems on that support. Together with Corollary~\ref{coro:simple_poly}, this shows that the simpleness of $\Xi$ is precisely the geometric condition under which A4b follows from A4a uniformly over all recourse systems. Our next corollary captures this geometric characterization formally.
	
	\begin{corollary}\label{coro:exact_chara}
		Let $\Xi$ be a pointed polyhedron, then the following conditions are equivalent:
		\begin{itemize}
			\item[(i)] In \Cref{thm:asympt_opt}, assumption A4b is implied by A4a for all two-stage problem~\eqref{eqn:two_stage}.
			\item[(ii)] All globally feasible recourse system~\eqref{eqn:linear_sys} on $\Xi$ is locally linear feasible at every vertex of $\Xi$.
			\item[(iii)] $\Xi$ is a simple polyhedron.
		\end{itemize}
	\end{corollary}
	
	\begin{proof}[Proof of \Cref{coro:exact_chara}]
		The equivalence between $(i)$ and $(ii)$ follows from the definition of global feasibility and locally linear feasibility. The implication from $(iii)$ to $(ii)$ is the direct consequence of \Cref{thm:no_more_d_edges}, and the implication from $(ii)$ to $(iii)$ is due to \Cref{thm:general_impossible}.
	\end{proof}
	
	\Cref{coro:exact_chara} provides a complete answer to the open question raised in \cite{bertsimas2022two}. The geometric characterization, \textit{i.e.}, the simpleness of the support $\Xi$, precisely captures when assumption A4b in \Cref{thm:asympt_opt} is implied by the milder assumption A4a. Therefore, it clarifies when the technical assumption A4b is redundant and can be removed in \Cref{thm:asympt_opt}, and when it instead deserves careful attention. Specifically, if the support set is a simple polyhedron, then asymptotic optimality is guaranteed, and the decision obtained from the tractable multi-policy approximation~\eqref{eqn:multi_policy} is reliable. However, if the support set is non-simple, the decision maker should proceed with caution, because assumption A4b may fail, asymptotic optimality is no longer ensured, and consequently the multi-policy approximation~\eqref{eqn:multi_policy} may lead to unreliable decisions.
	
	Although the open question has been resolved through the simpleness of the support $\Xi$, our result also raises two new questions for decision makers. The first is how to verify that the polyhedron $\Xi$ is simple, particularly for the supports commonly used in practice. The second is how a decision maker can check whether assumption A4b holds when $\Xi$ is non-simple. To address the first question, the next section examines the simpleness of two widely used supports, \textit{i.e.}, the two-sided and the one-sided budgeted uncertainty sets. The second question is then deferred to \Cref{sec:algorithm}.

	\section{Simpleness of Budgeted Uncertainty Sets}\label{sec:budgeted_uncertainty_set}
	
	Having established the geometric characterization for the asymptotic optimality of the multi-policy approximation~\eqref{eqn:multi_policy}, we now investigate the simpleness of one of the most popular uncertainty sets in practice, namely the budgeted uncertainty set. The budgeted uncertainty set was introduced in \cite{bertsimas2004price}, where the decision maker specifies a budget on the sum of adversarial deviations of the uncertain parameters from their nominal values. It arises in a broad range of applications, such as modeling uncertain demand in inventory planning \citep{bertsimas2006robust}, uncertain nodal net injections in the security-constrained unit commitment problem \citep{bertsimas2012adaptive}, and uncertain demand and travel time in the vehicle routing problem \citep{munari2019robust}. Despite its wide use, the simpleness of the budgeted uncertainty set and its associated geometric features, such as vertices and edges, have not yet been studied. This section aims to fill this gap by providing a complete geometric analysis of the budgeted uncertainty set.
	
	Motivated by modeling considerations in practice, we consider two types of budgeted uncertainty sets. The first type is the classical budgeted uncertainty set defined as the intersection of the $1$-norm ball and the $\infty$-norm ball, which we refer to as the \emph{two-sided budgeted uncertainty set}. In many applications, however, the underlying uncertain parameters are inherently non-negative, \textit{e.g.}, travel times in the vehicle routing problem cannot take negative values. To accommodate such cases, we further consider a non-negative counterpart of the first type of budgeted uncertainty sets, which we refer to as the \emph{one-sided budgeted uncertainty set}. We begin our analysis with the two-sided budgeted uncertainty set. The proof of results in this section is in \ref{ec_proof_budget}.
	
	\subsection{Geometric Analysis of the Two-Sided Budgeted Uncertainty Set}\label{subsec:two_sided}
	Throughout this subsection, we consider the following two-sided budgeted uncertainty set
	\begin{equation}\label{eqn:two_sided_budget}
		\Xi = \left\{\bm \xi \in \mb R^d ~:~ \|\bm \xi\|_{\infty} \leq 1, \; \|\bm \xi\|_{1} \leq \Gamma\right\},
	\end{equation}
	where $\Gamma>0$ is a given budget. We focus on $\Gamma\in (0, d)$, as otherwise $\Xi$ reduces to a trivially simple box. Our following proposition fully characterizes the vertices and edges of $\Xi$ for integer budget $\Gamma$.
	\begin{proposition}\label{prop:two_sides_budget_uncertainty_set_integer}
		Let $\Xi$ be the budgeted uncertainty set in \eqref{eqn:two_sided_budget} with integer budget $\Gamma \in [d-1]$:
		\begin{itemize}
			\item[(i)] A point $\bm \zeta_0 \in \mb R^d$ is a vertex of $\Xi$ if and only if exactly $\Gamma$ of its coordinates equal $\pm 1$ and the remaining $d-\Gamma$ coordinates equal $0$. Consequently, $\Xi$ has $2^\Gamma \binom{d}{\Gamma}$ vertices in total.
			\item[(ii)] Consider the vertex $\bm \zeta_0 = \sum_{i\in [\Gamma]} \bm e_i$. All edges incident to $\bm \zeta_0$ have directions of the form
			\begin{equation*}
				\bm v^{i, j, \sigma} = \sigma \bm e_j - \bm e_i, \quad \forall\, i\in [\Gamma],\; j\in [d] \backslash [\Gamma],\; \sigma \in \{-1, 1\}.
			\end{equation*}
			Therefore, the number of edges incident to $\bm \zeta_0$ in this case is given by $L = 2\Gamma(d-\Gamma)$.
		\end{itemize}
	\end{proposition}

	Although \Cref{prop:two_sides_budget_uncertainty_set_integer}$(ii)$ only describes the edges incident to the specific vertex $\bm \zeta_0 = \sum_{i\in [\Gamma]} \bm e_i$, the same conclusion applies to every vertex of $\Xi$. Indeed, since $\Xi$ is invariant under coordinate permutations and sign changes, every vertex has the same number of incident edges, and the edge directions at any other vertex can be obtained from those at $\bm \zeta_0$ through such symmetries. Therefore, the analysis in \Cref{prop:two_sides_budget_uncertainty_set_integer}$(ii)$ for the specific vertex is without loss of generality.
	
	Two consequences follow from \Cref{prop:two_sides_budget_uncertainty_set_integer}. First, the number of vertices, $2^\Gamma \binom{d}{\Gamma}$, can grow rapidly in the budget $\Gamma$, which poses a challenge for vertex-based approaches when the budget is large. Second, for any $d \geq 3$ and any integer budget $\Gamma \in [d-1]$, we have $L = 2\Gamma(d-\Gamma) > d$, so that the number of edges incident to each vertex strictly exceeds the dimension. Therefore, the two-sided budgeted uncertainty set with an integer budget is never simple whenever $d \geq 3$. This is an unfavorable finding, as it means that assumption A4b in \Cref{thm:asympt_opt} cannot be dropped in this case and must be verified separately to guarantee asymptotic optimality. It is then natural to ask whether the situation improves when the budget $\Gamma$ is non-integer, which we investigate next.

	\begin{proposition}\label{prop:two_sides_budget_uncertainty_set_float}
		Let $\Xi$ be the budgeted uncertainty set in \eqref{eqn:two_sided_budget} with non-integer budget $\Gamma \in (0, d)$, and denote the integer part of $\Gamma$ as $\uga$. Then the following statements hold:
		\begin{itemize}
			\item[(i)] A point $\bm \zeta_0 \in \mb R^d$ is a vertex of $\Xi$ if and only if exactly $\uga$ of its coordinates equal $\pm 1$, exactly one coordinate equal to $\pm (\Gamma-\uga)$, and the remaining $d-\uga-1$ coordinates (if any) equal $0$. Consequently, $\Xi$ has $2^{\uga+1}\binom{d}{\uga}(d-\uga)$ vertices in total.
			\item[(ii)] Consider the vertex $\bm \zeta_0 = \sum_{i\in [\uga]} \bm e_i + (\Gamma - \uga)\bm e_{\uga+1}$, then
			\begin{itemize}
				\item[(a)] If $\uga < \Gamma < d-1$, then the edges incident to $\bm \zeta_0$ have directions $\bm v^i$ and $\bm v^{j, \sigma}$ given by
				\begin{equation*}
					\bm v^i = \bm e_{\uga+1} - \bm e_i, \; \; \bm v^{j, \sigma} = \sigma\bm e_j - \bm e_{\uga+1} \quad \forall i\in [\uga], \; j\in [d] \backslash [\uga+1], \; \sigma\in \{-1, 1\}.
				\end{equation*}
				Therefore, the number of edges incident to $\bm \zeta_0$ in this case is given by $L = 2d-\uga-2$.
				
				\item[(b)] If $d-1 < \Gamma < d$, then the edges incident to $\bm \zeta_0$ have directions $\bm v^0$ and $\bm v^i$ given by
				\begin{equation*}
					\bm v^0 = -\bm e_d, \; \; \bm v^i = \bm e_{d} - \bm e_i \quad \forall i\in [\uga] = [d-1].
				\end{equation*}
				Therefore, the number of edges incident to $\bm \zeta_0$ in this case is given by $L = 1+d-1 = d$.
			\end{itemize}
		\end{itemize}
	\end{proposition}

	As in \Cref{prop:two_sides_budget_uncertainty_set_integer}$(ii)$, the result in \Cref{prop:two_sides_budget_uncertainty_set_float}$(ii)$ for the specific vertex is without loss of generality, since $\Xi$ is invariant under coordinate permutations and sign changes. Moreover, we observe that the number of edges incident to each vertex can equal $d$, raising the possibility that $\Xi$ is a simple polyhedron. This is formalized in the following result.
	
	\begin{corollary}[Simpleness of two-sided budgeted uncertainty set]\label{coro_simple_two_sided_budget}
		Let $\Xi$ be as in \eqref{eqn:two_sided_budget} with budget $\Gamma \in (0, d)$ and $d\geq 3$. Then $\Xi$ is simple if and only if $\Gamma \in (d-2, d-1) \cup (d-1, d)$.
	\end{corollary}
	
	\begin{proof}[Proof of \Cref{coro_simple_two_sided_budget}]
		By \Cref{prop:two_sides_budget_uncertainty_set_integer}, the budgeted uncertainty set $\Xi$ is never simple when $\Gamma$ is an integer. Therefore, $\Xi$ is simple only if $\Gamma$ is non-integer. By \Cref{prop:two_sides_budget_uncertainty_set_float}, $\Xi$ is simple if and only if $(i)$ $\uga < \Gamma < d-1$ and $2d-\uga-2 = d$, which is equivalent to $\Gamma \in (d-2, d-1)$; or $(ii)$ $d-1 < \Gamma < d$. Combining these two cases yield the desired conclusion. 
	\end{proof}
	
	\Cref{coro_simple_two_sided_budget} has a direct practical implication for the decision maker. When the budget of interest is a large integer, say $\Gamma = d-1$, $\Xi$ is not simple and the technical assumption A4b in \Cref{thm:asympt_opt} cannot be dropped, which may be difficult to verify in practice. By slightly perturbing the budget to a nearby non-integer value, \textit{e.g.}, $\Gamma^\prime = d-1-\varepsilon$ or $\Gamma^\prime = d-1+\varepsilon$ for a sufficiently small $\varepsilon > 0$, the uncertainty set becomes simple, so that assumption A4b is automatically satisfied and asymptotic optimality of the multi-policy approximation~\eqref{eqn:multi_policy} follows directly from \Cref{coro:simple_poly}. Since the perturbed set can be chosen arbitrarily close to the original one, this modification incurs negligible conservativeness while yielding substantially stronger theoretical guarantees.
	
	Having completed the geometric analysis of the two-sided budgeted uncertainty set, we now turn to its non-negative counterpart, the one-sided budgeted uncertainty set, in the next subsection.

	\subsection{Geometric Analysis of the One-Sided Budgeted Uncertainty Set}\label{subsec:one_sided}
	In this subsection, we study the one-sided budgeted uncertainty set in the form of
	\begin{equation}\label{eqn:one_sided_budget}
		\Xi = \left\{\bm \xi \in \mb R^d ~:~ \bm 0 \leq \bm \xi \leq \bm e, \; \bm e^\top \bm \xi \leq \Gamma\right\},
	\end{equation}
	where $\Gamma\in(0, d)$ is a prescribed budget. As in \Cref{subsec:two_sided}, we begin with the integer-budget case.
	
	\begin{proposition}\label{prop:one_sides_budget_uncertainty_set_integer}
		Let $\Xi$ be as in \eqref{eqn:one_sided_budget} with integer budget $\Gamma \in [d-1]$. Then we have
		\begin{itemize}
			\item[(i)] A point $\bm \zeta_0 \in \mb R^d$ is a vertex of $\Xi$ if and only if at most $\Gamma$ of its coordinates equal $1$ and the remaining all coordinates equal $0$. Consequently, $\Xi$ has $\sum_{k=0}^{\Gamma}\binom{d}{k}$ vertices in total.
			\item[(ii)] Consider the vertex $\bm \zeta_0 = \sum_{i\in [K]} \bm e_i$ where $K\leq \Gamma$, then
			\begin{itemize}
				\item[(a)] If $K<\Gamma$, then the edges incident to $\bm \zeta_0$ have directions $\bm v^j$ and $\bm u^i$ from
				\begin{equation*}
					\bm v^{j} = \bm e_j, \; \; \bm u^i = -\bm e_i \quad \forall i\in [K], \; j\in [d] \backslash [K].
				\end{equation*}
				Therefore, the number of edges incident to $\bm \zeta_0$ in this case is given by $L = d$.
				
				\item[(b)] If $K = \Gamma$, then the edges incident to $\bm \zeta_0$ have directions $\bm v^{i, j}$ or $\bm u^i$ from
				\begin{equation*}
					\bm v^{i, j} = \bm e_j - \bm e_i, \; \; \bm u^i = - \bm e_i \quad \forall i\in [\Gamma], \; j\in [d] \backslash [\Gamma].
				\end{equation*}
				Therefore, the number of edges incident to $\bm \zeta_0$ in this case is given by $L = \Gamma + \Gamma (d - \Gamma)$.         
			\end{itemize}
		\end{itemize}
	\end{proposition}
	
	As before, the result in \Cref{prop:one_sides_budget_uncertainty_set_integer}$(ii)$ for the specific vertex is without loss of generality, since $\Xi$ is invariant under coordinate permutations. A direct consequence of \Cref{prop:one_sides_budget_uncertainty_set_integer} is that, among integer budgets, $\Xi$ is simple only when $\Gamma = 1$, in which case $\Xi$ reduces to the standard simplex. More importantly, in contrast to the two-sided budgeted uncertainty set, the one-sided budgeted uncertainty set exhibits a more favorable structure: non-simple vertices arise only in the case studied in \Cref{prop:one_sides_budget_uncertainty_set_integer}$(ii)(b)$, \textit{i.e.}, at vertices that satisfy $\bm e^\top \bzz = \Gamma$, while all other vertices are simple. Consequently, by \Cref{thm:no_more_d_edges}, locally linear feasibility continues to hold at all simple vertices, and the number of non-simple vertices is at most $\binom{d}{\Gamma}$. This partial simpleness suggests that assumption A4b may be substantially easier to handle in the one-sided case. It is then natural to ask whether a non-integer budget can further eliminate the remaining non-simple vertices, which we study next.
	
	\begin{proposition}\label{prop:one_sides_budget_uncertainty_set_float}
		Let $\Xi$ be the budgeted uncertainty set in \eqref{eqn:one_sided_budget} with non-integer budget $\Gamma \in (0, d)$, and denote the integer part of $\Gamma$ as $\uga$. Then the following statements hold:
		\begin{itemize}
			\item[(i)]A point $\bm \zeta_0 \in \mb R^d$ is a vertex of $\Xi$ if and only if either at most $\uga$ coordinates equal to $1$ and the rest coordinates equal to $0$, or exactly $\uga$ coordinates equal to $1$, one coordinate equal to $\Gamma - \uga$, and the rest coordinates equal to $0$. Consequently, $\Xi$ has $\sum_{k=0}^{\uga}\binom{d}{k} + (d-\uga)\binom{d}{\uga}$ vertices in total.
			
			\item[(ii)] Consider the vertex $\bm \zeta_0 = \sum_{i\in [K]} \bm e_i$ where $K\leq \uga$, then the edges incident to $\bm \zeta_0$ have directions
			\begin{equation*}
				\bm v^j = \bm e_j,\; \; \bm u^i = -\bm e_i \quad \forall i\in [K], \; j\in [d] \backslash [K].
			\end{equation*}
			Therefore, the number of edges incident to $\bm \zeta_0$ in this case is given by $L = d$.
			
			\item[(iii)] Consider the vertex $\bm \zeta_0 = \sum_{i\in [\uga]} \bm e_i + (\Gamma - \uga)\bm e_{\uga +1}$, then the edges incident to $\bm \zeta_0$ have directions
			\begin{equation*}
				-\bm e_{\uga+1}, \; \; \bm v^j = \bm e_j-\bm e_{\uga+1},\; \; \bm u^i = \bm e_{\uga+1}-\bm e_i \quad \forall i\in [\uga], \; j\in [d] \backslash [\uga+1].
			\end{equation*}
			Therefore, the number of edges incident to $\bm \zeta_0$ in this case is given by $L = d$.
		\end{itemize}
	\end{proposition}
	
	As in \Cref{prop:one_sides_budget_uncertainty_set_integer}$(ii)$, the results in \Cref{prop:one_sides_budget_uncertainty_set_float}$(ii)$ and \ref{prop:one_sides_budget_uncertainty_set_float}$(iii)$ for the specific vertices are without loss of generality due to the invariance of $\Xi$ under coordinate permutations. More importantly, the number of edges incident to every vertex is exactly $d$, so the one-sided budgeted uncertainty set $\Xi$ with a non-integer budget is always simple. Combining this observation with \Cref{prop:one_sides_budget_uncertainty_set_integer}, we conclude the simpleness of $\Xi$ in \eqref{eqn:one_sided_budget} in the following corollary.

	\begin{corollary}[Simpleness of one-sided budgeted uncertainty set]\label{coro_simple_one_sided_budget}
		Let $\Xi$ be as in \eqref{eqn:one_sided_budget} with budget $\Gamma \in (0, d)$ and $d\geq 3$. Then $\Xi$ is simple if and only if $\Gamma \notin \{2, 3, \dots, d-1\}$.
	\end{corollary}
	
	\begin{proof}[Proof of \Cref{coro_simple_one_sided_budget}]
		By \Cref{prop:one_sides_budget_uncertainty_set_integer}, $\Xi$ with an integer budget is simple if and only if $\Gamma =1$. By \Cref{prop:one_sides_budget_uncertainty_set_float}, $\Xi$ with a non-integer budget is always simple. Then the conclusion follows.
	\end{proof}
	
	\Cref{coro_simple_one_sided_budget} explicitly characterizes when the one-sided budgeted uncertainty set is simple. More significantly, it carries a valuable practical implication for the decision maker. Whenever a one-sided budgeted uncertainty set is employed to model the uncertainty, one can always perturb the budget by an arbitrarily small amount to make $\Xi$ simple. In contrast to the two-sided case, such a perturbation is available for \emph{any} initial budget, and once $\Xi$ is simple, the asymptotic optimality of the tractable multi-policy approximation~\eqref{eqn:multi_policy} is automatically guaranteed by \Cref{coro:simple_poly}.
	
	Nevertheless, non-simple polyhedral supports remain prevalent in general. A representative example is the two-sided budgeted uncertainty set $\Xi$ in \eqref{eqn:two_sided_budget} with budget $\Gamma = d-3$, for which a small perturbation of $\Gamma$ does not restore simpleness, as shown in \Cref{coro_simple_two_sided_budget}. Even so, non-simpleness of the support does not necessarily imply that assumption A4b, or the asymptotic optimality of the multi-policy approximation~\eqref{eqn:multi_policy}, fails: indeed, \Cref{coro:exact_chara} shows that A4b fails only for certain specific recourse systems. For a given recourse system on a non-simple support, locally linear feasibility may still hold at the non-simple vertices. Therefore, once the support is non-simple, the remaining question becomes computational: how can one verify A4b at a vertex of interest? In the next section, we answer this question by developing a polynomial-time algorithm that checks locally linear feasibility from the local edge structure of the support.

	\section{A Verification Algorithm for Locally Linear Feasibility}\label{sec:algorithm}
	\Cref{sec:geometric_character} establishes that global feasibility implies locally linear feasibility whenever the domain $\Xi$ is simple. In this case, no further verification is needed. However, many uncertainty sets used in practice are not simple. One example is the two-sided budgeted uncertainty set with budget $\Gamma = d-3$, discussed in \Cref{sec:budgeted_uncertainty_set}. At a non-simple vertex, global feasibility alone does not guarantee locally linear feasibility, so this property must be checked for the given recourse system.
	
	In this section, we derive equivalent conditions for locally linear feasibility of system~\eqref{eqn:linear_sys} at a given vertex $\bzz$ and use these conditions to develop a verification algorithm. We show that, given $\bm W$, $\bm H$, $\bm b$, $\bzz$, and the directions of the edges of $\Xi$ incident to $\bzz$, the proposed verification algorithm can be performed in polynomial time by solving a sequence of structured linear programs. The proofs of all results in this section are provided in \ref{ec_proof_algorithm}.
	
	Throughout this section, we fix a vertex $\bzz$ of $\Xi$ and assume that the directions $\bm v_1,\ldots,\bm v_L$ of the edges incident to $\bzz$ are given. The tangent cone at $\bzz$ is therefore
	\[
	\mc T_{\bzz}={\rm cone}\{\bm v_1,\ldots,\bm v_L\}.
	\]
	With these edge directions given, the following lemma provides an equivalent condition for locally linear feasibility at $\bzz$, which forms the basis of our verification algorithm.
	
	\begin{lemma}\label{lem:llf_technical_condition}
		Let $\bzz$ be a vertex of $\Xi$, and let $\bm v_1,\ldots,\bm v_L$ be the directions of the edges incident to $\bzz$. Then system~\eqref{eqn:linear_sys} is locally linear feasible at $\bzz$ if and only if the set $\mc F \triangleq \cap_{i=1}^m \; \mc F_i$ is non-empty, where
		\begin{equation}\label{eqn:mcf}
			\mc F_i \triangleq \Big\{(\bm Y, \bm y) \,:\, R_i(\bm Y, \bm y) > b_i \Big\} \; \bigcup \; \Big\{(\bm Y, \bm y) \,:\, R_i(\bm Y, \bm y) = b_i, \;  \bm W_i \bm Y\bm v_{\ell} + \bm H_i \bm v_{\ell} \geq 0\quad \forall \ell\in [L]\Big\},
		\end{equation}
		and $R_i(\bm Y, \bm y) \triangleq \bm W_i(\bm Y\bzz + \bm y) + \bm H_i\bzz$ is an affine function of $(\bm Y, \bm y)$ for each $i\in [m]$, and $\bm W_i$ and $\bm H_i$ are the $i$-th row vectors of matrices $\bm W$ and $\bm H$ in~\eqref{eqn:linear_sys}, respectively.
	\end{lemma}
	
	\Cref{lem:llf_technical_condition} reduces the verification of locally linear feasibility to checking whether a set $\mc F$, defined as the intersection of all $\mc F_i$, is non-empty. Each set $\mc F_i$ allows two cases: the $i$-th constraint $R_i(\bm Y, \bm y) > b_i $ is strictly satisfied at $\bzz$, or it is tight at $\bzz$ and remains satisfied along every incident edge direction. While $\mc F_i$ is expressed as a union of two sets, it is convex, as shown below.
	
	\begin{proposition}\label{prop:convex_mcf}
		The sets $\mc F_i$ in \eqref{eqn:mcf}, for all $i\in [m]$,  and their intersection set $\mc F$ are convex.
	\end{proposition}
	
	Although the set $\mc F$ is convex by \Cref{prop:convex_mcf}, it is generally not closed because of the strict inequalities on $R_i(\bm Y, \bm y)$ in \eqref{eqn:mcf}. Consequently, for arbitrary input data, the set $\mc F$ cannot be represented as the feasible set of a linear program, a second-order conic program, or any other finite-dimensional convex optimization problem whose constraints define a closed convex set in $(\bm Y,\bm y)$. Therefore, one should not expect to verify whether the set $\mc F$ is non-empty, or whether local linear feasibility holds at a given point $\bzz$, by solving a single convex optimization problem.
	
	To address the above issue, a natural starting point is to consider the relaxation given by $R_i(\bm Y,\bm y)\geq b_i$ for all $i\in[m]$. Every point in $\mc F$ satisfies these inequalities. However, a point satisfying these inequalities need not belong to $\mc F$, because whenever $R_i(\bm Y,\bm y)=b_i$, the additional inequalities $\bm W_i\bm Y\bm v_{\ell}+\bm H_i\bm v_{\ell}\geq 0$ must hold for all $\ell\in[L]$. As shown in this relaxation, the central difficulty is that we do not know a priori which indices satisfy $R_i(\bm Y,\bm y)>b_i$ and which satisfy $R_i(\bm Y,\bm y)=b_i$ at points in $\mc F$. Checking all possible combinations of equalities and strict inequalities would require considering $2^m$ cases. To avoid such enumeration, we first examine how the relaxation can incorporate prior knowledge that some constraints hold with equality at every point in $\mc F$.
	
	Suppose we know an index set $\mc E\subseteq\{1,\ldots,m\}$ such that $R_i(\bm Y,\bm y)=b_i$ for all $(\bm Y,\bm y)\in\mc F$ and all $i\in\mc E$. For each $i\in\mc E$, the alternative $R_i(\bm Y,\bm y)>b_i$ in \eqref{eqn:mcf} is ruled out. Therefore, we can incorporate the known equalities and the corresponding inequalities along the incident edge directions into the relaxation. For indices outside $\mc E$, we retain only $R_i(\bm Y,\bm y)\geq b_i$, since no equality is assumed in advance. The following definition introduces the validity of $\mc E$ and the associated polyhedron $\mc G(\mc E)$.
	
	\begin{definition}\label{def:valid_approximation}
		Let $\mc E\subseteq\{1,\ldots,m\}$ be an index set. We call $\mc E$ \textbf{valid} if $R_i(\bm Y,\bm y)=b_i$ for all $(\bm Y,\bm y)\in\mc F$ and all $i\in\mc E$. For any index set $\mc E$, define the closed convex polyhedron
		\begin{equation*}
			\mathcal G(\mc E) \triangleq \left\{\left(\boldsymbol{Y}, \boldsymbol{y}\right): R_i(\bm Y, \bm y)=b_i, \;\; \bm W_i\bm Y\bm v_{\ell} + \bm H_i\bm v_{\ell} \geq 0\quad \forall \ell\in [L], \; i\in \mc E; \quad R_i(\bm Y, \bm y) \geq b_i \quad \forall i \notin \mc E\right\}.
		\end{equation*}
	\end{definition}
	
	Intuitively, we would like the index set $\mc E$ to be `valid' in the sense that it correctly identifies those indices for which $R_i(\bm Y, \bm y) = b_i$, and the closed convex polyhedron $\mathcal G(\mc E)$ is constructed as a surrogate for $\mc F$. The above definition raises two natural questions: how well does $\mc G(\mc E)$ approximate $\mc F$, and under what conditions does a point found in $\mc G(\mc E)$ also lie in $\mc F$ (thereby certifying the non-emptiness of $\mc F$)? The following proposition addresses both two questions.
	
	\begin{proposition}[Properties of the surrogate]\label{prop:sandwich} Let the index set $\mc E\subseteq\{1, \ldots, m\}$ be given.
		\begin{itemize}
			\item[(i)] If $\mc E$ is valid, then $\mc F\subset \mc G(\mc E)$.
			\item[(ii)] If a point $(\bm Y, \bm y)\in \mc G(\mc E)$ is with that $R_i(\bm Y, \bm y) > b_i$ for every $i\notin \mc E$, then $(\bm Y, \bm y)\in \mc F$.
		\end{itemize}
	\end{proposition}
	
	From \Cref{prop:sandwich}, we know that any surrogate $\mcge$ arising from a valid $\mc E$ contains $\mc F$. This result motivates the following approach: We initialize $\mc E = \emptyset$ and construct $\mc E$ incrementally at each iteration while maintaining validity. At each iteration, we search within $\mcge$ for a point $(\bm Y, \bm y)$ that achieves $R_i\Yy>b_i$ for every $i \notin \mc E$. There are three possible outcomes. First, the desired point $(\bm Y, \bm y) \in \mcge$ is found, then \Cref{prop:sandwich}$(ii)$ certifies $\Yy\in\mathcal{F} \neq \emptyset$. Second, the set $\mcge=\emptyset$, in which case \Cref{prop:sandwich}$(i)$ forces $\mathcal{F}=\emptyset$. Last, $\mcge$ contains points, but every point $\Yy$ in $\mcge$ admits $R_i\Yy = b_i$ for some $i \notin \mc E$. In such a case, we would identify at least one index $i \notin \mc E$ whose inclusion into $\mc E$ preserves validity, and we enlarge $\mc E$ accordingly in the next iteration.
	
	Our above discussion provides a framework for checking the locally linear feasibility of \eqref{eqn:linear_sys}. We adopt this framework and design \Cref{alg:check_llf} to check it. Specifically, in order to search for a point $\Yy\in \mcge$ with that $R_i(\bm Y, \bm y) > b_i$ for every $i\notin \mc E$, it suffices to check the feasibility of
	\begin{equation}\label{eqn:lp_feasibility}\tag*{LP(\ensuremath{\mc E})}
		\begin{array}{rcl}
			t^\star_{\mc E} =  &\displaystyle\max_{t, \; \Yy} & t\\
			& {\rm s.t.} & R_i(\bm Y, \bm y)=b_i, \;\; \bm W_i\bm Y\bm v_{\ell} + \bm H_i \bm v_{\ell} \geq 0\quad \forall \ell\in [L], \; i\in \mc E,\\
			&& R_i(\bm Y, \bm y) \geq b_i + t \quad \forall i \notin \mc E,\\
			&& \Yy \in \mb R^{r\times d} \times \mb R^r, \; t\in \mb R_+.
		\end{array}
	\end{equation}
	
	The decision variable $t$ in \ref{eqn:lp_feasibility} measures how much all the quantities $R_i(\bm Y,\bm y)$, $i\notin\mc E$, can exceed their respective bounds $b_i$ simultaneously. A feasible solution $(t, \bm Y, \bm y)$ with $t>0$ therefore gives a point satisfying $R_i(\bm Y,\bm y)>b_i$ for every $i\notin\mc E$. Moreover, the solution status of \ref{eqn:lp_feasibility} would precisely correspond to the three possible outcomes when we search for the desired point in \Cref{prop:sandwich}$(ii)$.
	
	If $t^\star_{\mc E}>0$, then $\mc F$ is non-empty by \Cref{prop:sandwich}$(ii)$, and system~\eqref{eqn:linear_sys} is locally linear feasible at $\bzz$. The same conclusion holds if the objective of \ref{eqn:lp_feasibility} is unbounded above, since a feasible solution with $t>0$ still exists. We include this possibility by allowing $t^\star_{\mc E}=+\infty$ in the condition $t^\star_{\mc E}>0$. If \ref{eqn:lp_feasibility} is infeasible, then $\mc G(\mc E)$ is empty: any point in $\mc G(\mc E)$ would give a feasible solution by setting $t=0$ and lead to a contradiction with \ref{eqn:lp_feasibility} is infeasible. \Cref{prop:sandwich}$(i)$ then implies $\mc F=\emptyset$.
	
	It remains to consider the case in which \ref{eqn:lp_feasibility} is feasible and $t^\star_{\mc E}=0$. In this case, $\mc G(\mc E)$ is non-empty, but no point in $\mc G(\mc E)$ satisfies $R_i(\bm Y,\bm y)>b_i$ for every $i\notin\mc E$ simultaneously. The value $t^\star_{\mc E}=0$ does not determine whether $\mc F$ is empty, because a point in $\mc F$ may satisfy $R_i(\bm Y,\bm y)=b_i$ on $i\notin \mc E$ provided that the inequalities along the incident edge directions hold. Therefore, we examine the indices outside $\mc E$ individually to determine which equalities must hold at all point in $\mc G(\mc E)$.
	
	For a fixed index $i\notin\mc E$, we ask whether there is any point in $\mc G(\mc E)$ at which $R_i(\bm Y,\bm y)>b_i$. This question can be answered by maximizing $R_i(\bm Y,\bm y)$ over the current polyhedron:
	\begin{equation}\label{eqn:sub_lp_feasibility}\tag*{Sub-LP(\ensuremath{\mc E, i})}
		\sigma^\star_{\mc E, i}
		= \displaystyle\max_{\Yy}
		\left\{R_i\Yy ~:~ \Yy\in\mcge\right\} \quad \forall i\notin \mc E.
	\end{equation}
	Every point in $\mc G(\mc E)$, by its definition, satisfies $R_i(\bm Y,\bm y)\geq b_i$ when $i\notin \mc E$. If $\sigma^\star_{\mc E,i}=b_i$, even the largest possible value of $R_i(\bm Y,\bm y)$ equals $b_i$, so $R_i(\bm Y,\bm y)=b_i$ at every point in $\mc G(\mc E)$. If $\sigma^\star_{\mc E,i}>b_i$, including an unbounded objective, then some point in $\mc G(\mc E)$ satisfies $R_i(\bm Y,\bm y)>b_i$. We collect all indices for which $R_i(\bm Y,\bm y)=b_i$ holds at every point in $\mc G(\mc E)$ in the set $\mc N=\{i\notin\mc E:\sigma^\star_{\mc E,i}=b_i\}$.
	For every $i\in\mc N$, the equality $R_i(\bm Y,\bm y)=b_i$ holds throughout $\mc G(\mc E)$ and hence throughout $\mc F$ by \Cref{prop:sandwich}$(i)$. Therefore, we can add all indices in $\mc N$ to $\mc E$ while the validity of the enlarged index set $\mc E$ is preserved by \Cref{def:valid_approximation}. In the next iteration, the enlarged (and still valid) index set $\mc E$ yields a smaller superset $\mc G(\mc E)$ of $\mc F$, thereby providing a tighter surrogate.
	
	\Cref{alg:check_llf} summarizes the procedure. The algorithm starts with $\mc E=\emptyset$, which is valid by definition. Each iteration either determines whether $\mc F$ is empty or adds at least one index to $\mc E$. If $\mc E=\{1,\ldots,m\}$, every constraint has been identified as an equality on $\mc F$, and all the corresponding inequalities along the incident edge directions have been imposed. Consequently, $\mc G(\mc E)=\mc F$, and the final feasibility check determines whether $\mc F$ is empty. The following theorem establishes correctness, termination, and polynomial bounds on the number and size of these linear programs.

	\begin{algorithm}[t]
		\caption{Check locally linear feasibility of \eqref{eqn:linear_sys} at $\bzz$ of $\Xi$}
		\label{alg:check_llf}
		\begin{algorithmic}[1]
			\State {\bfseries Input and Initialization:} Matrices $\bm W, \bm H$ and vectors $\bm b$ in \eqref{eqn:linear_sys}. Vertex $\bzz$ of $\Xi$. Directions $\{\bm v_1, \ldots, \bm v_L\}$ of all edges that are incident to the vertex $\bzz$ on $\Xi$. Initialization $\mc E = \emptyset$. \label{line:input}
			\State {\bfseries Output:} Whether system~\eqref{eqn:linear_sys} is locally linear feasible at $\bzz$.
			\While{$|\mc E| < m$}
			\State Solve \ref{eqn:lp_feasibility} and get the feasibility status of this linear program.
			\If{\ref{eqn:lp_feasibility} is infeasible}
			\State \Return  ``not locally linear feasible at $\bzz$.'' \label{line:return_empty}
			\ElsIf{\ref{eqn:lp_feasibility} is feasible with $t^\star_{\mc E} > 0$}
			\State \Return  ``locally linear feasible at $\bzz$.''\label{line:return_non_empty}
			\Else
			\State Solve \ref{eqn:sub_lp_feasibility} for all $i\notin \mc E$ and collect the new tight set $\mc N = \{i\notin \mc E ~:~ \sigma^\star_{\mc E, i} = b_i\}$. \label{line:find_new_tight}
			\State Update the index set $\mc E \gets \mc E \cup \mc N$. \label{line:update_mce}
			\EndIf
			\EndWhile
			\State Solve \ref{eqn:lp_feasibility} and get the feasibility status of this linear program. \label{line:LPT}
			\If{\ref{eqn:lp_feasibility} is infeasible}
			\State \Return  ``not locally linear feasible at $\bzz$.''\label{line:return_cannot_find}
			\Else
			\State \Return  ``locally linear feasible at $\bzz$.''\label{line:return_can_find}
			\EndIf
		\end{algorithmic}
	\end{algorithm}
	\begin{theorem}\label{thm:alg_solve_llf}
		\Cref{alg:check_llf} has the following properties.
		\begin{itemize}
			\item[(i)] (Termination) The while loop in \Cref{alg:check_llf} terminates in at most $m$ iterations, and the algorithm halts at one of Lines \ref{line:return_empty} and \ref{line:return_non_empty} in the loop, or at one of Lines \ref{line:return_cannot_find} and \ref{line:return_can_find} after the loop.
			
			\item[(ii)] (LP count) The algorithm solves at most $(m+2)(m+1)/2$ linear programs, each with at most $rd+r+1$ decision variables and $mL+m+1$ constraints, hence of size polynomial in $(m, r, d, L)$.
			
			\item[(iii)] (Correctness) The algorithm correctly determines local linear feasibility: it returns ``locally linearly feasible at $\bzz$'' if and only if \eqref{eqn:linear_sys} is locally linearly feasible at $\bzz$.
		\end{itemize}
	\end{theorem}


	Theorem~\ref{thm:alg_solve_llf} shows that locally linear feasibility at a vertex $\bzz$ can be verified by our proposed \Cref{alg:check_llf} in polynomial time, when the directions of the edges incident to the vertex $\bzz$ are available. In addition to the dimensions $(m, r, d)$ in~\eqref{eqn:linear_sys}, the complexity depends on the number $L$ of edges incident to $\bzz$. For the two-sided and one-sided budgeted uncertainty sets \eqref{eqn:two_sided_budget} and \eqref{eqn:one_sided_budget} studied in \Cref{sec:budgeted_uncertainty_set}, the number of edges incident to each vertex is bounded above by a polynomial in $d$. Therefore, using the explicit edge directions given in Propositions~\ref{prop:two_sides_budget_uncertainty_set_integer}--\ref{prop:one_sides_budget_uncertainty_set_float}, \Cref{alg:check_llf} verifies local linear feasibility at any vertex of these uncertainty sets in time polynomial in $(m, r, d)$.
	
	Though \Cref{alg:check_llf} efficiently verifies local linear feasibility at a single vertex, verifying assumption~A4b requires further consideration, as local linear feasibility must hold at \textit{every} vertex of $\Xi$ for the same first-stage decision. A direct approach is to apply \Cref{alg:check_llf} separately at each vertex of $\Xi$. Nevertheless, this approach would incur a total cost that scales linearly with the number of vertices of $\Xi$, which itself may grow exponentially in the dimension $d$. This phenomenon is a manifestation of the well-known combinatorial explosion of vertices in high-dimensional polyhedra. Therefore, while vertex-by-vertex application of \Cref{alg:check_llf} provides a conceptually straightforward, and to our knowledge the first, route to verifying assumption~A4b, it may become computationally prohibitive when the dimension $d$ of the uncertainty is large, reflecting the classical curse of dimensionality. The numerical experiments in the next section report both the per-vertex and the total verification time, thereby illustrating the computational implications of this procedure.
	
	\section{Numerical Experiments}\label{sec:casestudy}
	In this section, we investigate two two-stage applications to illustrate the characterization established in \Cref{sec:geometric_character} and examine the effectiveness and computational efficiency of \Cref{alg:check_llf}. We consider a multi-item newsvendor problem with a two-sided budgeted uncertainty set and a network lot-sizing problem with a one-sided budgeted uncertainty set. For each problem, we use the geometric results in \Cref{sec:budgeted_uncertainty_set} to identify when global feasibility guarantees local linear feasibility and apply \Cref{alg:check_llf} to check the local linear feasibility and assumption~A4b. The numerical experiments report both the verification time per vertex and the total verification time, and examine the convergence of the multi-policy approximation as the sample size increases. All experiments are run on a $2.3$ GHz $4$-Core Intel Core i$7$ CPU with $32$ GB $3733$ MHz DDR$4$ main memory. All linear programs in \Cref{alg:check_llf} are solved using the state-of-the-art solver MOSEK \citep{mosek}.
	
	\subsection{Multi-Item Newsvendor Problem on a Two-Sided Budgeted Uncertainty Set}\label{subsec:newsvendor_two_side_budget}
	We consider a multi-item newsvendor problem studied in \cite{ardestani2016robust}. A retailer stocks $n\geq 3$ items for a single selling season. Before the demand of each item is observed, the retailer orders $x_j\geq 0$ units of item $j\in [n]$ at unit cost $c_j$ with the maximum order quantity $\bar x_j$. After the ordering, the demand $w_j$ is realized and observed for item $j$. The retailer then sells $u_j$ units at unit price $v_j$, salvages leftover $s_j = x_j - u_j$ at unit price $g_j$, and pays a unit stock-out penalty $p_j$ on the unmet demand $z_j = w_j - u_j$ for item $j$. We follow the standard assumption $g_j < c_j < v_j$ of the newsvendor literature, so salvaging is never profitable, and every unit that can be sold is sold. 
	
	Following \cite{ardestani2016robust}, we assume that the uncertain demand $\tilde w_j$ for item $j$ is given by $\tilde w_j = \bar w_j + \hat w_j \tilde \xi_j$, where $\bar w_j$ is the point forecast of $\tilde w_j$, $\hat w_j \leq \bar w_j$ is its forecast scale, and the uncertain $\bm \xi$ is in the two-sided budgeted uncertainty set defined in \eqref{eqn:two_sided_budget} with budget $\Gamma\in (0, n)$. The uncertainty dimension is therefore $d=n$. We now adopt a distributionally robust approach that treats $\tilde{\bm \xi}$ as a random variable governed by an unknown distribution in the type-$\infty$ Wasserstein ambiguity set $\mathcal W \subset \mathcal P(\Xi)$. The ambiguity set $\mathcal W$ is centered at the empirical distribution $\hat{\mb P}^N = \sum_{i \in [N]} {\rm Dirac}(\bm \xi^i)/N$ supported on $N$ given samples $\{\bm \xi^1, \ldots, \bm \xi^N\}$ of $\tilde{\bm \xi}$, with radius $\epsilon_N > 0$. Writing $\bm y = [\bm u^\top \bm s^\top \bm z^\top]^\top \in \mb R^{3n}$, then the two-stage multi-item newsvendor problem is given by
	\begin{equation}\label{eqn:dro_lp}
		\min_{\bm 0 \leq \bm{x} \leq \bar{\bm x}} \bm{c}^{\top} \bm{x} + \sup_{\mb P\in \mc W}\mathbb{E}_{\tilde{\bm \xi}\sim \mb P} \left[Q(\bm{x}, \tilde{\bm \xi})\right],
	\end{equation}
	where the second-stage problem is given by
	\begin{equation*}
		Q(\bm x, \bm \xi) = \left[\begin{array}{cl}
			\displaystyle\min_{\bm u, \bm s, \bm z} & -
			\bm v^\top \bm u - \bm g^\top \bm s + \bm p^\top \bm z\\
			{\rm s.t.}& u_j + s_j \leq x_j, \;\; u_j + z_j \geq \bar w_j + \hat w_j \xi_j, \;\; u_j \leq \bar w_j + \hat w_j \xi_j \quad \forall j\in [n]\\
			&\bm u\geq \bm 0, \bm s\geq \bm 0, \bm z\geq \bm 0
		\end{array}\right],
	\end{equation*}
	By Proposition~$3$ in \cite{bertsimas2023data}, the above multi-item newsvendor problem can be equivalently reformulated as the sample robust optimization problem \eqref{eqn:sample_robust} with
	\begin{equation*}
		\bm q= \begin{bmatrix}
			-\bm v\\
			-\bm g\\
			\bm p
		\end{bmatrix}, \; \bm T = \begin{bmatrix}
			\bm I_n\\
			\bm O_{5n\times n}\\
			\bm I_n\\
			-\bm I_n
		\end{bmatrix}, \; \bm W = \begin{bmatrix}
			\begin{bmatrix}
				-\bm I_n & -\bm I_n & \bm O_n\\
				\bm I_n & \bm O_n & \bm I_n\\
				-\bm I_n & \bm O_n & \bm O_n
			\end{bmatrix}\\
			\bm I_{3n}\\
			\bm O_{2n\times 3n}
		\end{bmatrix}, \; \bm h(\bm \xi) = -\begin{bmatrix}
			\bm O_n\\
			-{\rm Diag}(\hat{\bm w})\\
			{\rm Diag}(\hat{\bm w})\\
			\bm O_{5n\times n}
		\end{bmatrix}\bm \xi + \begin{bmatrix}
			\bm 0_n\\
			\bar{\bm w}\\
			-\bar{\bm w}\\
			\bm 0_{4n}\\
			-\bar{\bm x}
		\end{bmatrix},
	\end{equation*}
	and $m=8n$, $r=3n$, and $d=n$. Notably, assumption~A4b holds regardless of whether $\Xi$ is simple. Indeed, in the second-stage problem $Q(\bm x, \bm \xi)$, one can always construct the feasible affine decision
	\begin{equation*}
		\hat{\bm y} = \begin{bmatrix}
			\hat{\bm u}\\
			\hat{\bm s}\\
			\hat{\bm z}
		\end{bmatrix} = \begin{bmatrix}
			\bm 0\\
			\bm 0\\
			\bar{\bm w} + {\rm Diag}(\hat{\bm w}) \bm \xi
		\end{bmatrix}.
	\end{equation*}
	
	In the following, we conduct numerical experiments to validate the efficiency of \Cref{alg:check_llf} and verify the asymptotic convergence of two-stage sample robust optimization. Specifically, we consider the multi-item newsvendor problem with $n\in\{3,4,5,6,8,10\}$ items, and set the budget in~\eqref{eqn:two_sided_budget} to $\Gamma /n \in\{0.5,0.65,0.8,0.95\}$. The remaining data $(\bm c, \bar{\bm x}, \bm v, \bm g, \bm p, \bar{\bm w}, \hat{\bm w})$ of the newsvendor problem are generated at random, independently across items, from $\bar{\bm x} = 10\bm e$ and
	\begin{equation*}
		\begin{array}{cl}
			v_j\sim {\rm Unif}([10,20]),\quad
			c_j / v_j\sim {\rm Unif}([0.4,0.7]),\quad
			g_j / c_j\sim {\rm Unif}([0.1,0.9])\quad & \forall j\in [n],\\
			p_j / v_j\sim {\rm Unif}([0.2,0.8]), \quad \bar w_j\sim {\rm Unif}([4,6]),\quad
			\hat w_j / \bar w_j \sim {\rm Unif}([0.3,0.8]) & \forall j\in [n],
		\end{array}
	\end{equation*}
	which enforces $g_j<c_j<v_j$ and $0\le\hat w_j\le\bar w_j$, so that the realized demand $\bm w=\bar{\bm w} + \hat{\bm w}\circ \bm \xi$ is nonnegative on $\Xi$ and the salvage value never exceeds the ordering cost. The true distribution of $\tilde{\bm \xi}$, which is unknown to the optimizer, is taken to be a truncated Gaussian distribution on $\Xi$.
	
	\begin{table}[htbp]
		\centering
		\caption{Numerical results and runtime for verifying A4b on the two-sided budget uncertainty set. Here, a tick (\cmark) indicates that $\Xi$ is simple or that A4b holds; a cross (\xmark) indicates the opposite.}
		\label{tab:newsvendor_a4b}
		\begin{tabular}{r @{\hspace{1.5em}} r @{\hspace{1.5em}} r @{\hspace{1.5em}} r @{\hspace{1.5em}} c @{\hspace{1.5em}} r @{\hspace{1.5em}} r @{\hspace{1.5em}} c @{\hspace{1.5em}} r @{\hspace{1.5em}} r}
			\toprule
			&&&&&&&& \multicolumn{2}{c}{runtime} \\
			\cmidrule(l){9-10}
			$n$ & $\Gamma$ & $\Gamma/n$ & $\underline{\Gamma}$ & simple & \#edges & \#vertices & A4b & per vertex (ms) & total (s)\\
			\midrule
			3  & 1.50 & 0.50 & 1 & \cmark & 3  & 24        & \cmark & 0.5451 & 0.0131 \\
			3  & 1.95 & 0.65 & 1 & \cmark & 3  & 24        & \cmark & 0.5303 & 0.0127 \\
			3  & 2.40 & 0.80 & 2 & \cmark & 3  & 24        & \cmark & 0.5407 & 0.0130 \\
			3  & 2.85 & 0.95 & 2 & \cmark & 3  & 24        & \cmark & 0.5279 & 0.0127 \\
			\midrule
			4  & 2.00 & 0.50 & 2 & \xmark & 8  & 24        & \cmark & 0.5958 & 0.0143 \\
			4  & 2.60 & 0.65 & 2 & \cmark & 4  & 96        & \cmark & 0.4107 & 0.0394 \\
			4  & 3.20 & 0.80 & 3 & \cmark & 4  & 64        & \cmark & 0.4417 & 0.0283 \\
			4  & 3.80 & 0.95 & 3 & \cmark & 4  & 64        & \cmark & 0.4297 & 0.0275 \\
			\midrule
			5  & 2.50 & 0.50 & 2 & \xmark & 6  & 240       & \cmark & 0.4236 & 0.1017 \\
			5  & 3.25 & 0.65 & 3 & \cmark & 5  & 320       & \cmark & 0.4177 & 0.1337 \\
			5  & 4.00 & 0.80 & 4 & \xmark & 8  & 80        & \cmark & 0.5174 & 0.0414 \\
			5  & 4.75 & 0.95 & 4 & \cmark & 5  & 160       & \cmark & 0.4776 & 0.0764 \\
			\midrule
			6  & 3.00 & 0.50 & 3 & \xmark & 18 & 160       & \cmark & 0.6240 & 0.0998 \\
			6  & 3.90 & 0.65 & 3 & \xmark & 7  & 960       & \cmark & 0.4932 & 0.4735 \\
			6  & 4.80 & 0.80 & 4 & \cmark & 6  & 960       & \cmark & 0.4813 & 0.4620 \\
			6  & 5.70 & 0.95 & 5 & \cmark & 6  & 384       & \cmark & 0.4864 & 0.1868 \\
			\midrule
			8  & 4.00 & 0.50 & 4 & \xmark & 32 & 1{,}120   & \cmark & 1.2557 & 1.4064 \\
			8  & 5.20 & 0.65 & 5 & \xmark & 9  & 10{,}752  & \cmark & 0.7158 & 7.6966 \\
			8  & 6.40 & 0.80 & 6 & \cmark & 8  & 7{,}168   & \cmark & 0.7353 & 5.2708 \\
			8  & 7.60 & 0.95 & 7 & \cmark & 8  & 2{,}048   & \cmark & 0.7449 & 1.5255 \\
			\midrule
			10 & 5.00 & 0.50 & 5 & \xmark & 50 & 8{,}064   & \cmark & 2.0936 & 16.8828 \\
			10 & 6.50 & 0.65 & 6 & \xmark & 12 & 107{,}520 & \cmark & 1.1497 & 123.6116 \\
			10 & 8.00 & 0.80 & 8 & \xmark & 32 & 11{,}520  & \cmark & 1.8318 & 21.1025 \\
			10 & 9.50 & 0.95 & 9 & \cmark & 10 & 10{,}240  & \cmark & 1.0839 & 11.0991 \\
			\bottomrule
		\end{tabular}
	\end{table}
	
	We evaluate assumption~A4b at the interior point $\hat{\bm x}=5\bm e$, for which $\bm b=\bm h^{0}-\bm T \hat{\bm x}$. The results and runtimes are reported in \Cref{tab:newsvendor_a4b} and \Cref{fig:a4b_timing}. Specifically, our \Cref{alg:check_llf} correctly certifies whether assumption~A4b holds across all instances. Moreover, the verification cost decomposes cleanly into a per-vertex factor and a combinatorial factor. The per-vertex time remains within a few milliseconds across all instances, and is independent of $\Gamma$ except through the edge count $L$. Indeed, \Cref{thm:alg_solve_llf} shows that, at fixed $n$ (and hence fixed $m,r,d$ for the multi-item newsvendor problem), the linear programs solved in \Cref{alg:check_llf} have dimensions that depend on $\Gamma$ only via the number of edges $L$. This is visible at $n=10$, where increasing $\Gamma$ does not necessarily increase the per-vertex runtime; rather, the per-vertex runtime rises with $L$. The total cost, by contrast, is governed entirely by the number of vertices, a combinatorial quantity given by \Cref{prop:two_sides_budget_uncertainty_set_integer}$(i)$ and \Cref{prop:two_sides_budget_uncertainty_set_float}$(i)$. For instance, the instances $n=10,\Gamma=0.65 n = 6.5$ and $n=10,\Gamma = 0.95 n = 9.5$ exhibit similar per-vertex runtimes (since their edge counts are comparable), yet their total times differ substantially: the former has $107{,}520$ vertices, far exceeding the latter. Therefore, the dominant computational burden of verifying A4b comes from the number of vertices of $\Xi$. Consequently, modelling the uncertainty set as a simple polytope, or more generally as one with fewer non-simple vertices, yields substantial computational savings when verifying A4b via \Cref{alg:check_llf}.

	\begin{figure}[t]
		\centering
		\includegraphics[width=1.0\textwidth]{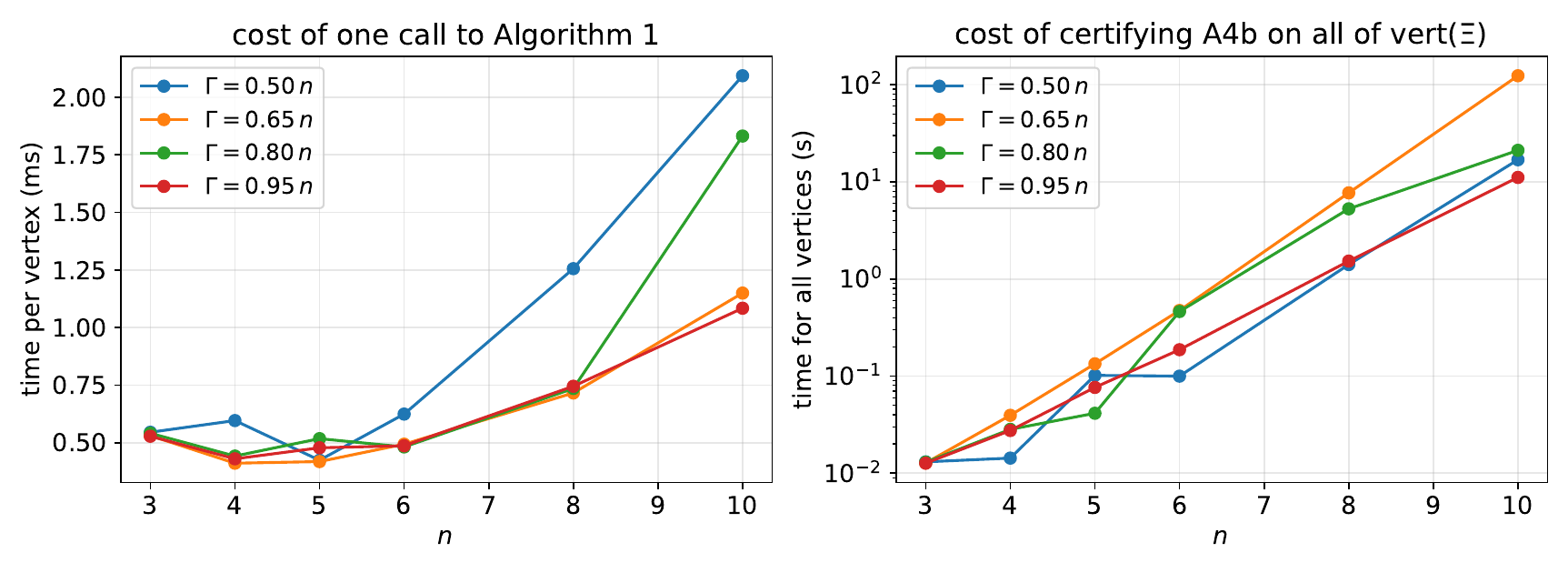}
		\caption{Runtime for verifying A4b on the problem instances with two-sided budget uncertainty sets.}
		\label{fig:a4b_timing}
	\end{figure}
	
	\begin{figure}[t]
		\centering
		\includegraphics[width=1.0\textwidth]{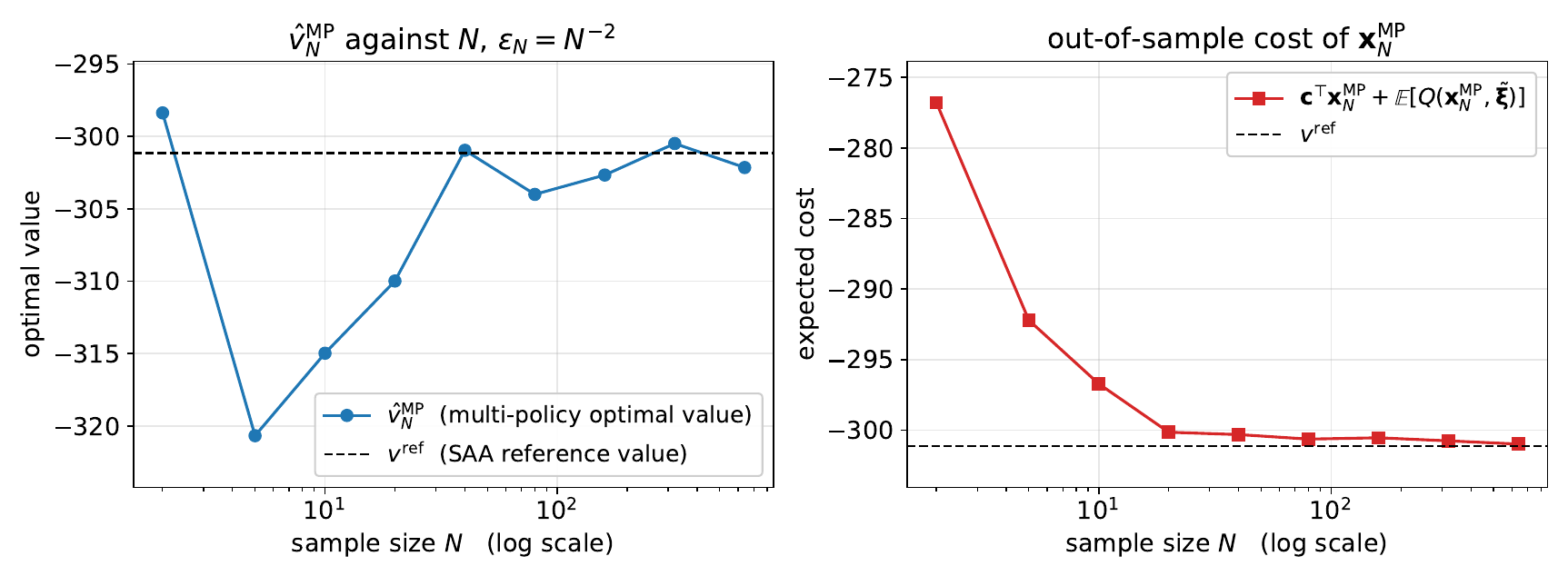}
		\caption{Optimal values of multi-policy approximation~\eqref{eqn:multi_policy} and out-of-sample costs of $\bm x^{\mathrm{MP}}_N$ with varying sample size $N$.}
		\label{fig:newsvendor_convergence}
	\end{figure}
	\begin{table}[t]
		\centering
		\caption{First-stage order quantities $\bm{x}_{N}^{\mathrm{MP}}$ obtained by solving~\eqref{eqn:multi_policy} for varying sample size $N$ with the reference solution $\bm x^{\mathrm{ref}}$.}
		\label{tab:newsvendor_convergence}
		\setlength{\tabcolsep}{2.4pt}
		\begin{tabular}{{c}r *{9}{c} c}
			\toprule
			& item 1 & item 2 & item 3 & item 4 & item 5 & item 6 & item 7 & item 8 & item 9 & item 10 & $\|\bm{x}_{N}^{\mathrm{MP}} - \bm x^{\mathrm{ref}}\|_{\infty}$ \\
			\midrule
			$N=2$            & 5.995 & 6.057 & 4.244 & 5.978 & 4.667 & 5.922 & 5.983 & 5.761 & 4.430 & 5.262 & 2.1583 \\
			$N=5$            & 6.237 & 5.637 & 5.625 & 6.008 & 5.916 & 6.448 & 5.500 & 6.468 & 4.269 & 7.097 & 1.4511 \\
			$N=10$           & 5.944 & 5.379 & 5.858 & 5.802 & 5.812 & 5.804 & 5.601 & 7.596 & 4.942 & 7.113 & 0.7592 \\
			$N=20$           & 5.952 & 5.372 & 5.870 & 5.606 & 5.710 & 6.394 & 5.613 & 7.599 & 5.102 & 6.696 & 0.3952 \\
			$N=40$           & 5.409 & 5.377 & 5.850 & 5.851 & 5.710 & 6.422 & 6.004 & 7.831 & 5.005 & 6.459 & 0.4152 \\
			$N=80$           & 5.715 & 5.599 & 6.170 & 5.685 & 5.640 & 6.489 & 6.304 & 7.892 & 5.070 & 6.503 & 0.3078 \\
			$N=160$          & 5.411 & 5.507 & 6.168 & 5.673 & 5.593 & 6.422 & 6.078 & 7.892 & 5.070 & 6.519 & 0.3025 \\
			$N=320$          & 5.582 & 5.424 & 5.984 & 5.588 & 5.680 & 6.422 & 6.108 & 7.730 & 5.081 & 6.438 & 0.2810 \\
			$N=640$          & 5.659 & 5.354 & 6.204 & 5.612 & 5.650 & 6.521 & 6.034 & 7.850 & 5.070 & 6.381 & 0.0696 \\
			\midrule
			$\bm x^{\mathrm{ref}}$ & 5.713 & 5.305 & 6.265 & 5.542 & 5.601 & 6.563 & 5.996 & 7.919 & 5.116 & 6.448 & --     \\
			\bottomrule
		\end{tabular}
	\end{table}
	
	To test the asymptotic optimality of two-stage sample robust optimization, we solve the multi-policy approximation~\eqref{eqn:multi_policy} for the multi-item newsvendor problem with $n=10$ and $\Gamma=8.5$, varying the sample size $N$. We set $\epsilon_N=N^{-2}$, so that assumption~A1 of \Cref{thm:asympt_opt} is satisfied. In this case, the budgeted uncertainty set $\Xi$ is simple by \Cref{coro_simple_two_sided_budget}, and the other assumptions in \Cref{thm:asympt_opt} to ensure asymptotic optimality are satisfied due to \Cref{coro:simple_poly}. As a benchmark, we also solve the stochastic problem~\eqref{eqn:two_stage} by sample average approximation (SAA) with $500{,}000$ samples, and treat the resulting value $v^{\mathrm{ref}}$ as the true optimum. \Cref{tab:newsvendor_convergence} reports the decisions $\bm{x}_{N}^{\mathrm{MP}}$ obtained from solving~\eqref{eqn:multi_policy} together with the reference decision $\bm x^{\mathrm{ref}}$ from SAA. As shown in it, the first-stage decisions $\bm{x}_{N}^{\mathrm{MP}}$ converge to the reference $\bm x^{\mathrm{ref}}$ as $N$ grows. \Cref{fig:newsvendor_convergence} then displays the optimal values of~\eqref{eqn:multi_policy} and the out-of-sample costs incurred by deploying $\bm{x}_{N}^{\mathrm{MP}}$ and illustrates that both of them converge to $v^{\mathrm{ref}}$ within a modest sample size of $N=640$. Together, these results demonstrate both the asymptotic optimality and the sample efficiency of the two-stage sample robust optimization.

	\subsection{Network Lot-Sizing on a One-Sided Budgeted Uncertainty Set}\label{subsec:lotsizing_one_side_budget}
	We consider a two-stage lot-sizing problem \citep{bertsimas2016duality, zhen2018adjustable, bertsimas2022two}. A firm operates $n$ stores and pre-positions $x_j\in[0,\bar x_j]$ units of stock at store $j\in[n]$ at unit cost $c_j$, where $\bar x_j$ denotes the storage capacity. After the decision $\bm x$ is made, a demand of $\hat D\,\xi_j$ is realized in the neighbourhood of store $j\in[n]$, where $\hat D>0$ is the maximal demand at any single store and $\xi_j\in[0,1]$ is the corresponding normalized demand. The firm then transports $y_{ij}\in[0,\bar y_{ij}]$ units from store $i$ to store $j$ at unit cost $q_{ij}$ to ensure that, after transshipment, the stock at each store covers its realized demand. Following \cite{bertsimas2016duality} and \cite{zhen2018adjustable}, we take the support of $\bm\xi$ to be the one-sided budgeted uncertainty set in \eqref{eqn:one_sided_budget} with prescribed budget $\Gamma\in(0,n)$. We consider the two-stage distributionally robust network lot-sizing problem with the same formulation~\eqref{eqn:dro_lp},
	where the second-stage problem is given by
	\begin{equation*}
		Q(\bm x, \bm \xi) = \left[\begin{array}{cll}
			\displaystyle\min_{\bm y} & \displaystyle\sum_{i \neq j} q_{i j} y_{i j}&\\
			{\rm s.t.}& x_i - \displaystyle\sum_{j\neq i}y_{ij} + \displaystyle\sum_{j\neq i}y_{ji} \geq \hat D\xi_i \quad & \forall i\in [n]\\
			& 0\leq y_{ij} \leq \bar y_{ij} \quad & \forall i\in [n], \; j\in [n], \; i\neq j
		\end{array}\right].
	\end{equation*}
	
	As implied by \cite{bertsimas2023data}, and analogously to \Cref{subsec:newsvendor_two_side_budget}, the above distributionally robust problem is equivalent to the sample robust problem \eqref{eqn:sample_robust} and approximated by \eqref{eqn:multi_policy} with
	\begin{equation*}
		\bm T = \begin{bmatrix}
			\bm I_n\\
			\bm O_{2(n-1)n\times n}\\
			\bm I_n\\
			-\bm I_n
		\end{bmatrix}, \;\; \bm W = \begin{bmatrix}
			\bar{\bm W}\\
			\bm I_{(n-1)n}\\
			-\bm I_{(n-1)n}\\
			\bm O_{2n\times (n-1)n}
		\end{bmatrix}, \;\; \bm h(\bm \xi) = -\begin{bmatrix}
			-\hat D \bm I_n\\
			\bm O_{2n^2 \times n}
		\end{bmatrix}\bm \xi + \begin{bmatrix}
			\bm 0_{n^2}\\
			-\bar{\bm y}\\
			\bm 0_n\\
			-\bar{\bm x}
		\end{bmatrix},
	\end{equation*}
	where
	\begin{equation*}
		\bar{\bm W} = \begin{bmatrix}
			-\bm e^\top \; & \bm e_1^\top \; & \cdots \; & \bm e_1^\top \; & \bm e_1^\top\\
			\bm e_1^\top \; & -\bm e^\top \; & \bm e_2^\top \; & \cdots \; & \bm e_2^\top\\
			\vdots \; & \ddots \; & \ddots \; & \ddots \; & \vdots\\
			\bm e_{n-2}^\top \; & \bm e_{n-2}^\top\; & \bm e_{n-2}^\top\; & -\bm e^\top\; & \bm e_{n-1}^\top \\
			\bm e_{n-1}^\top \; & \bm e_{n-1}^\top\; & \cdots\; & \bm e_{n-1}^\top\; & -\bm e^\top
		\end{bmatrix},
	\end{equation*}
	and problem size $m=2n^2+n$, $r=n(n-1)$ and $d=n$. Here, both the all-ones vector $\bm e$ and the $i$-th natural basis vector $\bm e_i$ are with length $n-1$, and we arrange the transportation variables as a vector $\bm y=(y_{12},\ldots,y_{1n},y_{21},y_{23},\ldots,y_{2n},\ldots,y_{n1},\ldots,y_{n,n-1})^\top$ in $\mb R^{n(n-1)}$.

	We conduct two sets of experiments to examine the computational cost of verifying A4b and the behavior of the multi-policy approximation when A4b holds or fails. For the verification experiment, we consider $n\in\{5,6,7,8,9,10\}$ and integer budgets $\Gamma\in\{2,3\}$. Store locations $\bm\ell_j$, sales volumes $v_j$, and unit storage costs $c_j$ are generated independently across stores according to
	\begin{equation*}
		\begin{array}{cl}
			\bm\ell_j\sim {\rm Unif}([0,10]^2),\quad
			v_j\sim {\rm Unif}([0.5,1.5]),\quad
			c_j\sim 20\cdot {\rm Unif}([0.9,1.1]) \quad 
			& \forall j\in[n],\\
			\hat D=20, \qquad q_{ij}=\|\bm\ell_i-\bm\ell_j\|_2\quad & \forall i\in [n], \; j\in [n], \; i\neq j.
		\end{array}
	\end{equation*}
	The $\Gamma$ stores with the largest sales volumes are designated as flagship stores, whose index set is denoted by $S_0$. Stock can be pre-positioned only at these stores. Let $R(i)$ contain the $n-2$ stores nearest to store $i$, excluding store $i$ itself. Transportation is allowed only from a flagship store $i$ to a store $j$ that is in $R(i)$ with a smaller sales volume $v_j < v_i$. Accordingly, we set
	\[
	\bar x_j=
	\begin{cases}
		\hat D,&j\in S_0,\\
		0,&j\notin S_0,
	\end{cases}
	\qquad
	\bar y_{ij}=
	\begin{cases}
		\hat D,&i\in S_0,\ j\in R(i),\ v_j<v_i,\\
		0,&\text{otherwise}.
	\end{cases}
	\]
	
	For each pair $(n,\Gamma)$, we generate an instance that satisfies A4a at $\bm x=\bar{\bm x}$, without imposing any restriction on whether A4b holds. Under the specified storage capacities, $\bar{\bm x}$ is the only first-stage decision that can satisfy A4a: global feasibility requires $\bm e^\top\bm x\geq\hat D\Gamma$, while $\bm x\leq\bar{\bm x}$ and $\bm e^\top\bar{\bm x}=\hat D\Gamma$. We therefore evaluate A4b at $\bar{\bm x}$ by checking locally linear feasibility at the vertices of $\Xi$. By \Cref{prop:one_sides_budget_uncertainty_set_integer}, the support has $\binom{n}{\Gamma}$ non-simple vertices, each with $L=\Gamma(n-\Gamma+1)$ incident edges. Global feasibility guarantees locally linear feasibility at all remaining vertices by \Cref{thm:no_more_d_edges}. Therefore, we apply \Cref{alg:check_llf} at every non-simple vertex to verify assumption~A4b.
	
	\begin{table}[t]
		\centering
		\caption{Verification results and runtime for verifying A4b on the one-sided budget uncertainty set. All retained instances satisfy A4a. Here, a tick (\cmark) indicates that A4b holds; a cross (\xmark) indicates the opposite.}
		\label{tab:network_ls_a4b}
		\begin{tabular}{rrrrcrr}
			\toprule
			&&&&& \multicolumn{2}{c}{runtime}\\
			\cmidrule(l){6-7}
			$n$ & $\Gamma$ & \# non-simple vertices & \# edges & A4b
			& per vertex (ms) & total (s)\\
			\midrule
			5  & 2 & 10  & 8  & \xmark & 22.7147  & 0.2271\\
			5  & 3 & 10  & 9  & \cmark & 27.3664  & 0.2737\\
			6  & 2 & 15  & 10 & \xmark & 40.1088  & 0.6016\\
			6  & 3 & 20  & 12 & \cmark & 41.8403  & 0.8368\\
			7  & 2 & 21  & 12 & \cmark & 74.7879  & 1.5705\\
			7  & 3 & 35  & 15 & \cmark & 79.9189  & 2.7972\\
			8  & 2 & 28  & 14 & \cmark & 142.1452 & 3.9801\\
			8  & 3 & 56  & 18 & \xmark & 147.8925 & 8.2820\\
			9  & 2 & 36  & 16 & \cmark & 253.1459 & 9.1133\\
			9  & 3 & 84  & 21 & \cmark & 265.9388 & 22.3389\\
			10 & 2 & 45  & 18 & \cmark & 463.9039 & 20.8757\\
			10 & 3 & 120 & 24 & \cmark & 483.6434 & 58.0372\\
			\bottomrule
		\end{tabular}
	\end{table}
	
	\begin{figure}[t]
		\centering
		\includegraphics[width=1.0\textwidth]{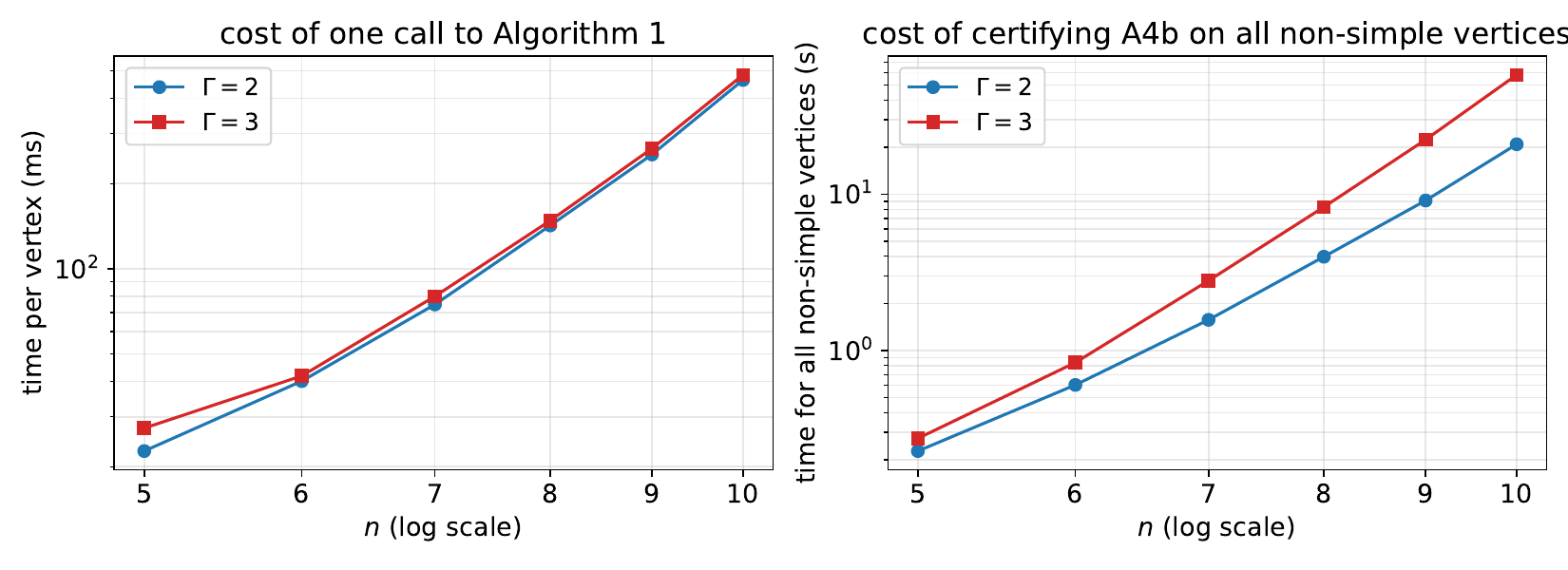}
		\caption{Average verification time per non-simple vertex (left) and average total verification time over all non-simple vertices (right) for the network lot-sizing problem. Both panels use logarithmic scales on both axes.}
		\label{fig:a4b_timing_one_sided}
	\end{figure}
	
	\begin{figure}[t]
		\centering
		\includegraphics[width=1.0\textwidth]{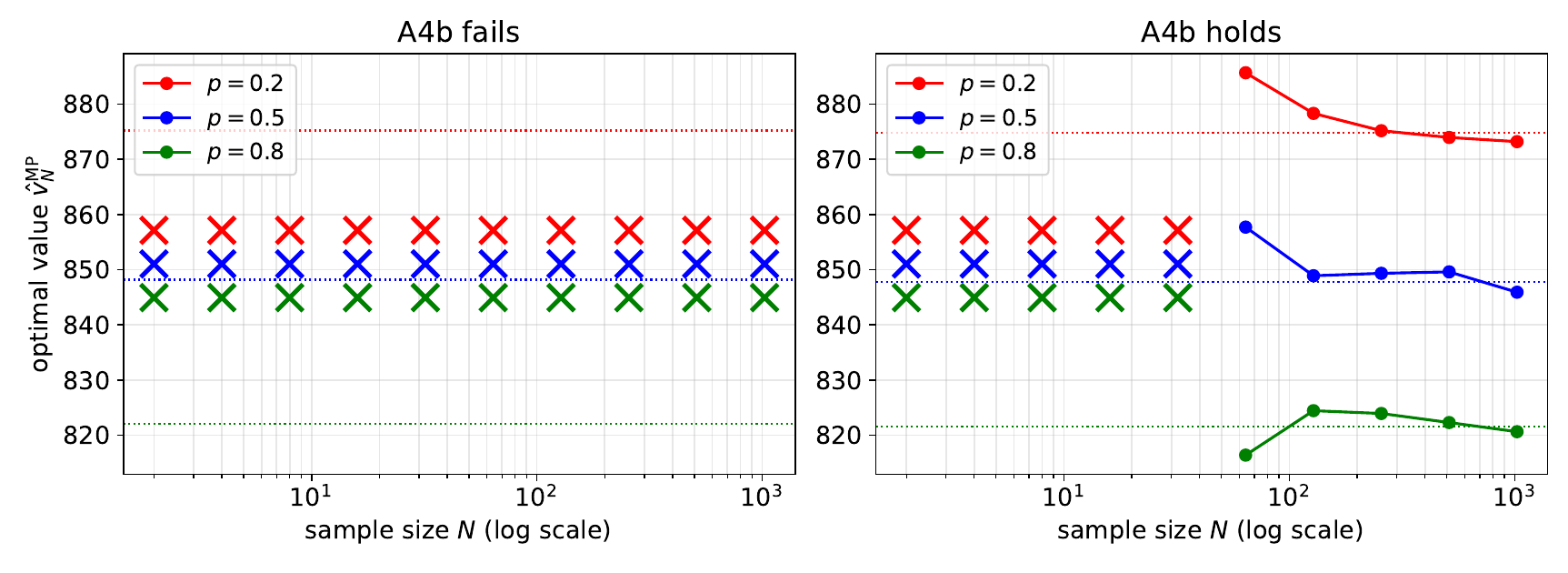}
		\caption{Optimal values of the multi-policy approximation~\eqref{eqn:multi_policy} for the instance with $\Gamma=2$, where A4b fails (left), and the comparison instance with $\Gamma=1.999$, where A4b holds (right). The three colors correspond to point-mass probabilities $p=0.2$, $0.5$, and $0.8$. Dotted lines show SAA reference values computed from $3{,}000$ independent samples. Crosses indicate solver-certified infeasibility; their vertical positions do not represent objective values.}
		\label{fig:lotsizing_convergence}
	\end{figure}
	
	The verification results and runtimes are reported in \Cref{tab:network_ls_a4b} and \Cref{fig:a4b_timing_one_sided}. Unlike the instances in \Cref{subsec:newsvendor_two_side_budget}, some lot-sizing instances violate A4b despite satisfying A4a since neither the support set $\Xi$ is simple nor a trivial globally feasible affine decision exists. \Cref{fig:a4b_timing_one_sided} compares the average verification time using logarithmic scales on both axes. The approximately linear curves are consistent with polynomial growth over the tested range of $n$. Indeed, the problem dimensions $(m,r,d)$ and the number of incident edges $L=\Gamma(n-\Gamma+1)$ are polynomial in $n$, so \Cref{thm:alg_solve_llf} guarantees polynomial-time verification cost of applying \Cref{alg:check_llf} at each vertex. Moreover, for fixed $\Gamma$, the number of non-simple vertices $\binom{n}{\Gamma}$ is also polynomial in $n$. Consequently, applying \Cref{alg:check_llf} at all non-simple vertices has polynomial complexity in $n$ for fixed $\Gamma$, although the total verification time can still be substantially greater than the time required at a single vertex.

	To examine the role of A4b in the convergence of the multi-policy approximation, we consider the lot-sizing instance with $n=5$ and $\Gamma=2$ for which A4a holds, but A4b fails. We construct a comparison instance by reducing the budget to $\Gamma=1.999$ while keeping all network costs and capacities unchanged. The smaller support is simple by \Cref{coro_simple_one_sided_budget}, and global feasibility is preserved, so A4b holds for the comparison instance by \Cref{thm:no_more_d_edges}. We generate samples from a distribution that assigns probability $p\in\{0.2,0.5,0.8\}$ to $\bzz=(0,0,0,1,1)^\top$ and otherwise follows a Gaussian distribution with mean parameter $(2/5)\bm e$ and covariance matrix $0.25^2\bm I_5$, truncated to the support in \eqref{eqn:one_sided_budget} with $\Gamma=2$. The original instance uses these samples directly, whereas the comparison instance uses their Euclidean projections onto the support in \eqref{eqn:one_sided_budget} with $\Gamma=1.999$.
	
	We solve the multi-policy approximation~\eqref{eqn:multi_policy} for $N\in\{2,4,8,16,32,64,128,256,512,1024\}$, using Euclidean sample neighborhoods with radius $\epsilon_N=N^{-2}$. For each probability $p$, we generate $3{,}000$ additional independent samples and obtain separate SAA reference values for the two instances using the original samples and their projections, respectively. \Cref{fig:lotsizing_convergence} reports the approximation values and the SAA reference values. When $\Gamma=1.999$, the multi-policy approximation~\eqref{eqn:multi_policy} is feasible for all tested sample sizes from $N=64$ onward, and the objective values approach their respective reference values as $N$ increases. When $\Gamma=2$, the multi-policy approximation \eqref{eqn:multi_policy} is infeasible at every tested sample size for all three probabilities. More importantly, the comparison illustrates how the geometric results in \Cref{sec:budgeted_uncertainty_set} can guide the choice of budget for the decision maker: reducing the budget $\Gamma$ from $2$ to $1.999$ makes $\Xi$ simple and restores A4b, with the projected samples yielding feasible approximations that approach their corresponding reference values.

	\section{Conclusion and Future Work}\label{sec:discussion}
	This paper resolves the open question of whether global feasibility implies locally linear feasibility in two-stage sample robust optimization. We show that the implication holds uniformly over all recourse systems if and only if the polyhedral support is simple. For every non-simple support, a globally feasible recourse system can be constructed for which locally linear feasibility fails. We further characterize the simpleness of one-sided and two-sided budgeted uncertainty sets and develop an algorithm that verifies locally linear feasibility at a given vertex in polynomial time when the incident edge directions are available. The newsvendor and network lot-sizing examples illustrate the verification of A4b and its role in the multi-policy approximation.
	
	Several directions remain for future research. First, extending our results to multistage sample robust optimization would require studying how local feasibility interacts with nonanticipativity; that is, decisions at intermediate stages must accommodate future uncertainty without using information that has not yet been observed. A geometric characterization of asymptotical optimality for multistage problems would therefore need to account for both the uncertainty support and the sequence of decisions. Second, many applications involve a second-stage constraint matrix $\bm W$ that depends on uncertainty~$\bm \xi$, commonly referred to as \textit{random recourse}. Extending our analysis to random recourse presents an additional difficulty: even when the matrix $\bm W$ depends affinely on $\bm \xi$, substituting an affine decision rule introduces quadratic terms into the constraints, so the present arguments based on linear inequalities no longer apply directly. Third, while this paper considers continuous recourse decisions, applications involving discrete actions, such as facility activation or production setup, call for integer recourse decisions. Extending the analysis to such problems presents a further challenge: integrality restrictions lead to nonconvex feasible sets and limit the applicability of affine decision rules. Identifying alternative policy classes and establishing conditions under which those policy classes contain locally feasible policies would be central to extending the approximation framework and its asymptotic guarantees to this broader class of problems.
	
	\linespread{1}
	\small
	
	\bibliographystyle{plainnat}
	\bibliography{references}

@article{bertsimas2012power,
  title={On the power and limitations of affine policies in two-stage adaptive optimization},
  author={Bertsimas, Dimitris and Goyal, Vineet},
  journal={Mathematical Programming},
  volume={134},
  number={2},
  pages={491--531},
  year={2012},
  publisher={Springer}
}

@article{bertsimas2022two,
  title={Two-stage sample robust optimization},
  author={Bertsimas, Dimitris and Shtern, Shimrit and Sturt, Bradley},
  journal={Operations Research},
  volume={70},
  number={1},
  pages={624--640},
  year={2022},
  publisher={INFORMS}
}

@article{bertsimas2023data,
  title={A data-driven approach to multistage stochastic linear optimization},
  author={Bertsimas, Dimitris and Shtern, Shimrit and Sturt, Bradley},
  journal={Management Science},
  volume={69},
  number={1},
  pages={51--74},
  year={2023},
  publisher={INFORMS}
}

@book{rockafellar1970convex,
  title={Convex Analysis},
  author={Rockafellar, RT},
  publisher={Princeton University Press},
  year={1970}
}

@book{conforti2014integer,
  author={Conforti, Michele and Cornu{\'e}jols, G{\'e}rard and Zambelli, Giacomo},
  title={Integer Programming},
  year={2014},
  publisher={Springer, Berlin}
}

@book{ziegler2012lectures,
  title={Lectures on polytopes},
  author={Ziegler, G{\"u}nter M},
  volume={152},
  year={2012},
  publisher={Springer Science \& Business Media}
}

@book{aliprantis2007cones,
  title={Cones and duality},
  author={Aliprantis, Charalambos D and Tourky, Rabee},
  volume={84},
  year={2007},
  publisher={American Mathematical Soc.}
}

@article{dantzig1955linear,
  title={Linear programming under uncertainty},
  author={Dantzig, George B},
  journal={Management Science},
  volume={1},
  number={3-4},
  pages={197--206},
  year={1955},
  publisher={Informs}
}

@article{beale1955minimizing,
  title={On minimizing a convex function subject to linear inequalities},
  author={Beale, Evelyn ML},
  journal={Journal of the Royal Statistical Society Series B: Statistical Methodology},
  volume={17},
  number={2},
  pages={173--184},
  year={1955},
  publisher={Oxford University Press}
}

@book{birge1997introduction,
  title={Introduction to stochastic programming},
  author={Birge, John R and Louveaux, Francois},
  year={1997},
  publisher={Springer}
}

@article{nishizaki2023two,
  title={A two-stage linear production planning model with partial cooperation under stochastic demands},
  author={Nishizaki, Ichiro and Hayashida, Tomohiro and Sekizaki, Shinya and Furumi, Kojiro},
  journal={Annals of Operations Research},
  volume={320},
  number={1},
  pages={293--324},
  year={2023},
  publisher={Springer}
}

@article{yu2025value,
  title={On the value of risk-averse multistage stochastic programming in capacity planning},
  author={Yu, Xian and Shen, Siqian},
  journal={INFORMS Journal on Computing},
  volume={37},
  number={5},
  pages={1143--1162},
  year={2025},
  publisher={INFORMS}
}

@article{koutsokosta2024stochastic,
  title={Stochastic transitions of a mixed-integer linear programming model for the construction supply chain: chance-constrained programming and two-stage programming},
  author={Koutsokosta, Aspasia and Katsavounis, Stefanos},
  journal={Operational Research},
  volume={24},
  number={3},
  pages={46},
  year={2024},
  publisher={Springer}
}

@article{sen2016mitigating,
  title={Mitigating uncertainty via compromise decisions in two-stage stochastic linear programming: Variance reduction},
  author={Sen, Suvrajeet and Liu, Yifan},
  journal={Operations Research},
  volume={64},
  number={6},
  pages={1422--1437},
  year={2016},
  publisher={INFORMS}
}

@article{berkelaar2002primal,
  title={A primal-dual decomposition-based interior point approach to two-stage stochastic linear programming},
  author={Berkelaar, Arjan and Dert, Cees and Oldenkamp, Bart and Zhang, Shuzhong},
  journal={Operations Research},
  volume={50},
  number={5},
  pages={904--915},
  year={2002},
  publisher={INFORMS}
}

@article{gangammanavar2021stochastic,
  title={Stochastic decomposition for two-stage stochastic linear programs with random cost coefficients},
  author={Gangammanavar, Harsha and Liu, Yifan and Sen, Suvrajeet},
  journal={INFORMS Journal on Computing},
  volume={33},
  number={1},
  pages={51--71},
  year={2021},
  publisher={INFORMS}
}

@article{pasupathy2021adaptive,
  title={Adaptive sequential sample average approximation for solving two-stage stochastic linear programs},
  author={Pasupathy, Raghu and Song, Yongjia},
  journal={SIAM Journal on Optimization},
  volume={31},
  number={1},
  pages={1017--1048},
  year={2021},
  publisher={SIAM}
}

@article{hanasusanto2018conic,
  title={Conic programming reformulations of two-stage distributionally robust linear programs over {W}asserstein balls},
  author={Hanasusanto, Grani A and Kuhn, Daniel},
  journal={Operations Research},
  volume={66},
  number={3},
  pages={849--869},
  year={2018},
  publisher={INFORMS}
}

@article{xie2020tractable,
  title={Tractable reformulations of two-stage distributionally robust linear programs over the type-$\infty$ {W}asserstein ball},
  author={Xie, Weijun},
  journal={Operations Research Letters},
  volume={48},
  number={4},
  pages={513--523},
  year={2020},
  publisher={Elsevier}
}

@inproceedings{feige2007robust,
  title={Robust combinatorial optimization with exponential scenarios},
  author={Feige, Uriel and Jain, Kamal and Mahdian, Mohammad and Mirrokni, Vahab},
  booktitle={International Conference on Integer Programming and Combinatorial Optimization},
  pages={439--453},
  year={2007},
  organization={Springer}
}

@book{schrijver1998theory,
  title={Theory of linear and integer programming},
  author={Schrijver, Alexander},
  year={1998},
  publisher={John Wiley \& Sons}
}

@article{bertsimas2004price,
  title={The price of robustness},
  author={Bertsimas, Dimitris and Sim, Melvyn},
  journal={Operations Research},
  volume={52},
  number={1},
  pages={35--53},
  year={2004},
  publisher={Informs}
}

@article{ben2004adjustable,
  title={Adjustable robust solutions of uncertain linear programs},
  author={Ben-Tal, Aharon and Goryashko, Alexander and Guslitzer, Elana and Nemirovski, Arkadi},
  journal={Mathematical programming},
  volume={99},
  number={2},
  pages={351--376},
  year={2004},
  publisher={Springer}
}

@article{ardestani2016robust,
  title={Robust optimization of sums of piecewise linear functions with application to inventory problems},
  author={Ardestani-Jaafari, Amir and Delage, Erick},
  journal={Operations Research},
  volume={64},
  number={2},
  pages={474--494},
  year={2016},
  publisher={INFORMS}
}

@article{bertsimas2016duality,
  title={Duality in two-stage adaptive linear optimization: Faster computation and stronger bounds},
  author={Bertsimas, Dimitris and de Ruiter, Frans JCT},
  journal={INFORMS Journal on Computing},
  volume={28},
  number={3},
  pages={500--511},
  year={2016},
  publisher={INFORMS}
}

@article{zhen2018adjustable,
  title={Adjustable robust optimization via Fourier--Motzkin elimination},
  author={Zhen, Jianzhe and Den Hertog, Dick and Sim, Melvyn},
  journal={Operations Research},
  volume={66},
  number={4},
  pages={1086--1100},
  year={2018},
  publisher={INFORMS}
}

@article{bertsimas2006robust,
  title={A robust optimization approach to inventory theory},
  author={Bertsimas, Dimitris and Thiele, Aur{\'e}lie},
  journal={Operations Research},
  volume={54},
  number={1},
  pages={150--168},
  year={2006},
  publisher={INFORMS}
}

@article{bertsimas2012adaptive,
  title={Adaptive robust optimization for the security constrained unit commitment problem},
  author={Bertsimas, Dimitris and Litvinov, Eugene and Sun, Xu Andy and Zhao, Jinye and Zheng, Tongxin},
  journal={IEEE transactions on power systems},
  volume={28},
  number={1},
  pages={52--63},
  year={2012},
  publisher={IEEE}
}

@article{munari2019robust,
  title={The robust vehicle routing problem with time windows: Compact formulation and branch-price-and-cut method},
  author={Munari, Pedro and Moreno, Alfredo and De La Vega, Jonathan and Alem, Douglas and Gondzio, Jacek and Morabito, Reinaldo},
  journal={Transportation Science},
  volume={53},
  number={4},
  pages={1043--1066},
  year={2019},
  publisher={INFORMS}
}

@manual{mosek,
  author = {{MOSEK ApS}},
  title  = {{MOSEK Optimizer API for Python, Version 11.2}},
  year   = {2026},
  url    = {https://docs.mosek.com/latest/pythonapi/index.html}
}
	
	\linespread{1.5}
	\normalsize
	
	\begin{appendix}
		
		\setcounter{section}{0}
		\renewcommand{\thesection}{Appendix \Alph{section}}
		
		\setcounter{equation}{0}
		\renewcommand{\theequation}{\Alph{section}.\arabic{equation}}
		
		\setcounter{proposition}{0}
		\renewcommand{\theproposition}{\Alph{section}.\arabic{proposition}}
		\makeatletter
		\@addtoreset{proposition}{section}
		\makeatother
		
		\setcounter{theorem}{0}
		\renewcommand{\thetheorem}{\Alph{section}.\arabic{theorem}}
		\makeatletter
		\@addtoreset{theorem}{section}
		\makeatother
		
		\setcounter{lemma}{0}
		\renewcommand{\thelemma}{\Alph{section}.\arabic{lemma}}
		\makeatletter
		\@addtoreset{lemma}{section}
		\makeatother
		
		\setcounter{remark}{0}
		\renewcommand{\theremark}{\Alph{section}.\arabic{remark}}
		\makeatletter
		\@addtoreset{remark}{section}
		\makeatother
		
		\setcounter{algorithm}{0}
		\renewcommand{\thealgorithm}{\Alph{section}.\arabic{algorithm}}
		\makeatletter
		\@addtoreset{algorithm}{section}
		\makeatother

		\section{Proof of Results in \Cref{sec:geometric_character}}\label{ec_proof_geometric}
		\subsection{Proof of Results in \Cref{subsec:sufficient}}\label{ec_proof_sufficient}
		Before proving \Cref{thm:no_more_d_edges}, we first record a Farkas-type characterization of global feasibility. 
		\begin{lemma}[Farkas characterization of globally feasibility]\label{lem:gf_charac}
			The recourse system~\eqref{eqn:linear_sys} is globally feasible if and only if every $\bm \mu \in \mathbb{R}_{+}^m$ with $\bm W^\top \bm \mu=\bm 0$ satisfies
			\begin{equation*}
				\bm \mu^{\top} \bm H \bm v_{k} \geq 0 \quad \forall k\in [K] \quad \text { and } \quad \bm \mu^{\top} (\bm b - \bm H\bm\zeta_s) \leq 0 \quad \forall s\in [S],
			\end{equation*}  
			where $\{\bm \zeta_s\}_{s\in [S]}$ and $\{\bm v_{k}\}_{k\in [K]}$ are vertices of $\Xi$ and extreme rays of the recession cone of $\Xi$, respectively. 
		\end{lemma}
		
		\begin{proof}[Proof of \Cref{lem:gf_charac}]
			By Farkas' lemma (see, \textit{e.g.} Theorem 3.4 in \citealt{conforti2014integer}), the system~\eqref{eqn:linear_sys} is feasible in $\bm y$ for fixed $\bm \zeta \in \Xi$ iff $\bm \mu^{\top}(\bm H\bm \zeta - \bm b) \geq 0$ for every $\bm \mu \in \mathbb{R}_{+}^m$ with $\bm W^\top \bm \mu =\bm 0$. Noting that $\bm \zeta = \sum_{s\in [S]}\alpha_s \bm \zeta_s + \sum_{k\in [K]} \beta_{k}\bm v_{k}$ for $\bm\alpha\in \Delta_S$ and $\bm \beta\in \mb R^K_+$, system~\eqref{eqn:linear_sys} is globally feasible iff 
			\begin{equation*}
				\bm \mu^{\top} \left(\bm b - \displaystyle\sum_{s\in [S]} \alpha_s \bm H\bm\zeta_s \right) \leq \displaystyle\sum_{k\in [K]} \beta_{k} \bm \mu^\top \bm H \bm v_{k} \quad \forall \bm \mu \in {\rm ker}(\bm W^\top) \cap \mb R^m_+, \; \bm\alpha\in \Delta_S, \; \bm \beta\in \mb R^K_+.
			\end{equation*}
			Now, for any fixed $\bm \mu$, the supremum of the left-hand-side over $\bm \alpha$ is not greater than the infimum of the right-hand-side over $\bm \beta$. Noting that $\bm \beta\in \mb R^K_+$ is unbounded, the infimum of the right-hand-side is either $0$, if $\bm \mu^{\top} \bm H \bm v_{k} \geq 0$ for every $k$, or $-\infty$ otherwise. Thus, system~\eqref{eqn:linear_sys} is globally feasible iff 
			\begin{equation*}
				\bm \mu^{\top} \bm H \bm v_{k} \geq 0 \quad \text{ and } \quad \sup_{\bm\alpha\in \Delta_S} \bm \mu^{\top} \left(\bm b - \displaystyle\sum_{s\in [S]} \alpha_s \bm H\bm\zeta_s \right) \leq 0 \quad \forall \bm \mu \in {\rm ker}(\bm W^\top) \cap \mb R^m_+, \; k\in [K].
			\end{equation*}
			Noticing that the above supremum problem is a linear program on a compact polyhedron, the supremum is attained at one vertex of $\Delta_S$. Then the conclusion follows by substituting $\bm \alpha$ by all the canonical basis of $\mb R^S$. 
		\end{proof}
		
		\begin{proposition}\label{prop:gf_is_llf}
			If $\Xi$ is a translated simplicial cone; that is, $\Xi$ is with only one vertex $\bm \zeta_1$, and the recession cone of $\Xi$ has $d$ extreme rays $\bm v_1, \ldots, \bm v_d$ in $\mb R^d$, then any globally feasible system~\eqref{eqn:linear_sys} on $\Xi$ is also locally linear feasible at $\bm \zeta_1$. Moreover, the locally linear feasible policy can be selected globally, which serves as a feasible linear decision rule $\bm y(\cdot)$ on $\Xi$.
		\end{proposition}
		
		\begin{proof}[Proof of \Cref{prop:gf_is_llf}]
			By the definition of an extreme ray, vectors $\{\bm v_1, \ldots, \bm v_d\}$ are linearly independent in $\mb R^d$. Therefore, the matrix $\bm V = [\bm v_1, \ldots, \bm v_d] \in \mb R^{d\times d}$ is invertible.
			
			Assume that~\eqref{eqn:linear_sys} is a globally feasible recourse system, then, there exists $\bm z$ such that $\bm W\bm z + \bm H \bm \zeta_1 \geq \bm b$. By \Cref{lem:gf_charac}, every $\bm \mu \in \mathbb{R}_{+}^m$ with $\bm W^\top \bm \mu=\bm 0$ satisfies $\bm \mu^{\top} \bm H \bm v_{\ell} \geq 0$ for each ${\ell}\in [d]$. This condition is equivalent to the condition that there is some $\bm y_{\ell}$ with $\bm W\bm y_{\ell} + \bm H\bm v_{\ell} \geq \bm 0$ for each ${\ell}\in [d]$. Now, we design $\bm Y = [\bm y_1, \ldots, \bm y_d]\bm V^{-1} \in \mb R^{r\times d}$, and prove that $\bm y(\bm \zeta) = \bm Y(\bm \zeta - \bm \zeta_1) + \bm z$ is a desired locally linear feasible policy in the following. For any $\bm \zeta\in \Xi$, we write $\bm \zeta = \bm \zeta_1 + \sum_{{\ell}\in [d]} \lambda_{\ell} \bm v_{\ell}$, where $\bm \lambda\geq \bm 0$ is the barycentric coordinates. Then we can compute that
			\begin{equation*}
				\begin{array}{rl}
					\bm W\bm y(\bm \zeta) + \bm H\bm \zeta -\bm b = & \bm W\left([\bm y_1, \ldots, \bm y_d]\bm V^{-1}\left(\bm \zeta - \bm \zeta_1 \right) + \bm z \right) + \bm H\bm \zeta -\bm b\\
					= & \displaystyle\sum_{{\ell}\in [d]} \lambda_{\ell} \left( \bm W [\bm y_1, \ldots, \bm y_d]\bm V^{-1} \bm v_{\ell} + \bm H\bm v_{\ell} \right) + \bm W\bm z + \bm H \bm \zeta_1 - \bm b\\
					= & \displaystyle\sum_{{\ell}\in [d]} \lambda_{\ell} \left( \bm W \bm y_{\ell} + \bm H\bm v_{\ell} \right) + \bm W\bm z + \bm H \bm \zeta_1 - \bm b \geq \bm 0.
				\end{array}
			\end{equation*}
			Therefore, the above defined $\bm y(\bm \zeta)$ is actually a globally feasible (on $\Xi$) linear decision rule, which implies that~\eqref{eqn:linear_sys} is also locally linear feasible at the only vertex $\bm \zeta_1$. 
		\end{proof}

		\begin{proof}[Proof of \Cref{thm:no_more_d_edges}]
			We first consider the case $\dim(\Xi)= d$. Let the direction of the $d$ edges incident to $\bm \zeta_0$ be $\{\bm \eta_1, \ldots, \bm \eta_d\}$. W.L.O.G., we assume that the length of those directions is small enough so that $\bm \zeta_0 + \bm \eta_{\ell} \in \Xi$ for each $\ell\in [d]$. We then claim that the tangent cone of $\Xi$ at $\bm \zeta_0$, defined as $\mc T_{\bm \zeta_0} = \{\bm \eta \in \mb R^d ~:~ \exists \; \lambda > 0 \text{ such that } \bm \zeta_0 + \lambda \bm \eta\in \Xi\}$, has the extreme rays $\{\bm \eta_1, \ldots, \bm \eta_d\}$.
			
			To see how the claim holds, we utilize the H-representation of $\Xi = \{\bm \xi ~:~ \bm G\bm \xi \leq \bm g\}$. Then it is true that the tangent cone $\mc T_{\bm \zeta_0} = \{ \bm \xi ~:~ \bm G_{\mc I} \bm \xi \leq \bm 0\}$, where $\mc I = \{i ~:~ \bm G_i \bm \zeta_0 = g_i\}$ is the index set of constraints active at $\bm \zeta_0$. Thus, $\mc T_{\bm \zeta_0}$ is a polyhedron cone. Moreover, $\bm \zeta_0$ is also a vertex of $\mc T_{\bm \zeta_0}$. Hence, $\mc T_{\bm \zeta_0}$ is a pointed polyhedron cone. By Theorem 3.35 of \cite{conforti2014integer}, any extreme ray $\bm r$ of $\mc T_{\bm \zeta_0}$ satisfies $d-1$ linearly independent inequalities of $\{ \bm \xi ~:~ \bm G_{\mc I} \bm \xi \leq \bm 0\}$ at equality. W.L.O.G., let $\mc I = [I]$, and let $\bm G_1, \ldots, \bm G_{d-1}$ be linearly independent, and $\bm G_{\ell} \bm r = 0$ for each $\ell\in [d-1]$. Then we construct the hyperplane $\mc H = \{\bm \xi ~:~ \bm \alpha^\top \bm \xi = \beta\}$, where the vector $\bm \alpha\in \mb R^d$ and scalar $\beta$ as
			\begin{equation*}
				\bm \alpha = \sum_{\ell\in [d-1]}\bm G_{\ell}^\top \quad \text{ and } \quad \beta = \sum_{\ell\in [d-1]}\bm G_{\ell} \bm \zeta_0 = \sum_{\ell\in [d-1]}g_{\ell}. 
			\end{equation*}
			It can be verified that $\Xi \subset \mc H^- = \{\bm \xi ~:~ \bm \alpha^\top \bm \xi \leq \beta\}$ and $\{\bm \zeta_0 + \lambda \bm r ~:~ \lambda\geq 0\} \subset \mc H$. Therefore, $\Xi\cap \mc H$ is a face of $\Xi$. Moreover, any $\bm \zeta \in \Xi\cap \mc H$, we have $\bm G_{\ell} (\bm \zeta-\bm \zeta_0) = 0$ for each $\ell\in [d-1]$. Since $\bm G_1, \ldots, \bm G_{d-1}$ is linearly independent, the face $\Xi\cap \mc H$ is at most $1$-dimensional. By $\{\bm \zeta_0 + \lambda \bm r ~:~ \lambda\geq 0\} \subset \mc H$, we get this face is an edge. So the extreme ray $\bm r$ of tangent cone $\mc T_{\bm \zeta_0}$ is an edge direction of $\Xi$ at $\bm \zeta_0$.
			
			On the other hand, consider any edge direction $\bm \eta_0$ at $\bm \zeta_0$. Assume that $\bm \eta_0$ is not an extreme ray of $\mc T_{\bm \zeta_0}$, then there exists $\bm \delta_1$ and $\bm \delta_2$, which are not parallel to $\bm \eta_0$, such that $\bm \eta_0 = \bm \delta_1 + \bm \delta_2$. W.L.O.G., assume that the length of $\bm \eta_0$ is small enough so that $\bm \zeta_0 + \bm \delta_1 \in \Xi$ and $\bm \zeta_0 + \bm \delta_2 \in \Xi$. Since $\bm \eta_0$ is an edge direction, there exists vector $\bm \alpha$ and scalar $\beta$ such that
			\begin{equation*}
				\bm \alpha^\top \bm \zeta_0 = \beta, \quad \bm \alpha^\top \left(\bm \zeta_0 + \bm \eta_0\right) = \beta, \quad \bm \alpha^\top \left(\bm \zeta_0 + \bm \delta_1\right) < \beta, \quad \text{ and } \quad \bm \alpha^\top \left(\bm \zeta_0 + \bm \delta_2\right) < \beta.
			\end{equation*}
			It follows that $\bm \alpha^\top\bm \eta_0 = 0$, $\bm \alpha^\top\bm \delta_1 < 0$, and $\bm \alpha^\top\bm \delta_2 < 0$. This is a contradiction with $\bm \eta_0 = \bm \delta_1 + \bm \delta_2$, which implies that any edge direction $\bm \eta_0$ at $\bm \zeta_0$ must be an extreme ray of $\mc T_{\bm \zeta_0}$. Thus, our claim is true.
			
			Next, noting that $\Xi\subset \mc T_{\bm \zeta_0}\subset \mb R^d$ and $\dim(\Xi) = d$, it follows $\dim(\mc T_{\bm \zeta_0}) =d$. Due to our above claim $\mc T_{\bm \zeta_0} = {\rm cone} \{\bm \eta_1, \ldots, \bm \eta_d\}$, it is true that edge directions $\bm \eta_1, \ldots, \bm \eta_d$ are linearly independent. So the matrix $\bm E = [\bm \eta_1, \ldots, \bm \eta_d]\in \mb R^{d\times d}$ is invertible. Now consider any globally feasible recourse system~\eqref{eqn:linear_sys}, by definition, there exists $\bm y_0, \bm y_1, \ldots, \bm y_d\in \mb R^{r}$ such that 
			\begin{equation*}
				\bm W \bm y_0 + \bm H\bm \zeta_0 \geq \bm b, \quad \text{ and }\quad \bm W \bm y_{\ell} + \bm H\left(\bm \zeta_0 + \bm \eta_{\ell}\right) \geq \bm b \quad \forall \ell\in [d].
			\end{equation*}
			Following the same construction in Theorem 1 of \cite{bertsimas2012power}, for each $\bm \zeta$ in the simplex ${\rm conv} \{\bm \zeta_0, \bm \zeta_0+\bm \eta_1, \ldots, \bm \zeta_0 +\bm \eta_d\}$, the vector $\bm \zeta-\bm \zeta_0$ admits a unique convex combination of $\{\bm \eta_1, \ldots, \bm \eta_d\}$. Moreover, the coefficients $\{\alpha_{\ell}\}_{\ell \in [d]}$ is an affine function of $\bm \zeta$. Then $\bm y(\bm \zeta) = \bm y_0 + \sum_{\ell\in [d]}\alpha_{\ell}(\bm y_{\ell}- \bm y_0)$ is an affine function that is feasible to~\eqref{eqn:linear_sys} for any $\bm \zeta \in {\rm conv} \{\bm \zeta_0, \bm \zeta_0+\bm \eta_1, \ldots, \bm \zeta_0 +\bm \eta_d\}$. So the above $\bm y(\cdot)$ provides a locally linear feasible policy at $\bm \zeta_0$. 
			
			Then, we consider the case of $\dim(\Xi) < d$. Denote $q= \dim(\Xi)$, then by Corollary 1.6.1 in \cite{rockafellar1970convex}, we can find a one-to-one affine mapping from ${\rm aff}(\Xi)$ onto $\mb R^q$. Specifically, we consider a set of orthonormal bases of the subspace ${\rm aff}(\Xi)-\bm \zeta_0$, denoted as $\{\bm \xi_1, \ldots, \bm \xi_q\}$, then
			\begin{equation*}
				{\rm aff}(\Xi) = \bm \zeta_0 + {\rm span}\{\bm \xi_1, \ldots, \bm \xi_q\}.
			\end{equation*}
			Each point $\bm \zeta \in {\rm aff}(\Xi)$, there is a unique $\bm \theta\in \mb R^q$ such that
			\begin{equation*}
				\bm \zeta = \bm \zeta_0 + [\bm \xi_1, \ldots, \bm \xi_q]\bm \theta, \quad \text{ and } \quad \bm \theta = \begin{bmatrix}
					\bm \xi_1^\top\\
					\vdots\\
					\bm \xi_q^\top
				\end{bmatrix} \left(\bm \zeta-\bm \zeta_0\right).
			\end{equation*}
			The above affine mapping is nothing but a composition of a translation and a linear isometry. Thus, the image of $\Xi$, denoted as $\Theta$, is also a polyhedron. Moreover, the invertible affine mapping establishes a bijective correspondence between the vertices, edges, and faces of $\Xi$ and those of $\Theta$. Therefore, considering the global feasibility (or locally linear feasibility at $\bm \zeta_0$) of system~\eqref{eqn:linear_sys} on $\Xi$ is equivalent to considering the global feasibility (or locally linear feasibility at $\bm 0\in \mb R^q$) of the system
			\begin{equation}\label{eqn:proj_sys}
				\bm W\bm y + \bm H \left(\bm \zeta_0 + [\bm \xi_1, \ldots, \bm \xi_q]\bm \theta\right) \geq \bm b.
			\end{equation}
			Now, system~\eqref{eqn:proj_sys} is globally feasible on $\Theta\subset \mb R^q$. Together with $\dim(\Theta) = \dim(\Xi) = q$, it follows from our previous analysis that system~\eqref{eqn:proj_sys} is locally linear feasible at $\bm 0\in \mb R^q$ on $\Theta$. Therefore, system~\eqref{eqn:linear_sys} is also locally linear feasible at $\bm \zeta_0$ on $\Xi$, which finishes the proof.    
		\end{proof}
		
		\subsection{Proof of Results in \Cref{subsec:necessary}}\label{ec_proof_necessary}
		\begin{proof}[Proof of \Cref{prop:counterexample}]
			We first verify that~\eqref{eqn:linear_sys} is globally feasible on the tangent cone $\mc T_{\bm 0} = {\rm cone}\{\bm p_1, \ldots, \bm p_7\}$, hence is also globally feasible on $\Xi$. To see it, by \Cref{lem:gf_charac}, it suffices to show 
			\begin{equation}\label{eqn:equiv_global_feasible}
				\forall \bm \lambda \in \mathcal N \triangleq {\rm ker}(\bm W^\top) \cap \mathbb R^m_{+}, \text{ it holds that }     \bm H^\top \bm \lambda \in \mc T_{\bm 0}^\star,
			\end{equation}
			where $\mc T_{\bm 0}^\star$ is the dual cone of $\mc T_{\bm 0}$. At the same time, by construction of $\bm W$, we have
			\begin{equation*}
				\mathcal N = \left\{\bm \lambda \in \mathbb R^4 ~:~ \lambda_1 + \lambda_2 = \lambda_3 +\lambda_4, \; \lambda_i\geq 0, \; \forall i\in [4] \right\}.
			\end{equation*}
			The extreme rays of $\mathcal N$ is given by
			\begin{equation*}
				\bm \lambda^{1} = \begin{bmatrix} 1 \\ 0 \\ 1 \\ 0 \end{bmatrix}, \quad
				\bm \lambda^{2} = \begin{bmatrix} 1 \\ 0 \\ 0 \\ 1 \end{bmatrix}, \quad
				\bm \lambda^{3} = \begin{bmatrix} 0 \\ 1 \\ 1 \\ 0 \end{bmatrix}, \quad \text{and} \quad
				\bm \lambda^{4} = \begin{bmatrix} 0 \\ 1 \\ 0 \\ 1 \end{bmatrix},
			\end{equation*}
			and it could be verified that
			\begin{equation*}
				\mathcal N = {\rm cone}\left\{\bm \lambda^{1}, \bm \lambda^{3}, \bm \lambda^{2}, \bm \lambda^{4}\right\}.
			\end{equation*}
			By construction of $\bm H$, we have
			\begin{equation*}
				\bm H^\top \bm \lambda^{1} = \begin{bmatrix}
					3\\0\\3
				\end{bmatrix}, \bm H^\top \bm \lambda^{2} = \begin{bmatrix}
					0\\4\\4
				\end{bmatrix}, \bm H^\top \bm \lambda^{3} = \begin{bmatrix}
					3\\-4\\5
				\end{bmatrix}, \text{ and }  \bm H^\top \bm \lambda^{4} = \begin{bmatrix}
					0\\0\\6
				\end{bmatrix}.
			\end{equation*}
			We note that, for each above vector, the square of last component is not less than the sum of square of the other two components. Therefore, it can be verified that 
			\begin{equation*}
				\bm p_j^\top \bm H^\top \bm \lambda^{i} \geq 0 \quad \forall i\in [4], \; j\in [7].
			\end{equation*}
			So, we get that $\bm H^\top \bm \lambda^{i} \in \mc T_{\bm 0}^\star$ for all $i\in [4]$. Then~\eqref{eqn:equiv_global_feasible} follows from that $\mathcal N$ is the convex cone generated by $\{\bm \lambda^i\}_{i\in [4]}$. Thus, we get that~\eqref{eqn:linear_sys} is indeed globally feasible on $\mc T_{\bm 0}$ and $\Xi$.
			
			Then we prove~\eqref{eqn:linear_sys} is not locally linear feasible at $\bm \zeta_0=\bm 0$. We argue it by contradiction. Assume that there exists $\epsilon>0$ and $a_0, a_1, a_2, a_3 \in \mathbb R$ such that
			\begin{equation*}
				\bm W \left(a_0 + a_1\xi_1 + a_2\xi_2 + a_3\xi_3\right) +\bm H\bm \xi \geq \bm b \quad \forall \bm \xi\in \Xi\cap B(\bm \zeta_0, \epsilon).
			\end{equation*}
			By substituting $\bm \xi = \bm \zeta_0$ into the above recourse system, we immediately get that $a_0=0$. Let $\delta \triangleq \min\{1, \min_{j\in [7]}\{\epsilon/(2\|\bm p_j\|)\}\} >0$, then it follows that 
			\begin{equation*}
				\bm W \left( \bm a^\top \delta \bm p_j\right) + \delta \bm H \bm p_j \geq \bm b \quad \forall j\in [7]. 
			\end{equation*}
			Substituting $\bm b =\bm 0$ implies that 
			\begin{equation*}
				-\left(\bm W \bm p_j^\top \right)\bm a \leq \bm H\bm p_j\quad \forall j\in [7]. 
			\end{equation*}
			Define the matrix and vector 
			\begin{equation*}
				\hat{\bm A} \triangleq \begin{bmatrix}
					-\bm W\bm p_1^\top \\
					-\bm W\bm p_2^\top \\
					\vdots\\
					-\bm W\bm p_7^\top \\
				\end{bmatrix}\in \mathbb{R}^{28\times 3} \quad \text{ and } \quad \hat{\bm b}\triangleq \begin{bmatrix}
					\bm H\bm p_1\\
					\bm H\bm p_2\\
					\vdots\\
					\bm H\bm p_7\\
				\end{bmatrix}\in \mathbb R^{28}.
			\end{equation*}
			Then according to our assumption, we actually found a vector $\bm a\in \mathbb R^3$ such that $\hat{\bm A}\bm a \leq \hat{\bm b}$. By Farkas' lemma, this is equivalent to that, for all $\bm \mu\in \mathbb R^{28}_+$, if it holds that $\hat{\bm A}^\top \bm \mu = \bm 0$, then $\hat{\bm b}^\top \bm \mu \geq 0$. Therefore, to derive a contradiction, it suffices to show that
			\begin{equation}\label{eqn:equiv_not_LLF}
				\exists \; \bm \mu\in \mathbb R^{28} \text{ with } \bm \mu \geq \bm 0 \text{ and } \hat{\bm A}^\top \bm \mu = \bm 0, \; \text{ but } \hat{\bm b}^\top \bm \mu < 0.
			\end{equation}
			To construct such $\bm \mu$, we partition it into block subvectors as follows:
			\begin{equation*}
				\bm \mu = \begin{bmatrix}
					\bm \mu_{\star, 1}\\
					\bm \mu_{\star, 2}\\
					\vdots\\
					\bm \mu_{\star, 7}\\
				\end{bmatrix}, \; \text{ where } \bm \mu_{\star, j} = \begin{bmatrix}
					\mu_{1, j}\\
					\mu_{2, j}\\
					\mu_{3, j}\\
					\mu_{4, j}
				\end{bmatrix} \in \mb R^4.
			\end{equation*}
			Then, we let $\mu_{1, 5}, \mu_{2, 3}, \mu_{3,4}$ and $\mu_{4, 7}$ be the solution to the linear equations
			\begin{equation}\label{eqn:define_mu}
				\left\{\begin{array}{ll}
					\mu_{1, 5} \cdot \bm p_5 +  \mu_{2, 3} \cdot \bm p_3 - \mu_{3, 4} \cdot \bm p_4 - \mu_{4, 7} \cdot \bm p_7 = \bm 0\\
					\mu_{1, 5} + \mu_{2, 3} + \mu_{3,4} + \mu_{4, 7} = 1
				\end{array} \right.,
			\end{equation}
			and $\mu_{i, j} = 0$ for all the other pairs. Before we verify that such $\bm \mu$ is valid for~\eqref{eqn:equiv_not_LLF}, we note that the determinant of coefficient matrix of~\eqref{eqn:define_mu} is
			\begin{equation*}
				\begin{vmatrix}
					\cos \displaystyle\frac{8\pi}{7}& \cos \displaystyle\frac{4\pi}{7} & -\cos \displaystyle\frac{6\pi}{7} & -\cos \displaystyle\frac{12\pi}{7}\\[1em]
					\sin \displaystyle\frac{8\pi}{7}& \sin \displaystyle\frac{4\pi}{7} & -\sin \displaystyle\frac{6\pi}{7} & -\sin \displaystyle\frac{12\pi}{7}\\[1em]
					1& 1 & -1 & -1\\[1em]
					1 & 1 & 1 & 1
				\end{vmatrix} \approx -5.9449 <0.
			\end{equation*}
			Therefore, there exists a unique solution to~\eqref{eqn:define_mu}, and the claimed $\bm \mu$ is well-defined. In the follows, we verify the above $\bm \mu$ satisfies~\eqref{eqn:equiv_not_LLF}. First of all, by solving the linear equations~\eqref{eqn:define_mu}, we obtain the approximate values
			\begin{equation*}
				\mu_{1, 5} \approx 0.2775, \; \mu_{2, 3} \approx 0.2225, \;  \mu_{3,4} \approx 0.4010, \; \text{ and } \mu_{4, 7} \approx 0.0990.
			\end{equation*}
			Thus, we get that $\bm \mu \geq \bm 0$. Moreover,
			\begin{equation*}
				\begin{array}{rl}
					\hat{\bm A}^\top \bm \mu = \left[-\bm p_1 \bm W^\top, -\bm p_2 \bm W^\top, \ldots, -\bm p_7 \bm W^\top  \right] \begin{bmatrix}
						\bm \mu_{\star, 1}\\
						\bm \mu_{\star, 2}\\
						\vdots\\
						\bm \mu_{\star, 7}\\
					\end{bmatrix} = \displaystyle\sum_{j\in[7]} -\bm p_j \bm W^\top \bm \mu_{\star, j}    = & \displaystyle\sum_{j\in[7]} -\bm p_j \sum_{i\in [4]} w_i \mu_{i, j}\\
					= & -\displaystyle\sum_{i\in [4], j\in[7]} w_i \mu_{i, j} \cdot \bm p_j,
				\end{array}
			\end{equation*}
			where $w_i$ is the $i$-th component of $\bm W$ for each $i\in [4]$. Noting that $\mu_{i, j} = 0$ for most of the pairs, it follows that
			\begin{equation*}
				\hat{\bm A}^\top \bm \mu = -\left( \mu_{1, 5} \cdot \bm p_5 +  \mu_{2, 3} \cdot \bm p_3 - \mu_{3, 4} \cdot \bm p_4 - \mu_{4, 7} \cdot \bm p_7 \right) = \bm 0,
			\end{equation*}
			where the last step is from the linear equations~\eqref{eqn:define_mu}. Furthermore, we compute
			\begin{equation*}
				\begin{array}{rl}
					\hat{\bm b}^\top \bm \mu =  \left[\bm p_1^\top \bm H^\top, \bm p_2^\top \bm H^\top, \ldots, \bm p_7^\top \bm H^\top  \right] \begin{bmatrix}
						\bm \mu_{\star, 1}\\
						\bm \mu_{\star, 2}\\
						\vdots\\
						\bm \mu_{\star, 7}\\
					\end{bmatrix} = \displaystyle\sum_{j\in[7]} \bm p_j^\top \bm H^\top \bm \mu_{\star, j} = & \displaystyle\sum_{j\in[7]} \bm p_j^\top \sum_{i\in [4]} \mu_{i, j} \bm h_i\\
					= & \displaystyle\sum_{i\in [4], j\in[7]} \mu_{i, j}\cdot \bm p_j^\top \bm h_i,
				\end{array}
			\end{equation*}
			where $\bm h_i$ is the $i$-th column of $\bm H^\top$ for each $i\in [4]$. Similarly, noting that $\mu_{i, j} = 0$ for most of the pairs, it follows that
			\begin{equation*}
				\begin{array}{rl}
					\hat{\bm b}^\top \bm \mu = & \mu_{1, 5} \cdot \bm p_5^\top \bm h_1 +  \mu_{2, 3} \cdot \bm p_3^\top \bm h_2 + \mu_{3, 4} \cdot \bm p_4^\top \bm h_3 + \mu_{4, 7} \cdot \bm p_7^\top \bm h_4\\[1em]
					= & \left(\cos \displaystyle\frac{8\pi}{7} + 2\sin \displaystyle\frac{8\pi}{7} +1 \right)\mu_{1, 5} + \left(\cos \displaystyle\frac{4\pi}{7} - 2\sin \displaystyle\frac{4\pi}{7} + 3\right)\mu_{2, 3}\\[1em]
					& \quad + \left(2\cos \displaystyle\frac{6\pi}{7} -2\sin \displaystyle\frac{6\pi}{7} +2\right)\mu_{3, 4} + \left(-\cos \displaystyle\frac{12\pi}{7} + 2\sin \displaystyle\frac{12\pi}{7} +3\right) \mu_{4, 7}\\[1em]
					\approx & -0.2172 < 0.
				\end{array}
			\end{equation*}
			Therefore, the designed $\bm \mu$ indeed satisfies~\eqref{eqn:equiv_not_LLF}. Thus, a contradiction occurs, and the recourse system~\eqref{eqn:linear_sys} is not locally linear feasible. Hence, the example claimed in this theorem serves as a counterexample that is globally feasible but not locally linear feasible. 
		\end{proof}
		
		\subsection{Proof of Results in \Cref{subsec:necessary}}
		
		Without loss of generality, we assume throughout this subsection that $\Xi$ is a full-dimensional pointed polyhedron in $\mb R^d$. The case $\dim(\Xi)<d$ can be reduced to this setting by an invertible affine mapping on ${\rm aff}(\Xi)$, as shown in the proof of \Cref{thm:no_more_d_edges}. For notational convenience, we fix $\bzz$ as a non-simple vertex of $\Xi$ and let $\mc K = \mc T_{\bzz}$ denote the tangent cone of $\Xi$ at $\bzz$.
		
		We aim at finding a recourse system~\eqref{eqn:linear_sys}; that is, $\bm W, \bm H$ and $\bm b$ with proper dimensions, such that~\eqref{eqn:linear_sys} is globally feasible on the polyhedron $\bm \zeta_0 + \mc K$, hence is also globally feasible on $\Xi$, but not locally linear feasible at $\bm \zeta_0$. In particular, we construct them as 
		\begin{equation}\label{eqn:general_example_whb}
			\bm W=\begin{bmatrix}
				1\\1\\-1\\-1
			\end{bmatrix}, \quad \bm H = \begin{bmatrix}
				-\bm p_1^\top\\
				-\bm p_2^\top\\
				\bm q_1^\top\\
				\bm q_2^\top\\
			\end{bmatrix}, \quad \text{and} \quad \bm b=\bm H\bm \zeta_0,
		\end{equation}
		where the vectors $\bm p_1, \bm p_2, \bm q_1, \bm q_2 \in \mb R^d$ are constructed later according to the geometry of the tangent cone $\mc K$ at the non-simple vertex $\bzz$. The next lemma shows why this construction is useful. Specifically, under the structure prescribed in \eqref{eqn:general_example_whb} for \eqref{eqn:linear_sys}, both global feasibility and locally linear feasibility admit an exact characterization through a partial order induced by the dual cone $\mc K^\star$.
		
		\begin{lemma}\label{lem:HK_CK_condition}
			Consider the recourse system~\eqref{eqn:linear_sys} with matrices and vectors in~\eqref{eqn:general_example_whb}, then we have
			\begin{itemize}
				\item[(i)] System~\eqref{eqn:linear_sys} is globally feasible on $\bm \zeta_0 + \mc K$ if and only if $\bm p_i \preceq_{\mc K^\star} \bm q_j$ for all $i, j\in \{1, 2\}$.
				\item[(ii)] System~\eqref{eqn:linear_sys} is locally linear feasible at $\bzz$ if and only if there exists a vector $\bm z\in \mb R^d$ that satisfies $\bm p_i \preceq_{\mc K^\star} \bm z \preceq_{\mc K^\star} \bm q_j$ for all $i, j\in \{1, 2\}$.
			\end{itemize}
		\end{lemma}
		
		\begin{proof}[Proof of \Cref{lem:HK_CK_condition}]
			Since $\bm W\in \mb R^{4\times 1}$, the scalar $y$ is one-dimensional. Then~\eqref{eqn:linear_sys} is globally feasible if and only if for any $\bm k\in \mc K$, there exists $y(\bm k)$ such that $\bm W y + \bm H\bm k \geq \bm 0$. By construction~\eqref{eqn:general_example_whb}, the above vector inequality is equivalent to that $\bm p_i^\top \bm k \leq y \leq \bm q_j^\top \bm k$, or $\bm p_i^\top \bm k \leq \bm q_j^\top \bm k$ for all $i, j\in \{1, 2\}$. Noting that the inequality is true for all $\bm k\in \mc K$, it follows that~\eqref{eqn:linear_sys} is globally feasible on $\bm \zeta_0 + \mc K$ if and only if $\bm q_j - \bm p_i \in \mc K^\star$ for all $i, j\in \{1, 2\}$. Then, part $(i)$ is a direct consequence of this condition. 
			
			For part $(ii)$, system~\eqref{eqn:linear_sys} is locally linear feasible at $\bzz$ if and only if there exists $\bm z\in \mb R^d, c\in \mb R$ and $\epsilon>0$, such that $\bm W (\bm z^\top\bm \zeta + c) + \bm H\bm \zeta \bm \geq \bm b$ for all $\bm \zeta \in \Xi \cap B(\bzz, \epsilon)$. Substituting $\bz = \bzz$ implies $c=-\bm z^\top \bzz$. Thus, the condition is equivalent to that $\exists \; \bm z$ such that $\bm W \bm z^\top \bm v_{\ell} + \bm H\bm v_{\ell} \geq \bm 0$ for all $\ell\in [L]$, where $\{\bm v_1, \ldots, \bm v_L\}$ are the edge directions of $\Xi$ at $\bzz$. According to the proof of \Cref{thm:no_more_d_edges}, $\{\bm v_1, \ldots, \bm v_L\}$ are also the extreme rays of $\mc K$. So, we can further write the condition as $\exists \; \bm z$ such that the transpose of each row vector of $\bm W \bm z^\top + \bm H$ is in $\mc K^\star$, which implies the desired condition. 
		\end{proof}
		
		Lemma~\ref{lem:HK_CK_condition} reduces the problem to an interpolation question in a partially ordered vector space. Global feasibility requires two lower bounds to lie below two upper bounds. Local linear feasibility is stronger. It requires an intermediate vector between them. The next lemma recalls the classical fact that this interpolation property holds exactly for simplicial cones.
		
		\begin{lemma}[Riesz interpolation property]\label{lem:riesz}
			Let $\mc C\subset \mb R^d$ be any pointed polyhedron cone with full dimension $\dim \mc C = d$, then the following are equivalent.
			\begin{itemize}
				\item[(i)] For every $\bm a_1, \bm a_2, \bm b_1, \bm b_2 \in \mb R^d$ with $\bm a_i \preceq_{\mc C} \bm b_j$ for all $i, j\in \{1, 2\}$, there exists $\bm z\in \mb R^d$ such that $\bm a_i \preceq_{\mc C} \bm z \preceq_{\mc C} \bm b_j$ for all $i, j\in \{1, 2\}$.
				\item[(ii)] The cone $\mc C$ is a simplicial cone; that is, there are exactly $d$ linearly independent edges of $\mc C$ incident to $\bm 0$, or $d$ linearly independent extreme rays of $\mc C$.
			\end{itemize}
		\end{lemma}
		
		\begin{proof}[Proof of \Cref{lem:riesz}]
			Since $\dim \mc C = d$, we have $\mc C - \mc C = \mb R^d$, so $\mc C$ is a generating wedge of $\mb R^d$ (see Definition~1.5 of \citealt{aliprantis2007cones}). Condition~$(i)$ is the Riesz interpolation property (see Definition~1.52 and Lemma~1.53 of \citealt{aliprantis2007cones}). Since the Riesz interpolation property is equivalent to the Riesz decomposition property (Theorem~1.54 of \citealt{aliprantis2007cones}), which in turn is equivalent to $(\mb R^d, \preceq_{\mc C})$ being a Riesz space (Corollary~2.48 of \citealt{aliprantis2007cones}); that is, every pair $\{\bm x, \bm y\} \subset \mb R^d$ has a least upper bound in the partial order $\preceq_{\mc C}$ (see Definition~1.14 of \citealt{aliprantis2007cones}). 
			
			We now show that the condition that $(\mb R^d, \preceq_{\mc C})$ is a Riesz space is equivalent to condition~$(ii)$. If $(ii)$ holds, then $(\mb R^d, \preceq_{\mc C})$ is order isomorphic to $(\mb R^d, \leq)$. Since $(\mb R^d, \leq)$ is a Riesz space, so is $(\mb R^d, \preceq_{\mc C})$, and $(\mb R^d, \preceq_{\mc C})$ is a Riesz space. Conversely, if $(\mb R^d, \preceq_{\mc C})$ is a Riesz space, then $(\mb R^d, \preceq_{\mc C})$ is Archimedean (see Definition~1.10 in \citealt{aliprantis2007cones}). Therefore, $\mc C$ is an Archimedean lattice cone. By Yudin's Theorem (Theorem~3.21 in \citealt{aliprantis2007cones}), $\mc C$ is a Yudin cone (see Definition~3.15 in \citealt{aliprantis2007cones}), and hence a simplicial cone. Thus the condition that $(\mb R^d, \preceq_{\mc C})$ is a Riesz space implies $(ii)$. Therefore, conditions $(i)$ and $(ii)$ are equivalent.
		\end{proof}
		
		We will apply Lemma~\ref{lem:riesz} to the dual cone $\mc K^\star$ of the tangent cone. To transfer the non-simpliciality assumption from $\mc K$ to $\mc K^\star$, we need the following elementary fact. It shows that simpliciality, and hence non-simpliciality, is preserved under polarity for pointed polyhedral cones.
		
		\begin{lemma}\label{lem:kkstar_simplex}
			Let $\mc C\subset \mb R^d$ be a full-dimensional pointed polyhedral cone, then the cone $\mc C$ is a simplicial cone if and only if the dual cone $\mc C^\star$ is a simplicial cone.
		\end{lemma}
		
		\begin{proof}[Proof of \Cref{lem:kkstar_simplex}]
			We first assume that $\mc C$ is a simplicial cone with $\mc C = {\rm cone}\{\bm v_1, \ldots, \bm v_d\}$, where $\{\bm v_1, \ldots, \bm v_d\}$ is linearly independent. Let $\bm V = [\bm v_1, \ldots, \bm v_d]\in \mb R^{d\times d}$ be invertible, then
			\begin{equation*}
				\begin{array}{rl}
					\mc C^\star = \{\bm u\in \mb R^d ~:~ \bm u^\top \bm V\bm \lambda \geq 0 \quad \forall \bm \lambda \geq \bm 0 \} \; = \; & \{\bm u\in \mb R^d ~:~ (\bm V^\top \bm u)^\top\bm \lambda \geq 0 \quad \forall \bm \lambda \geq \bm 0 \}\\
					= \; & \{\bm u\in \mb R^d ~:~ \bm V^\top \bm u \geq \bm 0 \}\\
					= \; & \{\bm u\in \mb R^d ~:~ \exists \; \bm x\geq \bm 0 \text{ such that } \bm u = (\bm V^{\top})^{-1} \bm x \}.
				\end{array}
			\end{equation*}
			Then, let $\bm w_1, \ldots, \bm w_d$ be the column vectors of $(\bm V^\top)^{-1}$, which is linearly independent. It follows that $\mc C^\star = {\rm cone}\{\bm w_1, \ldots, \bm w_d\}$, which proves that $\mc C^\star$ is also a simplicial cone. On the other hand, if $\mc C^\star$ is a simplicial cone, it follows similarly that $\mc C^{\star\star}$ is a simplicial cone. Since $\mc C$ is a polyhedron cone, we have $\mc C = \mc C^{\star\star}$, which implies that $\mc C$ is also a simplicial cone and finishes the proof.
		\end{proof}
		
		We now combine the preceding ingredients. Since the tangent cone at a non-simple vertex is not simplicial, Lemmas~\ref{lem:riesz} and~\ref{lem:kkstar_simplex} imply that the order induced by $\mc K^\star$ fails the interpolation property. The vectors witnessing this failure are then inserted into the recourse-system construction~\eqref{eqn:general_example_whb}.
		
		
		\begin{proof}[Proof of \Cref{thm:general_impossible}]
			W.L.O.G., we assume that ${\rm dim}(\Xi) = d$. Otherwise, ${\rm dim}(\Xi) < d$, and our result applies in $\mb R^q$ with $q = {\rm dim}(\Xi)$, as shown in the proof of \Cref{thm:no_more_d_edges}. Consider the tangent cone $\mc K=\mc T_{\bzz}$, which is a pointed full-dimensional polyhedral cone in $\mb R^d$ with more than $d$ extreme rays since $\bzz$ is a non-simple vertex. Thus, $\mc K$ is not simplicial. By \Cref{lem:kkstar_simplex}, the dual cone $\mc K^\star$ is not simplicial either. Applying \Cref{lem:riesz} to $\mc C=\mc K^\star$, condition $(i)$ in \Cref{lem:riesz} fails. Hence, there exist vectors $\bm p_1,\bm p_2,\bm q_1,\bm q_2\in\mb R^d$ such that
			\[
			\bm p_i \preceq_{\mc K^\star} \bm q_j
			\qquad \forall i,j\in\{1,2\},
			\]
			but there does not exist $\bm z\in\mb R^d$ satisfying
			\[
			\bm p_i \preceq_{\mc K^\star} \bm z \preceq_{\mc K^\star} \bm q_j
			\qquad \forall i,j\in\{1,2\}.
			\]
			We design $\bm W,\bm H$ and $\bm b$ according to~\eqref{eqn:general_example_whb} with these vectors. By \Cref{lem:HK_CK_condition}, the inequalities
			$\bm p_i \preceq_{\mc K^\star} \bm q_j$ for all $i,j\in\{1,2\}$ imply that system~\eqref{eqn:linear_sys} is globally feasible on $\bzz+\mc K$. Since $\Xi\subseteq \bzz+\mc K$, the same system is globally feasible on $\Xi$. In addition, the nonexistence of an interpolating vector $\bm z$ implies that system~\eqref{eqn:linear_sys} is not locally linear feasible at $\bzz$ by \Cref{lem:HK_CK_condition}. Thus, we obtained a recourse system that is globally feasible on $\Xi$ but not locally linear feasible at $\bzz$.   
		\end{proof}

		\section{Proof of Results in \Cref{sec:budgeted_uncertainty_set}}\label{ec_proof_budget}
		
		\subsection{Proof of Results in \Cref{subsec:two_sided}}
		We first present the following lemma that characterizes the vertices of $\Xi$ in \eqref{eqn:two_sided_budget}.
		\begin{lemma}\label{lem:two_sides_chara}
			Let $\Xi$ be as in \eqref{eqn:two_sided_budget} with a budget $\Gamma\in (0, d)$, then a point $\bzz \in \Xi$ is a vertex of $\Xi$ if and only if $\|\bzz\|_1 = \Gamma$ and $|\{j ~:~ 0< |\zeta_{0, j}| <1\}| \leq1$, where $\zeta_{0, j}$ denotes the $j$-th component of $\bzz$.
		\end{lemma}
		\begin{proof}[Proof of \Cref{lem:two_sides_chara}]
			We first prove the necessity. Suppose $\bzz$ is a vertex of $\Xi$. Assume that $\|\bzz\|_1 < \Gamma$. Then there exists some components $\zeta_{0, j}$ of $\bzz$ such that $|\zeta_{0, j}| <1$, otherwise $\|\bzz\|_1 = d > \Gamma$. Pick any real number $t$ such that $0<t<\min\{\Gamma - \|\bzz\|_1, 1-|\zeta_{0, j}|\}$. Then, we claim that $\bzz \pm t\bm e_j \in \Xi$. Indeed, we have $\|\bzz \pm t\bm e_j\|_\infty \leq \max\{\|\bzz\|_{\infty}, |\zeta_{0, j}| + t\} \leq 1$. In addition, we also have $\|\bzz \pm t\bm e_j\|_1 \leq \|\bzz\|_1 + t \leq \Gamma$. Thus, our claim is true. Since $\bzz = (\bzz + t \bm e_j)/2 + (\bzz - t \bm e_j)/2$, this contradicts $\bzz$ being a vertex of $\Xi$. Therefore, it follows that $\|\bzz\|_1 = \Gamma$.
			
			Next, assume for contradiction that $|\{j : 0 < |\zeta_{0, j}| < 1\}| > 1$. Then there exist $i \neq j$ with $0 < |\zeta_{0, i}| < 1$ and $0 < |\zeta_{0, j}| < 1$. We define the vector $\bm w$ and pick the real number $t$ according to
			\begin{equation*}
				\bm w = {\rm sgn}(\zeta_{0, i}) \bm e_i - {\rm sgn}(\zeta_{0, j}) \bm e_j, \;\; \text{ and } \;\; 0 < t < \min\{|\zeta_{0, i}|\;,\; |\zeta_{0, j}|\;,\; 1- |\zeta_{0, i}|\;,\; 1-|\zeta_{0, j}|\}.
			\end{equation*}
			We claim that $\bzz \pm t \bm w \in \Xi$. Indeed,
			\begin{equation*}
				\|\bzz \pm t \bm w\|_\infty \leq \max\{\|\bzz\|_\infty,\, |\zeta_{0, i}| + t,\, |\zeta_{0, j}| + t\} \leq 1, \quad \text{and} \quad \|\bzz \pm t \bm w\|_1 = \|\bzz\|_1 \leq \Gamma,
			\end{equation*}
			where the second equality holds because the perturbation increases $|\zeta_{0, i}|$ by $t$ while decreasing $|\zeta_{0, j}|$ by $t$ (or vice versa). Since $\bzz = (\bzz + t \bm w)/2 + (\bzz - t \bm w)/2$, this again contradicts $\bzz$ being a vertex. Hence, the assumption is false, and it implies $|\{j ~:~ 0< |\zeta_{0, j}| <1\}| \leq 1$.
			
			Then we prove the sufficiency. Suppose $\bzz \in \Xi$ is with $\|\bzz\|_1 = \Gamma$ and $|\{j ~:~ 0< |\zeta_{0, j}| <1\}| \leq 1$. Assume that $\bzz = (\bm u+ \bm v)/2$ for some $\bm u, \bm v\in \Xi$, and $\bm u\neq \bzz$ and $\bm v\neq \bzz$. Then it follows that
			\begin{equation*}
				\Gamma = \|\bzz\|_1 = \left\|\frac{\bm u+ \bm v}{2}\right\|_1 \leq \frac{\|\bm u\|_1}{2} + \frac{\|\bm v\|_1}{2} \leq \Gamma,
			\end{equation*}
			where the last step is due to $\bm u, \bm v\in \Xi$. Thus, all the above inequalities must be equalities, which implies $\|\bm u\|_1 = \|\bm v\|_1 = \Gamma$ and $|\zeta_{0, j}| = (|u_j| + |v_j|)/2$ for all $j\in [d]$. We divide all the coordinates of $\bzz$ into three cases: 
			\begin{itemize}
				\item[(i)] For those $j\in [d]$ with $\zeta_{0, j} = 0$: It follows that $0 = |\zeta_{0, j}| = (|u_j| + |v_j|)/2$. So $u_j = v_j = 0$.
				
				\item[(ii)] For those $j\in [d]$ with $|\zeta_{0, j}| = 1$: We have $1 = |\zeta_{0, j}| = (|u_j| + |v_j|)/2$. As $\bm u, \bm v\in \Xi$, we get $|u_j| \leq 1$ and $|v_j| \leq 1$. So $|u_j| = |v_j| = 1$. In addition, since $\zeta_{0, j} = (u_j + v_j)/2$, we have $u_j = v_j = \zeta_{0, j}$.
				
				\item[(iii)] For those $j\in [d]$ with $0 < |\zeta_{0, j}| < 1$: By assumption, this occurs for at most one coordinate. From the previous two cases, we have $\zeta_{0, i} = u_i = v_i$ for all $i\in [d]$ except $i=j$. Since $\|\bm u\|_1 = \|\bm v\|_1 = \Gamma = \|\bzz\|_1$, it follows that $|\zeta_{0, j}| = |u_j| = |v_j|$. Then $\zeta_{0, j} = (u_j + v_j)/2$ implies $u_j = v_j = \zeta_{0, j}$.
			\end{itemize}
			
			Combining all cases yields $\bm u = \bm v = \bzz$. So the assumption is false, and $\bzz$ is a vertex of $\Xi$.  
		\end{proof}
		
		\begin{lemma}\label{lem:re-arrange}
			Let $d$ and $\Gamma$ be positive integers with $\Gamma < d$. Then for any $\bm w \in \mb R^d_+$ satisfying
			\begin{equation*}
				\sum_{i=1}^{\Gamma} w_i \;\geq\; \sum_{j=\Gamma+1}^{d} w_j,
			\end{equation*}
			there exist nonnegative scalars $\{\mu_{ij}\}_{i \in [\Gamma],\, j \in [d] \backslash [\Gamma]}$ such that
			\begin{equation*}
				\sum_{i=1}^{\Gamma} \mu_{ij} = w_j \quad \forall\, j \in [d] \backslash [\Gamma], \qquad \text{and} \qquad \sum_{j=\Gamma+1}^{d} \mu_{ij} \leq w_i \quad \forall\, i \in [\Gamma].
			\end{equation*}
		\end{lemma}
		
		\begin{proof}[Proof of \Cref{lem:re-arrange}]
			We construct the scalars $\{\mu_{ij}\}$ by the following iterative procedure. Initialize $\mu_{ij} = 0$ for all $i \in [\Gamma]$ and $j \in [d] \backslash [\Gamma]$, and set $(i_1, j_1) = (1, \Gamma+1)$ as the first index pair. At each step $k \geq 1$, given the current index pair $(i_k, j_k)$, we assign
			\begin{equation*}
				\mu_{i_k j_k} \;=\; \min\!\left\{\, w_{i_k} - \sum_{j=\Gamma+1}^{j_k-1} \mu_{i_k j},\;\; w_{j_k} - \sum_{i=1}^{i_k-1} \mu_{i j_k} \,\right\},
			\end{equation*}
			and then advance the index pair according to the rule
			\begin{equation*}
				(i_{k+1}, j_{k+1}) \;=\; 
				\begin{cases}
					(i_k,\, j_k + 1), & \text{if } \mu_{i_k j_k} = w_{j_k} - \sum_{i=1}^{i_k-1} \mu_{i j_k}, \\[4pt]
					(i_k + 1,\, j_k), & \text{otherwise}.
				\end{cases}
			\end{equation*}
			The procedure terminates at some step $K$ when $i_K > \Gamma$ or $j_K > d$. By a straightforward induction, all $\mu_{ij}$ produced are nonnegative, and at every step $k$,
			\begin{equation}\label{eq:induction-property}
				\sum_{j=\Gamma+1}^{d} \mu_{i^\prime j} \;=\; \sum_{j=\Gamma+1}^{j_k-1} \mu_{i^\prime j} \;=\; w_{i^\prime}, \quad \sum_{i=1}^{\Gamma} \mu_{i j^\prime} \;=\; \sum_{i=1}^{i_k-1} \mu_{i j^\prime} \;=\; w_{j^\prime} \quad \forall i^\prime \in [i_k - 1], \; j^\prime \in [j_k - 1] \backslash [\Gamma].
			\end{equation}
			
			We claim that the procedure can only terminate with $j_K > d$; that is, $i_K > \Gamma$ never occurs. Suppose for contradiction that $i_K > \Gamma$ and $j_K \leq d$ at termination. Then $i_{K-1} = \Gamma$ and $j_{K-1} = j_K \leq d$, and
			\begin{equation*}
				R \;\triangleq\; w_{j_K} - \sum_{i=1}^{\Gamma} \mu_{i j_K} - \left( w_\Gamma - \sum_{j=\Gamma+1}^{d-1} \mu_{\Gamma j} \right) \;>\; 0, \qquad \text{and} \qquad \mu_{\Gamma j_K} \;=\; w_\Gamma - \sum_{j=\Gamma+1}^{d-1} \mu_{\Gamma j}.
			\end{equation*}
			Since all $\mu_{ij}$ are nonnegative, this forces $w_{j_K} > 0$. Combined with \eqref{eq:induction-property}, we have
			\begin{equation*}
				\sum_{j=\Gamma+1}^{d} \mu_{i^\prime j} = w_{i^\prime} \quad \forall\, i^\prime \in [\Gamma], \qquad \text{and} \qquad \sum_{i=1}^{\Gamma} \mu_{i j^\prime} = w_{j^\prime} \quad \forall\, j^\prime \in [j_K - 1] \backslash [\Gamma].
			\end{equation*}
			Consequently,
			\begin{equation*}
				\sum_{i^\prime=1}^{\Gamma} w_{i^\prime} \;=\; \sum_{i^\prime=1}^{\Gamma} \sum_{j=\Gamma+1}^{j_K-1} \mu_{i^\prime j} \;=\; \sum_{j=\Gamma+1}^{j_K-1} \sum_{i^\prime=1}^{\Gamma} \mu_{i^\prime j} \;=\; \sum_{j=\Gamma+1}^{j_K-1} w_{j} \;<\; \sum_{j=\Gamma+1}^{d} w_{j} \;\leq\; \sum_{i^\prime=1}^{\Gamma} w_{i^\prime},
			\end{equation*}
			where the first and third equalities follow from \eqref{eq:induction-property}, the second from swapping the order of summation, the strict inequality from $w_{j_K} > 0$, and the final inequality from the hypothesis of the lemma. This contradiction proves the claim.
			
			Thus, the procedure terminates with $i_K \leq \Gamma$ and $j_K > d$, hence $i_{K-1} = i_K$ and $j_{K-1} = d$. By \eqref{eq:induction-property},
			\begin{equation*}
				\sum_{j=\Gamma+1}^{d} \mu_{i^\prime j} = w_{i^\prime} \quad \forall\, i^\prime \in [i_K - 1], \qquad \text{and} \qquad \sum_{i=1}^{\Gamma} \mu_{i j^\prime} = w_{j^\prime} \quad \forall\, j^\prime \in [d] \backslash [\Gamma].
			\end{equation*}
			Moreover, at $i^\prime = i_K$ we have $\sum_{j=\Gamma+1}^{d} \mu_{i_K j} \leq w_{i_K}$, since otherwise the procedure would have advanced to $(i_K + 1, d)$ rather than terminating. Finally, for every $i^\prime \in [\Gamma] \backslash [i_K]$ and $j \in [d] \backslash [\Gamma]$, we have $\mu_{i^\prime j} = 0$ because these entries remain at their initial value throughout the procedure; hence
			\begin{equation*}
				\sum_{j=\Gamma+1}^{d} \mu_{i^\prime j} \;=\; 0 \;\leq\; w_{i^\prime} \quad \forall\, i^\prime \in [\Gamma] \backslash [i_K].
			\end{equation*}
			Combining the three cases yields the desired conclusion. 
		\end{proof}

		\begin{proof}[Proof of \Cref{prop:two_sides_budget_uncertainty_set_integer}]
			We first prove part $(i)$. By \Cref{lem:two_sides_chara}, a point $\bzz\in\Xi$ is a vertex of $\Xi$ if and only if $\|\bzz\|_1=\Gamma$ and $|\{j:0<|\zeta_{0,j}|<1\}|\leq1$. Moreover, exactly one coordinate with absolute value strictly between $0$ and $1$ is impossible in this case, since all other coordinates would have absolute value $0$ or $1$, making $\|\bzz\|_1$ non-integer. Therefore, a point $\bzz\in\mb R^d$ is a vertex of $\Xi$ if and only if exactly $\Gamma$ of its coordinates have absolute value $1$ and the remaining coordinates are zero. There are $\binom{d}{\Gamma}$ ways to choose the nonzero coordinates and $2^\Gamma$ ways to assign their signs. Thus, $\Xi$ has $2^\Gamma\binom{d}{\Gamma}$ vertices in total. This finishes the proof of part $(i)$.
			
			Now we prove part $(ii)$. Fix the vertex $\bzz=\sum_{i\in[\Gamma]}\bm e_i$, and let $S=\{1\}^{\Gamma}\times\{-1,1\}^{d-\Gamma}$ be the set of vectors whose first $\Gamma$ components are $1$ and whose remaining components are $\pm1$. Since $\|\bm\xi\|_1\leq\Gamma$ is equivalent to $\bm s^\top\bm\xi\leq\Gamma$ for all $\bm s\in\{-1,1\}^d$, the active constraints at $\bzz$ are
			\begin{equation*}
				\bm s^\top\bzz=\Gamma, \;\;
				\bm e_i^\top\bzz=1
				\quad \forall\,\bm s\in S,\;i\in[\Gamma].
			\end{equation*}
			Therefore, the tangent cone of $\Xi$ at $\bzz$ is given by
			\begin{equation*}
				\begin{array}{rl}
					\mc T_{\bzz} = &
					\left\{
					\bz \;:\;
					\bm s^\top\bz\leq0,\;
					\bm e_i^\top\bz\leq0
					\quad \forall\,\bm s\in S,\;i\in[\Gamma]
					\right\}\\
					= &
					\left\{
					\bz \;:\;
					\zeta_i\leq0,\;
					\displaystyle\sum_{i\in[\Gamma]}\zeta_i
					+\displaystyle\sum_{j\in[d]\backslash[\Gamma]}|\zeta_j|
					\leq0
					\quad \forall\,i\in[\Gamma]
					\right\},
				\end{array}
			\end{equation*}
			where the second equality follows by maximizing $\bm s^\top\bz$ over the unrestricted signs of $\bm s$ on $[d]\backslash[\Gamma]$.
			
			By the proof of \Cref{thm:no_more_d_edges}, the directions of the edges incident to $\bzz$ correspond exactly to the extreme rays of $\mc T_{\bzz}$. Therefore, to prove part $(ii)$, it suffices to show that
			\begin{equation*}
				\bm v^{i,j,\sigma}=\sigma\bm e_j-\bm e_i
				\quad \forall\,i\in[\Gamma],\;
				j\in[d]\backslash[\Gamma],\;
				\sigma\in\{-1,1\},
			\end{equation*}
			span exactly the extreme rays of $\mc T_{\bzz}$.
			
			First, every $\bm v^{i,j,\sigma}$ belongs to $\mc T_{\bzz}$. Indeed, its coordinates in $[\Gamma]$ are nonpositive, and the expression in the budget inequality equals $-1+|\sigma|=0$. Next, we show that each of these directions spans an extreme ray. Fix $i\in[\Gamma]$ and $j\in[d]\backslash[\Gamma]$. Since $\mc T_{\bzz}$ is invariant under sign changes of the coordinates in $[d]\backslash[\Gamma]$, we may assume without loss of generality that $\sigma=1$. The active constraints at $\bm v^{i,j,1}$ are
			\begin{equation*}
				\bm e_{i^\prime}^\top\bm v^{i,j,1}=0, \;\;
				\bm s^\top\bm v^{i,j,1}=0
				\quad \forall\,i^\prime\in[\Gamma]\backslash\{i\},\;\bm s\in S^\prime,
			\end{equation*}
			where $S^\prime$ consists of the vectors whose first $\Gamma$ components and $j$-th component are $1$, and whose remaining components are $\pm1$.
			
			For every $\ell\in[d]\backslash([\Gamma]\cup\{j\})$, the set $S^\prime$ contains two vectors that differ only in coordinate $\ell$. Their difference is $2\bm e_\ell$, so $\bm e_\ell$ belongs to the span of the active normals. In addition, subtracting the coordinate vectors indexed by $[\Gamma]\backslash\{i\}$ and $[d]\backslash([\Gamma]\cup\{j\})$ from the all-ones vector $\bm e\in S^\prime$ gives $\bm e_i+\bm e_j$. Conversely, every active normal belongs to the span of
			\begin{equation*}
				\left\{\bm e_{i^\prime}:i^\prime\in[\Gamma]\backslash\{i\}\right\}
				\ \cup\
				\left\{\bm e_\ell:
				\ell\in[d]\backslash([\Gamma]\cup\{j\})\right\}
				\ \cup\
				\{\bm e_i+\bm e_j\}.
			\end{equation*}
			These vectors are linearly independent. Indeed, in any linear combination equal to zero, coordinate $i$ first forces the coefficient of $\bm e_i+\bm e_j$ to be zero, after which all remaining coefficients must also be zero. Therefore,
			\begin{equation*}
				{\rm rank}\left(
				\left\{\bm e_{i^\prime}:i^\prime\in[\Gamma]\backslash\{i\}\right\}
				\cup
				\left\{\bm s:\bm s\in S^\prime\right\}
				\right)
				=
				(\Gamma-1)+(d-\Gamma-1)+1
				=
				d-1.
			\end{equation*}
			Hence, $\bm v^{i,j,1}$ spans an extreme ray of $\mc T_{\bzz}$. By the sign symmetry noted above, the same holds for $\bm v^{i,j,-1}$. Since $i$ and $j$ were arbitrary, every proposed direction is an extreme ray.
			
			To show that these are all the extreme rays, it remains to prove
			\begin{equation}\label{eqn:sub_cone}
				\mc T_{\bzz}
				\subseteq
				{\rm cone}\left\{
				\bm v^{i,j,\sigma}:
				i\in[\Gamma],\;
				j\in[d]\backslash[\Gamma],\;
				\sigma\in\{-1,1\}
				\right\}.
			\end{equation}
			Pick any $\bz\in\mc T_{\bzz}$. Then we have
			\begin{equation*}
				-\zeta_i\geq0, \;\;
				\sum_{i\in[\Gamma]}(-\zeta_i)
				\geq
				\sum_{j\in[d]\backslash[\Gamma]}|\zeta_j|
				\quad \forall\,i\in[\Gamma].
			\end{equation*}
			Applying \Cref{lem:re-arrange} with $w_i=-\zeta_i$ for $i\in[\Gamma]$ and $w_j=|\zeta_j|$ for $j\in[d]\backslash[\Gamma]$, we obtain nonnegative scalars $\{\mu_{ij}\}_{i\in[\Gamma],\,j\in[d]\backslash[\Gamma]}$ such that
			\begin{equation*}
				\sum_{i\in[\Gamma]}\mu_{ij}=|\zeta_j|, \;\;
				\sum_{j\in[d]\backslash[\Gamma]}\mu_{ij}\leq-\zeta_i
				\quad \forall\,i\in[\Gamma],\;j\in[d]\backslash[\Gamma].
			\end{equation*}
			For each $j\in[d]\backslash[\Gamma]$, define $\sigma_j=1$ if $\zeta_j\geq0$ and $\sigma_j=-1$ otherwise, so that $\sigma_j|\zeta_j|=\zeta_j$. In addition, define
			\begin{equation*}
				r_i=-\zeta_i-\sum_{j\in[d]\backslash[\Gamma]}\mu_{ij}
				\geq0
				\quad \forall\,i\in[\Gamma].
			\end{equation*}
			It follows that
			\begin{equation*}
				\begin{aligned}
					\sum_{i\in[\Gamma]}
					\sum_{j\in[d]\backslash[\Gamma]}
					\mu_{ij}\bm v^{i,j,\sigma_j}
					&=
					\sum_{j\in[d]\backslash[\Gamma]}
					\sigma_j\left(\sum_{i\in[\Gamma]}\mu_{ij}\right)\bm e_j
					-
					\sum_{i\in[\Gamma]}
					\left(\sum_{j\in[d]\backslash[\Gamma]}\mu_{ij}\right)\bm e_i\\
					&=
					\sum_{j\in[d]\backslash[\Gamma]}\zeta_j\bm e_j
					+
					\sum_{i\in[\Gamma]}(\zeta_i+r_i)\bm e_i\\
					&=\bz+\sum_{i\in[\Gamma]}r_i\bm e_i.
				\end{aligned}
			\end{equation*}
			Therefore,
			\begin{equation*}
				\bz
				=
				\sum_{i\in[\Gamma]}
				\sum_{j\in[d]\backslash[\Gamma]}
				\mu_{ij}\bm v^{i,j,\sigma_j}
				+
				\sum_{i\in[\Gamma]}r_i(-\bm e_i).
			\end{equation*}
			It remains to express each $-\bm e_i$ as a nonnegative combination of the proposed directions. Since $\Gamma<d$, we can fix an index $j_0\in[d]\backslash[\Gamma]$. For every $i\in[\Gamma]$, we have
			\begin{equation*}
				-\bm e_i
				=
				\frac12(\bm e_{j_0}-\bm e_i)
				+
				\frac12(-\bm e_{j_0}-\bm e_i)
				=
				\frac12\bm v^{i,j_0,1}
				+
				\frac12\bm v^{i,j_0,-1}.
			\end{equation*}
			Substituting this identity into the preceding expression yields
			\begin{equation*}
				\bz
				=
				\sum_{i\in[\Gamma]}
				\sum_{j\in[d]\backslash[\Gamma]}
				\mu_{ij}\bm v^{i,j,\sigma_j}
				+
				\sum_{i\in[\Gamma]}
				\frac{r_i}{2}
				\left(\bm v^{i,j_0,1}+\bm v^{i,j_0,-1}\right).
			\end{equation*}
			All coefficients in this expression are nonnegative. Hence, $\bz$ belongs to the cone on the right-hand side of \eqref{eqn:sub_cone}. Since $\bz\in\mc T_{\bzz}$ was arbitrary, this proves \eqref{eqn:sub_cone}.
			
			Combining this inclusion with the fact that every proposed direction belongs to $\mc T_{\bzz}$, we obtain
			\begin{equation*}
				\mc T_{\bzz}
				=
				{\rm cone}\left\{
				\bm v^{i,j,\sigma}:
				i\in[\Gamma],\;
				j\in[d]\backslash[\Gamma],\;
				\sigma\in\{-1,1\}
				\right\}.
			\end{equation*}
			Thus, these directions are exactly the extreme rays of $\mc T_{\bzz}$ and therefore give exactly the directions of the edges incident to $\bzz$. Therefore, the number of edges incident to $\bzz$ is $L=2\Gamma(d-\Gamma)$. 
		\end{proof}
		
		\begin{lemma}\label{lem:two_sides_float_cone_zero}
			Let $\Gamma\in(0,d-1)$ be non-integer, and let $\uga$ denote its integer part. Consider the cone
			\begin{equation}\label{eqn:two_sides_float_cone_zero}
				\mc C=\left\{\bz\in\mb R^d:\;
				\zeta_i\leq0\ \ \forall\,i\in[\uga],\quad
				\sum_{i\in[\uga]}\zeta_i+\zeta_{\uga+1}+
				\sum_{j\in [d]\backslash[\uga+1]}|\zeta_j|\leq0\right\}.
			\end{equation}
			Then the extreme rays of $\mc C$ are spanned exactly by the following distinct directions:
			\begin{equation*}
				\bm v^i=\bm e_{\uga+1}-\bm e_i,\quad
				\bm v^{j,\sigma}=\sigma\bm e_j-\bm e_{\uga+1}
				\quad\forall\,i\in[\uga],\;j\in [d]\backslash[\uga+1],\;\sigma\in\{-1,1\}.
			\end{equation*}
			Therefore, the cone $\mc C$ has $2d-\uga-2$ extreme rays.
		\end{lemma}
		
		\begin{proof}[Proof of \Cref{lem:two_sides_float_cone_zero}]
			The cone $\mc C$ is pointed: if $\pm\bz\in\mc C$, the coordinate inequalities force $\zeta_i=0$ for $i\in[\uga]$, and the two budget inequalities then force all remaining coordinates to vanish. Every proposed direction belongs to $\mc C$ by direct substitution. To establish extremality, write the last inequality in \eqref{eqn:two_sides_float_cone_zero} as $\bm s^\top\bz\leq0$ for all $\bm s\in S$, where $S=\{1\}^{\uga+1}\times\{-1,1\}^{d-\uga-1}$.
			
			Fix $i\in[\uga]$. At $\bm v^i$, the active normals in $\mc C$ are $\bm e_{i^\prime}$ for $i^\prime\in[\uga]\backslash\{i\}$, and all $\bm s\in S$ because $\bm s^\top\bm v^i=1-1=0$. For each $\ell\in [d]\backslash[\uga+1]$, two vectors in $S$ differing only in coordinate $\ell$ have difference $2\bm e_\ell$. Therefore, $\bm e_\ell$ belongs to the span of the active normals. Moreover, subtracting the coordinate vectors indexed by $[\uga]\backslash\{i\}$ and $[d]\backslash[\uga+1]$ from the all-ones vector $\bm e\in S$ gives $\bm e_i+\bm e_{\uga+1}$. Conversely, every active normal belongs to the span of
			\begin{equation*}
				\left\{\bm e_{i^\prime}:i^\prime\in[\uga]\backslash\{i\}\right\}
				\ \cup\ 
				\left\{\bm e_\ell:\ell\in [d]\backslash[\uga+1]\right\}
				\ \cup\ 
				\{\bm e_i+\bm e_{\uga+1}\}.
			\end{equation*}
			These vectors are linearly independent, so the active normals have rank $(\uga-1)+(d-\uga-1)+1=d-1$. Hence, $\bm v^i$ spans an extreme ray.
			
			Now fix $j\in [d]\backslash[\uga+1]$ and $\sigma\in\{-1,1\}$. At $\bm v^{j,\sigma}$, the active normals in $\mc C$ are $\bm e_i$, $i\in[\uga]$, and all $\bm s\in S$ with $s_j=\sigma$. As above, differences of these normals yield $\bm e_\ell$ for $\ell\in[d]\backslash([\uga+1]\cup\{j\})$; subtracting their coordinate components leaves $\bm e_{\uga+1}+\sigma\bm e_j$. Conversely, every active normal belongs to the span of
			\begin{equation*}
				\left\{\bm e_i:i\in[\uga]\right\}
				\ \cup\ 
				\left\{\bm e_\ell:\ell\in [d]\backslash([\uga+1]\cup\{j\})\right\}
				\ \cup\ 
				\{\bm e_{\uga+1}+\sigma\bm e_j\}.
			\end{equation*}
			These vectors are linearly independent and number $\uga+(d-\uga-2)+1=d-1$. Thus, $\bm v^{j,\sigma}$ also spans an extreme ray.
			
			To show that there are no other extreme rays, we prove that every vector in $\mc C$ is a nonnegative combination of the proposed directions. Pick any $\bz\in\mc C$. For each $j\in[d]\backslash[\uga+1]$, define $\sigma_j=1$ if $\zeta_j\geq0$ and $\sigma_j=-1$ otherwise, so that $\sigma_j|\zeta_j|=\zeta_j$. Let
			\begin{equation*}
				r=-\sum_{i\in[\uga]}\zeta_i-\zeta_{\uga+1}
				-\sum_{j\in[d]\backslash[\uga+1]}|\zeta_j|\geq0,
			\end{equation*}
			where nonnegativity follows from the last inequality defining $\mc C$. Then
			\begin{equation*}
				\begin{array}{rl}
					\displaystyle\sum_{i\in[\uga]}(-\zeta_i)\bm v^i
					+\sum_{j\in[d]\backslash[\uga+1]}|\zeta_j|\bm v^{j,\sigma_j}= & \displaystyle\sum_{i\in[\uga]}\zeta_i\bm e_i
					+\sum_{j\in[d]\backslash[\uga+1]}\zeta_j\bm e_j
					+\left(-\sum_{i\in[\uga]}\zeta_i
					-\sum_{j\in[d]\backslash[\uga+1]}|\zeta_j|\right)\bm e_{\uga+1}\\
					=&\bz+r\bm e_{\uga+1}.
				\end{array}
			\end{equation*}
			
			Since $[d]\backslash[\uga+1]\neq\emptyset$, we may fix $j_0\in [d]\backslash[\uga+1]$. Then
			\begin{equation*}
				-\bm e_{\uga+1}
				=\frac12(\bm e_{j_0}-\bm e_{\uga+1})
				+\frac12(-\bm e_{j_0}-\bm e_{\uga+1})
				=\frac12\bm v^{j_0,1}+\frac12\bm v^{j_0,-1}.
			\end{equation*}
			Consequently,
			\begin{equation*}
				\bz=
				\sum_{i\in[\uga]}(-\zeta_i)\bm v^i
				+\sum_{j\in [d]\backslash[\uga+1]}|\zeta_j|\bm v^{j,\sigma_j}
				+\frac r2\bm v^{j_0,1}+\frac r2\bm v^{j_0,-1}.
			\end{equation*}
			Therefore, we get that $\bm \zeta$ is in the cone generated by the proposed directions, and
			\begin{equation*}
				\mc C={\rm cone}\left(
				\{\bm v^i:i\in[\uga]\}
				\cup
				\{\bm v^{j,\sigma}:j\in [d]\backslash[\uga+1],\;\sigma\in\{-1,1\}\}
				\right).
			\end{equation*}
			Since the above directions are pairwise distinct, it follows $\mc C$ has $2d-\uga-2$ extreme rays. 
		\end{proof}
		
		\begin{lemma}\label{lem:two_sides_float_cone_full}
			Let $d\geq1$ be an integer, and consider the cone
			\begin{equation}\label{eqn:two_sides_float_cone_full}
				\mc D=\left\{\bz\in\mb R^d:\;
				\zeta_i\leq0\ \ \forall\,i\in[d-1],\quad
				\sum_{i\in[d]}\zeta_i\leq0\right\}.
			\end{equation}
			Then the extreme rays of $\mc D$ are spanned exactly by the following distinct directions:
			\begin{equation*}
				\bm v^0=-\bm e_d,\qquad
				\bm v^i=\bm e_d-\bm e_i
				\quad\forall\,i\in[d-1].
			\end{equation*}
			Therefore, the cone $\mc D$ has exactly $d$ extreme rays.
		\end{lemma}
		
		\begin{proof}[Proof of \Cref{lem:two_sides_float_cone_full}]
			Every proposed direction belongs to $\mc D$. Conversely, each $\bz\in\mc D$ admits the representation
			\begin{equation*}
				\bz=\left(-\sum_{i\in[d]}\zeta_i\right)\bm v^0
				+\sum_{i\in[d-1]}(-\zeta_i)\bm v^i,
			\end{equation*}
			whose coefficients are nonnegative by the inequalities that define $\mc D$. Hence, $\mc D={\rm cone}\{\bm v^0,\bm v^1,\ldots,\bm v^{d-1}\}$. These $d$ generators are linearly independent: coordinates $1,\ldots,d-1$ determine the coefficients of $\bm v^1,\ldots,\bm v^{d-1}$, and coordinate $d$ then determines that of $\bm v^0$. Therefore, they span exactly the $d$ extreme rays of $\mc D$. 
		\end{proof}
		
		\begin{proof}[Proof of \Cref{prop:two_sides_budget_uncertainty_set_float}]
			We first prove part $(i)$. By \Cref{lem:two_sides_chara}, a point $\bzz\in\Xi$ is a vertex if and only if $\|\bzz\|_1=\Gamma$ and at most one coordinate has absolute value strictly between $0$ and $1$. Since $\Gamma$ is non-integer, exactly one such coordinate must be present. The equality $\|\bzz\|_1=\Gamma$ then forces exactly $\uga$ coordinates to have absolute value $1$ and one coordinate to have absolute value $\Gamma-\uga$. Conversely, every point in this form satisfies the conditions in \Cref{lem:two_sides_chara} and is therefore a vertex. The total number of vertices is thus $\binom{d}{\uga}(d-\uga)2^{\uga+1}$ vertices. This proves part $(i)$.
			
			Now we prove part $(ii)$. Fix
			\begin{equation*}
				\bzz=\sum_{i\in[\uga]}\bm e_i+(\Gamma-\uga)\bm e_{\uga+1},
				\qquad
				S=\{1\}^{\uga+1}\times\{-1,1\}^{d-\uga-1}.
			\end{equation*}
			In the representation $\bm s^\top\bm\xi\leq\Gamma$ for all $\bm s\in\{-1,1\}^d$ of the constraint $\|\bm \zeta\|_1 \leq \Gamma$, equality at $\bzz$ holds exactly when $s_i=1$ for all $i\in[\uga+1]$. The active normals are therefore all the vectors $\bm s\in S$, and $\bm e_i$ for all $i\in[\uga]$.
			Therefore, the tangent cone at $\bzz$ is
			\begin{equation*}
				\begin{aligned}
					\mc T_{\bzz}
					&=\left\{\bz\in\mb R^d:\;
					\bm e_i^\top\bz\leq0\ \ \forall\,i\in[\uga],\quad
					\bm s^\top\bz\leq0\ \ \forall\,\bm s\in S\right\}\\
					&=\left\{\bz\in\mb R^d:\;
					\zeta_i\leq0\ \ \forall\,i\in[\uga],\quad
					\sum_{i\in[\uga]}\zeta_i+\zeta_{\uga+1}
					+\sum_{j\in[d]\backslash[\uga+1]}|\zeta_j|\leq0\right\}.
				\end{aligned}
			\end{equation*}
			The second equality follows by maximizing over the unrestricted signs on $[d]\backslash[\uga+1]$. By the proof of \Cref{thm:no_more_d_edges}, the directions of the edges incident to $\bzz$ correspond exactly to the extreme rays of $\mc T_{\bzz}$. We distinguish the following two cases.
			
			\begin{itemize}
				\item[$(a)$] Suppose that $\uga<\Gamma<d-1$. Then $0\leq\uga\leq d-2$, so $[d]\backslash[\uga+1]$ is nonempty. Applying \Cref{lem:two_sides_float_cone_zero} shows that the extreme rays of $\mc T_{\bzz}$ are spanned exactly by
				\begin{equation*}
					\bm v^i=\bm e_{\uga+1}-\bm e_i,\qquad
					\bm v^{j,\sigma}=\sigma\bm e_j-\bm e_{\uga+1}
					\quad\forall\,i\in[\uga],\;j\in[d]\backslash[\uga+1],\;
					\sigma\in\{-1,1\}.
				\end{equation*}
				These $L=2d-\uga-2$ vectors are therefore exactly the directions of the edges incident to $\bzz$.
				
				\item[$(b)$] Suppose that $d-1<\Gamma<d$. Then $\uga=d-1$, and there are no zero coordinates at $\bzz$. In particular, $S$ consists only of the all-ones vector, and the tangent cone reduces to
				\begin{equation*}
					\mc T_{\bzz}=\left\{\bz\in\mb R^d:\;
					\zeta_i\leq0\ \ \forall\,i\in[d-1],\quad
					\sum_{i\in[d]}\zeta_i\leq0\right\}.
				\end{equation*}
				By \Cref{lem:two_sides_float_cone_full}, its extreme rays are spanned exactly by
				\begin{equation*}
					\bm v^0=-\bm e_d,\qquad
					\bm v^i=\bm e_d-\bm e_i
					\quad\forall\,i\in[d-1]=[\uga].
				\end{equation*}
				These $L=d$ vectors are therefore exactly the directions of the edges incident to $\bzz$.
			\end{itemize}
			The two cases exhaust all non-integer budgets in $(0,d)$, completing the proof of part $(ii)$. 
		\end{proof}
		
		\subsection{Proof of Results in \Cref{subsec:one_sided}}

		We first characterize the vertices of the one-sided budgeted uncertainty set. 
		
		\begin{lemma}\label{lem:one_sides_chara}
			Let $\Xi$ be as in \eqref{eqn:one_sided_budget} with $\Gamma\in(0,d)$. A point $\bzz\in\Xi$ is a vertex if and only if either all its coordinates belong to $\{0,1\}$, or exactly one coordinate lies strictly between $0$ and $1$ and $\bm e^\top\bzz=\Gamma$.
		\end{lemma}
		
		\begin{proof}[Proof of \Cref{lem:one_sides_chara}]
			At $\bzz$, the active normals of the box constraints are $-\bm e_j$ for coordinates equal to $0$ and $\bm e_j$ for coordinates equal to $1$. Their rank is
			\begin{equation*}
				d-\left|\{j\in[d]:0<\zeta_{0,j}<1\}\right|.
			\end{equation*}
			The only additional possible active normal is $\bm e$, which is present precisely when $\bm e^\top\bzz=\Gamma$. If there is a coordinate strictly between $0$ and $1$, then $\bm e$ is independent of the active box normals, since they all vanish on that coordinate. Thus, the active normals have rank $d$ exactly in the two cases stated in the lemma. The conclusion follows from the active-constraint characterization of vertices. 
		\end{proof}
		
		The following lemma will be used for vertices at which the integer budget is attained.
		
		\begin{lemma}\label{lem:one_sides_integer_cone}
			Let $\Gamma\in[d-1]$ be an integer. The extreme rays of the cone
			\begin{equation}\label{eqn:one_sides_integer_cone}
				\mc C=\left\{\bz\in\mb R^d:\;
				\zeta_j\geq0\ \ \forall\,j\in[d]\backslash[\Gamma],\quad
				\zeta_i\leq0\ \ \forall\,i\in[\Gamma],\quad
				\bm e^\top\bz\leq0\right\}
			\end{equation}
			are spanned exactly by
			\begin{equation*}
				\bm v^{i,j}=\bm e_j-\bm e_i,\qquad
				\bm u^i=-\bm e_i
				\quad\forall\,i\in[\Gamma],\;j\in[d]\backslash[\Gamma].
			\end{equation*}
			In particular, $\mc C$ has $\Gamma(d-\Gamma)+\Gamma$ extreme rays.
		\end{lemma}
		
		\begin{proof}[Proof of \Cref{lem:one_sides_integer_cone}]
			The cone $\mc C$ is pointed because its coordinate inequalities restrict it to an orthant. Every proposed direction belongs to $\mc C$ by direct substitution.
			
			At $\bm u^i$, precisely the coordinate inequalities other than the $i$-th are active, and their normals have rank $d-1$. At $\bm v^{i,j}$, the active normals are the signed coordinate vectors indexed by $[d]\backslash\{i,j\}$ together with $\bm e$. These have rank $(d-2)+1=d-1$, since $\bm e$ is nonzero on coordinates $i$ and $j$. Hence, every proposed direction spans an extreme ray of $\mc C$.
			
			To show that these are all the extreme rays, pick any $\bz\in\mc C$. Its defining inequalities give
			\begin{equation*}
				-\zeta_i\geq0\quad\forall\,i\in[\Gamma],\qquad
				\zeta_j\geq0\quad\forall\,j\in[d]\backslash[\Gamma],\qquad
				\sum_{i\in[\Gamma]}(-\zeta_i)
				\geq\sum_{j\in[d]\backslash[\Gamma]}\zeta_j.
			\end{equation*}
			By \Cref{lem:re-arrange}, there exist nonnegative scalars $\mu_{ij}$ such that
			\begin{equation*}
				\sum_{i\in[\Gamma]}\mu_{ij}=\zeta_j
				\quad\forall\,j\in[d]\backslash[\Gamma],\qquad
				\sum_{j\in[d]\backslash[\Gamma]}\mu_{ij}\leq-\zeta_i
				\quad\forall\,i\in[\Gamma].
			\end{equation*}
			Consequently,
			\begin{equation*}
				\bz=
				\sum_{i\in[\Gamma]}\sum_{j\in[d]\backslash[\Gamma]}
				\mu_{ij}\bm v^{i,j}
				+\sum_{i\in[\Gamma]}
				\left(-\zeta_i-\sum_{j\in[d]\backslash[\Gamma]}\mu_{ij}\right)\bm u^i.
			\end{equation*}
			Since all the above coefficients are nonnegative, $\bm\zeta$ is in the cone generated by the proposed directions, and these directions are $\Gamma(d-\Gamma)+\Gamma$ extreme rays of $\mc C$. 
		\end{proof}

		\begin{proof}[Proof of \Cref{prop:one_sides_budget_uncertainty_set_integer}]
			We first prove part $(i)$. By \Cref{lem:one_sides_chara}, every vertex either has all coordinates in $\{0,1\}$ or has exactly one fractional coordinate and satisfies $\bm e^\top\bzz=\Gamma$. The latter case is impossible when $\Gamma$ is an integer, since the coordinate sum would be non-integer. Thus, the vertices are exactly the feasible $0/1$ vectors, namely those with at most $\Gamma$ coordinates equal to $1$. Choosing the coordinates equal to $1$ gives $\sum_{k=0}^{\Gamma}\binom{d}{k}$ vertices.
			
			Now we prove part $(ii)$. Fix $\bzz=\sum_{i\in[K]}\bm e_i$, where $0\leq K\leq\Gamma$. The active normals of box constraints at $\bzz$ are $\bm e_j$ for $j\in[K]$ and $-\bm e_i$ for $i\in[d]\backslash[K]$. The budget constraint is active if and only if $K=\Gamma$. By the proof of \Cref{thm:no_more_d_edges}, the incident edge directions correspond exactly to the extreme rays of the tangent cone. We distinguish the two cases.
			
			\begin{itemize}
				\item[$(a)$] If $K<\Gamma$, the budget constraint is inactive, so
				\begin{equation*}
					\mc T_{\bzz}=\left\{\bz\in\mb R^d:\;
					\zeta_j\geq0\ \ \forall\,j\in[d]\backslash[K],\quad
					\zeta_i\leq0\ \ \forall\,i\in[K]\right\}.
				\end{equation*}
				This orthant has exactly the $d$ extreme-ray directions
				\begin{equation*}
					\bm v^j=\bm e_j,\qquad \bm u^i=-\bm e_i
					\quad\forall\,i\in[K],\;j\in[d]\backslash[K].
				\end{equation*}
				Hence, these vectors are the incident edge directions and $L=d$.
				
				\item[$(b)$] If $K=\Gamma$, the active budget constraint additionally imposes $\bm e^\top\bz\leq0$ in the tangent cone, so $\mc T_{\bzz}$ is the cone in \eqref{eqn:one_sides_integer_cone}. By \Cref{lem:one_sides_integer_cone}, the incident edge directions are exactly
				\begin{equation*}
					\bm v^{i,j}=\bm e_j-\bm e_i,\qquad \bm u^i=-\bm e_i
					\quad\forall\,i\in[\Gamma],\;j\in[d]\backslash[\Gamma],
				\end{equation*}
				and their number is $L=\Gamma(d-\Gamma)+\Gamma$.
			\end{itemize}
			This completes the proof of part $(ii)$. 
		\end{proof}
		
		\begin{proof}[Proof of \Cref{prop:one_sides_budget_uncertainty_set_float}]
			We first prove part $(i)$. By \Cref{lem:one_sides_chara}, vertices with no fractional coordinate are precisely the feasible $0/1$ vectors, which have at most $\uga$ coordinates equal to $1$. Any remaining vertex has exactly one fractional coordinate and satisfies $\bm e^\top\bzz=\Gamma$. Since there is a unique coordinate lies strictly between $0$ and $1$, the number of coordinates that equal to $1$ must be $\uga$, and the fractional coordinate must equal $\Gamma-\uga$. Conversely, every point in this form is a vertex by \Cref{lem:one_sides_chara}. The two cases contain $\sum_{k=0}^{\uga}\binom{d}{k}$ and $\binom{d}{\uga}(d-\uga)$ vertices.
			
			For part $(ii)$, fix $\bzz=\sum_{i\in[K]}\bm e_i$ with $0\leq K\leq\uga$. Since $K\leq\uga<\Gamma$, the budget constraint is inactive. Thus, exactly as in part $(ii)(a)$ of the proof of \Cref{prop:one_sides_budget_uncertainty_set_integer}, the tangent cone is the orthant whose extreme-ray directions are
			\begin{equation*}
				\bm v^j=\bm e_j,\qquad \bm u^i=-\bm e_i
				\quad\forall\,i\in[K],\;j\in[d]\backslash[K].
			\end{equation*}
			These are therefore exactly the incident edge directions, and $L=d$.
			
			Finally, we prove part $(iii)$. Fix
			\begin{equation*}
				\bzz=\sum_{i\in[\uga]}\bm e_i+(\Gamma-\uga)\bm e_{\uga+1}.
			\end{equation*}
			The active normals of $\bzz$ for the box constraints are $\bm e_j$, $j\in[\uga]$, and $-\bm e_i$, $i\in[d]\backslash[\uga+1]$. Neither box bound is active on coordinate $\uga+1$, whereas the budget constraint is active. Hence,
			\begin{equation*}
				\mc T_{\bzz}=\left\{\bz\in\mb R^d:\;
				\zeta_j\geq0\ \ \forall\,j\in[d]\backslash[\uga+1],\quad
				\zeta_i\leq0\ \ \forall\,i\in[\uga],\quad
				\bm e^\top\bz\leq0\right\}.
			\end{equation*}
			Every proposed direction $-\bm e_{\uga+1}$, $\bm v^j=\bm e_j-\bm e_{\uga+1}$, and $\bm u^i=\bm e_{\uga+1}-\bm e_i$ belongs to this cone. Conversely, every $\bz\in\mc T_{\bzz}$ satisfies
			\begin{equation*}
				\bz=
				\left(-\sum_{\ell\in[d]}\zeta_\ell\right)(-\bm e_{\uga+1})
				+\sum_{j\in[d]\backslash[\uga+1]}\zeta_j\bm v^j
				+\sum_{i\in[\uga]}(-\zeta_i)\bm u^i.
			\end{equation*}
			Since all the above coefficients are nonnegative, $\bm\zeta$ is in the cone generated by the proposed directions, and these directions are $1+(d-\uga-1)+\uga=d$ extreme rays of $\mc T_{\bzz}$. This completes the proof. 
		\end{proof}

		\section{Proof of Results in \Cref{sec:algorithm}}\label{ec_proof_algorithm}
		
		\begin{proof}[Proof of \Cref{lem:llf_technical_condition}]
			By definition, system~\eqref{eqn:linear_sys} is locally linear feasible at $\bzz$ if and only if 
			\begin{equation*}
				\bm W \left(\bm Y\bz + \bm y\right) + \bm H\bz \geq \bm b \quad \forall\bz \in \Xi\cap B(\bzz, \epsilon)
			\end{equation*}
			for some $\bm Y$, $\bm y$, and $\epsilon>0$. Noting that there always exists $0 < \epsilon_1 < \epsilon_2$ with
			\begin{equation*}
				\bzz + \left\{\sum_{\ell\in [L]} \lambda_{\ell} \bm v_{\ell} ~:~  \bm \lambda \geq \bm 0, \; \|\bm \lambda\| \leq \epsilon_1\right\} \; \subset  \; \Xi\cap B(\bzz, \epsilon) \; \subset \; \bzz + \left\{\sum_{\ell\in [L]} \lambda_{\ell} \bm v_{\ell} ~:~  \bm \lambda \geq \bm 0, \; \|\bm \lambda\| \leq \epsilon_2\right\},
			\end{equation*}
			it follows system~\eqref{eqn:linear_sys} is locally linear feasible at $\bzz$ if and only if 
			\begin{equation*}
				\bm W \left(\bm Y\left(\bzz + \sum_{\ell\in [L]} \lambda_{\ell} \bm v_{\ell}\right) + \bm y\right) + \bm H\left(\bzz + \sum_{\ell\in [L]} \lambda_{\ell} \bm v_{\ell}\right) \geq \bm b \quad \forall \bm \lambda \in \mb R^L_+ \cap B(\bm 0, \epsilon).
			\end{equation*}
			Taking $\bm \lambda$ as the zero vector and the natural basis of $\mb R^L_+$ with scale $\epsilon/2$ implies
			\begin{equation*}
				\bm W \left(\bm Y\left(\bzz + \frac{\epsilon}{2} \bm v_{\ell}\right) + \bm y\right) + \bm H\left(\bzz + \frac{\epsilon}{2} \bm v_{\ell}\right) \geq \bm b, \quad \text{ and } \quad \bm W\left(\bm Y\bzz + \bm y\right) + \bm H\bzz \geq \bm b \quad \forall \ell\in [L].
			\end{equation*}
			Therefore, it follows that we have to find $\bm Y$ and $\bm y$ such that for sufficiently small $\epsilon>0$, it holds that
			\begin{equation*}
				\bm W\left(\bm Y\bzz + \bm y\right) + \bm H\bzz + \frac{\epsilon}{2}\left(\bm W\bm Y\bm v_{\ell} +\bm H\bm v_{\ell}\right) \geq \bm b, \quad \text{ and } \quad \bm W\left(\bm Y\bzz + \bm y\right) + \bm H\bzz \geq \bm b \quad \forall \ell\in [L].
			\end{equation*}
			To analyze the above inequalities, we shall consider them component-wisely. For each row $i\in [m]$, if $(\bm Y, \bm y)$ is with $\bm W_i(\bm Y\bzz + \bm y) + \bm H_i\bzz > b_i$, then another inequality follows directly for sufficiently small $\epsilon>0$ no matter with the sign of $\bm W_i \bm Y\bm v_{\ell} + \bm H_i \bm v_{\ell}$. On the other hand, if $(\bm Y, \bm y)$ is with $\bm W_i(\bm Y\bzz + \bm y) + \bm H_i\bzz = b_i$, then we must request $\bm W_i \bm Y\bm v_{\ell} + \bm H_i \bm v_{\ell} \geq 0$ so that another inequality holds for small $\epsilon>0$. Therefore, the pair $(\bm Y, \bm y)$ should belong to $\mc F_i$ for each $i\in [m]$, to ensure that the locally linear feasibility holds, which is equivalent to $\mc F\neq \emptyset$ and finishes the proof. 
		\end{proof}

		\begin{proof}[Proof of \Cref{prop:convex_mcf}]
			For any $(\bm Y^1, \bm y^1), (\bm Y^2, \bm y^2) \in \mc F_i$, $\lambda\in (0, 1)$ and $i\in[m]$, we prove that $(\lambda \bm Y^1 + (1-\lambda) \bm Y^2, \lambda \bm y^1 + (1-\lambda) \bm y^2) \in \mc F_i$ by verifying it for different cases. If $R_i(\bm Y^1, \bm y^1) > b_i$ or $R_i(\bm Y^2, \bm y^2) > b_i$, then it follows $R_i(\lambda \bm Y^1 + (1-\lambda) \bm Y^2, \lambda \bm y^1 + (1-\lambda) \bm y^2) > b_i$ from $\lambda \in (0, 1)$. If both $R_i(\bm Y^1, \bm y^1) = b_i$ and $R_i(\bm Y^2, \bm y^2) = b_i$, then it must have $\bm W_i\bm Y^j\bm v_{\ell} + \bm H_i \bm v_{\ell} \geq 0$ for all $j\in \{1, 2\}$ and $\ell\in [L]$. Since both $R_i$ and $\bm W_i\bm Y\bm v_{\ell} + \bm H_i \bm v_{\ell}$ are affine functions in $(\bm Y, \bm y)$, it follows that $R_i(\lambda \bm Y^1 + (1-\lambda) \bm Y^2, \lambda \bm y^1 + (1-\lambda) \bm y^2) = b_i$ and $\bm W_i(\lambda \bm Y^1 + (1-\lambda) \bm Y^2)\bm v_{\ell} + \bm H_i \bm v_{\ell} \geq 0$ for all $\ell\in [L]$. In both cases we verified that $\lambda (\bm Y^1, \bm y^1) + (1-\lambda)(\bm Y^2, \bm y^2)\in \mc F_i$, which implies that $\mc F_i$ is a convex set. 
		\end{proof}
		
		\begin{proof}[Proof of \Cref{prop:sandwich}]
			To prove part $(i)$, we fix any $(\bm Y, \bm y)\in \mc F$, and show that $(\bm Y, \bm y)\in \mcge$ by checking every defining constraint of $\mcge$ in \Cref{def:valid_approximation}. Specifically, for all $i\in \mc E$, it follows that $R_i(\bm Y, \bm y) = b_i$ from the definition of validity. Noting that $(\bm Y, \bm y) \in \mc F \subset \mc F_i$ and $R_i(\bm Y, \bm y) = b_i$, the definition~\eqref{eqn:mcf} of $\mc F_i$ implies that $\bm W_i\bm Y\bm v_{\ell} + \bm H_i \bm v_{\ell} \geq 0\;\; \forall \ell\in [L]$. On the other hand, for all $i\notin \mc E$, we still have $(\bm Y, \bm y) \in \mc F \subset \mc F_i$. By definition of $\mc F_i$ in~\eqref{eqn:mcf}, it follows that $R_i(\bm Y, \bm y) \geq b_i$. Therefore, every defining constraint of $\mcge$ in \Cref{def:valid_approximation} is verified, and it follows that $(\bm Y, \bm y)\in \mcge$.
			
			To show part $(ii)$, it suffices to prove that any point $(\bm Y, \bm y)$ that satisfies the claimed condition belongs to $\mc F_i$ for all $i\in [m]$. For $i\in \mc E$, the definition of $\mcge$ gives $R_i(\bm Y, \bm y) = b_i$ and $\bm W_i\bm Y \bm v_{\ell} + \bm H_i \bm v_{\ell} \geq 0$, so $\Yy \in \mc F_i$. If $i\notin \mc E$, then the claimed condition implies that $R_i(\bm Y, \bm y) > b_i$, which also verifies $\Yy\in \mc F_i$. Therefore, the point $\Yy$ is in $\mc F_i$ for all $i\in [m]$, and is also in $\mc F$. 
		\end{proof}

		\begin{proof}[Proof of \Cref{thm:alg_solve_llf}]
			We divide our proof into three parts. First, we show that the index set $\mc E$ within each while loop iteration is always valid. Second, we prove that at Line~\ref{line:find_new_tight}, the new tight set $\mc N$ is non-empty. With the above results, we verify the contents of the theorem in the last part.
			
			We first check the validity of $\mc E$ using induction. Initially, $\mc E = \emptyset$. The condition $i\in \mc E$ is never met, so $\mc E$ is valid vacuously. Next, we assume that the current $\mc E$ is valid, and Line~\ref{line:update_mce} enlarges $\mc E$ to $\mc E^\prime = \mc E\cup \mc N$. Then for all $(\bm Y^0, \bm y^0)\in \mc F$, if we take any $i\in \mc E$, it holds that $R_i(\bm Y^0, \bm y^0) = b_i$ since $\mc E$ is valid by inductive assumption. If we take any $i\in \mc N$, then it implies $\sigma^\star_{\mc E, i} = b_i$. By the definition of \ref{eqn:sub_lp_feasibility}, it follows that $R_i\Yy \leq b_i$ for all $\Yy\in \mcge$. Hence, \Cref{prop:sandwich}$(i)$ implies that $R_i\Yy \leq b_i$ for all $\Yy\in \mc F$. Definition \eqref{eqn:mcf} of $\mc F_i$ implies that $R_i\Yy \geq b_i$ for all $\Yy\in \mc F$. Therefore, we have $R_i\Yy = b_i$ for all $\Yy\in \mc F$, which also implies $R_i(\bm Y^0, \bm y^0) = b_i$. Thus, $R_i(\bm Y^0, \bm y^0) = b_i$ holds for all $(\bm Y^0, \bm y^0)\in \mc F$ and all $i\in \mc E\cup \mc N$. Therefore, the updated $\mc E^\prime$ is also valid.
			
			Then, we consider the new tight set $\mc N$. We note that, at Line~\ref{line:find_new_tight}, we are provided that, $|\mc E| < m$, and the linear program \ref{eqn:lp_feasibility} is feasible with $t^\star_{\mc E} = 0$. We assume to the contrary that $\mc N = \{i\notin \mc E ~:~ \sigma^\star_{\mc E, i} = b_i\} = \emptyset$. Since $|\mc E| < m$, the complementary to $\mc E$ is non-empty, and for all those $i\notin \mc E$, it holds that $\sigma^\star_{\mc E, i} \neq b_i$. Since \ref{eqn:lp_feasibility} is feasible, we have $\mcge \neq \emptyset$, and the definition of $\mcge$ implies that $\sigma^\star_{\mc E, i} \geq b_i$ for all $i\notin \mc E$. Therefore, it follows $\sigma^\star_{\mc E, i} > b_i$ for all $i\notin \mc E$. Now, for any fixed $i\notin \mc E$, we pick $\Yyi\in \mcge$ with that $R_i\Yyi > b_i$. We define $\hatYy$ as average of $\Yyi$ outside $\mc E$ 
			\begin{equation*}
				\hatYy = \displaystyle\frac{1}{m-|\mc E|} \sum_{i\notin \mc E}\Yyi,
			\end{equation*}
			which is well-defined since $|\mc E| < m$. Since $\mcge$ is a convex and closed polyhedron, it follows that $\hatYy \in \mcge$. Noting that for any $j\notin \mc E$, $R_j$ is an affine function in $\Yy$, it follows that
			\begin{equation*}
				R_j\hatYy = \displaystyle\frac{1}{m-|\mc E|} \sum_{i\notin \mc E}R_j\Yyi = \frac{1}{m-|\mc E|}\left(R_j\Yyj + \sum_{i\notin \mc E, \; i\neq j}R_j\Yyi\right),
			\end{equation*}
			Now, we fix $j\notin \mc E$. By design, we have $R_j\Yyj > b_j$. For each $i\notin \mc E, i\neq j$, we have $R_j\Yyi \geq b_j$ due to $\Yyi\in \mcge$ and the definition of $\mcge$. Therefore, $R_j\hatYy > b_j$. Next, we take $\hat t = \min_{j\in [m]\backslash\mc E}\{R_j\hatYy - b_j\} > 0$, then $(\hat{\bm Y}, \hat{\bm y}, \hat t)$ is a feasible solution to \ref{eqn:lp_feasibility}. Noting that $\hat t>0$, it follows that the optimal value of \ref{eqn:lp_feasibility} is strictly positive, which contradicts the fact that \ref{eqn:lp_feasibility} is feasible with $t^\star_{\mc E} = 0$ when we come to Line~\ref{line:find_new_tight}. Therefore, the new tight set $\mc N$ is indeed non-empty.

			With the above result, we verify the properties of \Cref{alg:check_llf}. Each iteration of the while loop either terminates at Lines~\ref{line:return_empty} or~\ref{line:return_non_empty}, or enlarges $\mc E$. Since $\mc N$ is always non-empty, the index set $\mc E$ grows strictly monotonically. Thus, the while loop terminates in at most $m$ iterations. In each iteration, the algorithm solves \ref{eqn:lp_feasibility} once and \ref{eqn:sub_lp_feasibility} at most $m - |\mc E|$ times, giving a total of at most $(m+2)(m+1)/2$ linear programs. Among all of these, the largest instance is \ref{eqn:lp_feasibility} with $\mc E = [m]$ (solved at Line~\ref{line:LPT}), which has $rd+r+1$ decision variables and $mL+m+1$ constraints.
			
			To verify correctness, we note that \Cref{alg:check_llf} halts at one of Lines~\ref{line:return_empty}, \ref{line:return_non_empty}, \ref{line:return_cannot_find}, and~\ref{line:return_can_find}. When \ref{eqn:lp_feasibility} is infeasible, we have $\mcge = \emptyset$ and hence $\mc F = \emptyset$ by \Cref{prop:sandwich}(i). Then \Cref{lem:llf_technical_condition} implies that \eqref{eqn:linear_sys} is not locally linearly feasible at $\bzz$. If \ref{eqn:lp_feasibility} is feasible with $t^\star_{\mc E} > 0$, then $\mc F \neq \emptyset$ by \Cref{prop:sandwich}(ii), and \eqref{eqn:linear_sys} is locally linearly feasible at $\bzz$ by \Cref{lem:llf_technical_condition}. Now assume that \Cref{alg:check_llf} does not terminate at Line~\ref{line:return_empty} or~\ref{line:return_non_empty}. Then the while loop sequentially enlarges $\mc E$ and shrinks $\mcge$, until we reach Line~\ref{line:LPT} where \ref{eqn:lp_feasibility} is solved with $\mc E = [m]$. If \ref{eqn:lp_feasibility} is infeasible, the same argument as above shows that \eqref{eqn:linear_sys} is not locally linearly feasible at $\bzz$. If \ref{eqn:lp_feasibility} is feasible, then $\mcge \neq \emptyset$. Since $\mc E = [m]$, there is no index $i \notin \mc E$, so the condition $R_i(\bm Y, \bm y) > b_i$ for every $i \notin \mc E$ in \Cref{prop:sandwich}(ii) is vacuous. Hence every $(\bm Y, \bm y) \in \mcge$ also belongs to $\mc F$, which together with \Cref{prop:sandwich}(i) gives $\mc F = \mcge \neq \emptyset$. Local linear feasibility of \eqref{eqn:linear_sys} at $\bzz$ then follows from \Cref{lem:llf_technical_condition}. 
		\end{proof}

	\end{appendix}
	
\end{document}